\documentclass[10pt]{amsart}

\usepackage{amsfonts}
\usepackage{mathrsfs}
\usepackage{amscd}
\usepackage{amsmath}
\usepackage{amssymb}
\usepackage{latexsym}
\usepackage{lscape}
\usepackage{xypic}
\usepackage{comment}
\usepackage{amscd}
\usepackage{wasysym}  
\usepackage{stmaryrd}  
\usepackage{tikz} 
\usetikzlibrary{matrix,arrows}
\usepackage{appendix}
\usepackage{geometry}
\usepackage[nice]{nicefrac}
\usepackage{dsfont}

\usepackage[pagebackref,hyperindex,linktocpage=true]{hyperref}
\hypersetup{
    colorlinks,
    linkcolor={red!50!black},
    citecolor={blue!50!black},
    urlcolor={blue!80!black}
}

\usepackage{color}

\newcommand{\cb}{\color{blue}}
\newcommand{\cm}{\color{magenta}}
\newcommand{\cg}{\color{green}}

\newcommand{\cy}{\color{yellow}}

\newtheorem{thm}{Theorem}[section]

\newtheorem{prop}[thm]{Proposition}
\newtheorem{lem}[thm]{Lemma}

\newtheorem{lem-def}[thm]{Lemma-Definition}
\newtheorem{cor}[thm]{Corollary}

\theoremstyle{definition}

\newtheorem{ex}[thm]{Example}

\newtheorem{rmk}[thm]{Remark}
\newtheorem{dfn}[thm]{Definition}

\numberwithin{equation}{section}

\newcommand{\nc}{\newcommand}

\nc{\on}{\operatorname}
\nc{\fraka}{{\mathfrak a}} \nc{\bba}{{\mathbf a}}
\nc{\frakb}{{\mathfrak b}}
\nc{\frakc}{{\mathfrak c}}
\nc{\frakd}{{\mathfrak d}}
\nc{\frake}{{\mathfrak e}}
\nc{\frakf}{{\mathfrak f}}
\nc{\frakg}{{\mathfrak g}}
\nc{\frakh}{{\mathfrak h}}
\nc{\fraki}{{\mathfrak i}}
\nc{\frakj}{{\mathfrak j}}
\nc{\frakk}{{\mathfrak k}}
\nc{\frakl}{{\mathfrak l}}
\nc{\frakm}{{\mathfrak m}}
\nc{\frakn}{{\mathfrak n}}
\nc{\frako}{{\mathfrak o}}
\nc{\frakp}{{\mathfrak p}}
\nc{\frakq}{{\mathfrak q}}
\nc{\frakr}{{\mathfrak r}}
\nc{\fraks}{{\mathfrak s}}
\nc{\frakt}{{\mathfrak t}}
\nc{\fraku}{{\mathfrak u}}
\nc{\frakv}{{\mathfrak v}}
\nc{\frakw}{{\mathfrak w}}
\nc{\frakx}{{\mathfrak x}}
\nc{\fraky}{{\mathfrak y}}
\nc{\frakz}{{\mathfrak z}}
\nc{\frakA}{{\mathfrak A}}
\nc{\frakB}{{\mathfrak B}}
\nc{\frakC}{{\mathfrak C}}
\nc{\frakD}{{\mathfrak D}}
\nc{\frakE}{{\mathfrak E}}
\nc{\frakF}{{\mathfrak F}}
\nc{\frakG}{{\mathfrak G}}
\nc{\frakH}{{\mathfrak H}}
\nc{\frakI}{{\mathfrak I}}
\nc{\frakJ}{{\mathfrak J}}
\nc{\frakK}{{\mathfrak K}}
\nc{\frakL}{{\mathfrak L}}
\nc{\frakM}{{\mathfrak M}}
\nc{\frakN}{{\mathfrak N}}
\nc{\frakO}{{\mathfrak O}}
\nc{\frakP}{{\mathfrak P}}
\nc{\frakQ}{{\mathfrak Q}}
\nc{\frakR}{{\mathfrak R}}
\nc{\frakS}{{\mathfrak S}}
\nc{\frakT}{{\mathfrak T}}
\nc{\frakU}{{\mathfrak U}}
\nc{\frakV}{{\mathfrak V}}
\nc{\frakW}{{\mathfrak W}}
\nc{\frakX}{{\mathfrak X}}
\nc{\frakY}{{\mathfrak Y}}
\nc{\frakZ}{{\mathfrak Z}}
\nc{\bbA}{{\mathbb A}}
\nc{\bbB}{{\mathbb B}}
\nc{\bbC}{{\mathbb C}}
\nc{\bbD}{{\mathbb D}}
\nc{\bbE}{{\mathbb E}}
\nc{\bbF}{{\mathbb F}} \nc{\bbf}{{\mathbf f}}
\nc{\bbG}{{\mathbb G}}
\nc{\bbH}{{\mathbb H}}
\nc{\bbI}{{\mathbb I}}
\nc{\bbJ}{{\mathbb J}}
\nc{\bbK}{{\mathbb K}}
\nc{\bbL}{{\mathbb L}}
\nc{\bbM}{{\mathbb M}}
\nc{\bbN}{{\mathbb N}}
\nc{\bbO}{{\mathbb O}}
\nc{\bbP}{{\mathbb P}}
\nc{\bbQ}{{\mathbb Q}}
\nc{\bbR}{{\mathbb R}}
\nc{\bbS}{{\mathbb S}}
\nc{\bbT}{{\mathbb T}}
\nc{\bbU}{{\mathbb U}}
\nc{\bbV}{{\mathbb V}}
\nc{\bbW}{{\mathbb W}}
\nc{\bbX}{{\mathbb X}}
\nc{\bbY}{{\mathbb Y}}
\nc{\bbZ}{{\mathbb Z}}
\nc{\calA}{{\mathcal A}}
\nc{\calB}{{\mathcal B}}
\nc{\calC}{{\mathcal C}}
\nc{\calD}{{\mathcal D}}
\nc{\calE}{{\mathcal E}}
\nc{\calF}{{\mathcal F}}
\nc{\calG}{{\mathcal G}}
\nc{\calH}{{\mathcal H}}
\nc{\calI}{{\mathcal I}}
\nc{\calJ}{{\mathcal J}}
\nc{\calK}{{\mathcal K}}
\nc{\calL}{{\mathcal L}}
\nc{\calM}{{\mathcal M}}
\nc{\calN}{{\mathcal N}}
\nc{\calO}{{\mathcal O}}
\nc{\calP}{{\mathcal P}}
\nc{\calQ}{{\mathcal Q}}
\nc{\calR}{{\mathcal R}}
\nc{\calS}{{\mathcal S}}
\nc{\calT}{{\mathcal T}}
\nc{\calU}{{\mathcal U}}
\nc{\calV}{{\mathcal V}}
\nc{\calW}{{\mathcal W}}
\nc{\calX}{{\mathcal X}}
\nc{\calY}{{\mathcal Y}}
\nc{\calZ}{{\mathcal Z}}

\nc{\scrA}{{\mathscr A}}
\nc{\scrB}{{\mathscr B}}
\nc{\scrR}{{\mathscr R}}

\nc{\bnu}{{\bar{ \nu}}}

\nc{\olO}{\bar{\calO}}

\nc{\al}{{\alpha}} 
\nc{\be}{{\beta}}
\nc{\ga}{{\gamma}} \nc{\Ga}{{\Gamma}}
 \nc{\hGa}{\hat{\Gamma}}
\nc{\ve}{{\varepsilon}} 
\nc{\la}{{\lambda}} \nc{\La}{{\Lambda}}
\nc{\om}{\omega} \nc{\Om}{\Omega} 
\nc{\sig}{{\sigma}} \nc{\Sig}{{\Sigma}}

\nc{\tnb}{\psi_{\rm tame}}
\nc{\oM}{\overline{{M}}}
\nc{\op}{{\on{op}}}
\nc{\ad}{{\on{ad}}}
\nc{\alg}{{\on{alg}}}
\nc{\Ad}{{\on{Ad}}}
\nc{\Adm}{{\on{Adm}}} \nc{\aff}{{\on{af}}}
\nc{\Aut}{{\on{Aut}}}
\nc{\Bun}{{\on{Bun}}}
\nc{\cha}{{\on{char}}}
\nc{\der}{{\on{der}}}
\nc{\Der}{{\on{Der}}}
\nc{\diag}{{\on{diag}}}
\nc{\End}{{\on{End}}}
\nc{\Fl}{{\calF\!\ell}}
\nc{\Tr}{{\on{Transp}}}
\nc{\TR}{{\calT\!\calR}}
\nc{\Gal}{{\on{Gal}}}
\nc{\Gr}{{\on{Gr}}}
\nc{\rH}{{\on{H}}}
\nc{\Hom}{{\on{Hom}}}
\nc{\IC}{{\on{IC}}}
\nc{\id}{{\on{id}}}
\nc{\Id}{{\on{Id}}}
\nc{\ind}{{\on{ind}}}
\nc{\Ind}{{\on{Ind}}}
\nc{\Lie}{{\on{Lie}}}
\nc{\Pic}{{\on{Pic}}}
\nc{\pr}{{\on{pr}}}
\nc{\Res}{{\on{Res}}}
\nc{\res}{{\on{res}}} \nc{\Sat}{{\on{Sat}}}
\nc{\s}{{\on{sc}}}
\nc{\drv}{{\on{der}}}
\nc{\sgn}{{\on{sgn}}}
\nc{\Spec}{{\on{Spec}}}\nc{\Spf}{\on{Spf}} 
\nc{\Sph}{\on{Sph}}
\nc{\St}{{\on{St}}}
\nc{\tr}{{\on{tr}}}
\nc{\Mod}{{\mathrm{-Mod}}}
\nc{\Hilb}{{\on{Hilb}}} 
\nc{\Ext}{{\on{Ext}}} 
\nc{\vs}{{\on{Vec}}}
\nc{\ev}{{\on{ev}}}
\nc{\nO}{{\breve{\calO}}}
\nc{\tS}{{\tilde{S}}}
\nc{\spe}{{\on{sp}}}
\nc{\loc}{{\on{loc}}}

\nc{\nscrR}{{\mathscr{R}^{\on{nr}}}}

\nc{\GL}{{\on{GL}}}
\nc{\U}{{\on{U}}}
\nc{\Gl}{\on{Gl}} 
\nc{\GSp}{{\on{GSp}}}
\nc{\gl}{{\frakg\frakl}}
\nc{\SL}{{\on{SL}}} 
\nc{\SU}{{\on{SU}}} 
\nc{\SO}{{\on{SO}}}
\nc{\PGL}{{\on{PGL}}}

\nc{\Conv}{{\on{Conv}}}
\nc{\Rep}{{\on{Rep}}}
\nc{\Dom}{{\on{Dom}}}
\nc{\red}{{\on{red}}}
\nc{\act}{{\on{act}}}
\nc{\nr}{{\on{nr}}}
\nc{\ctf}{{\on{ctf}}}

\nc{\str}{{\on{-}}} 
\nc{\os}{{\bar{s}}}
\nc{\oeta}{{\bar{\eta}}}

\nc{\hookto}{\hookrightarrow}
\nc{\longto}{\longrightarrow}
\nc{\leftto}{\leftarrow}
\nc{\onto}{\twoheadrightarrow}
\nc{\lonto}{\twoheadleftarrow}

\nc{\uG}{{\underline{G}}}
\nc{\uA}{{\underline{A}}}
\nc{\uS}{{\underline{S}}}
\nc{\uT}{{\underline{T}}}
\nc{\uM}{{\underline{M}}}
\nc{\uP}{{\underline{P}}}
\nc{\uB}{{\underline{B}}}
\nc{\uN}{{\underline{N}}}

\nc{\ucG}{{\underline{\calG}}}
\nc{\ucA}{{\underline{\calA}}}
\nc{\ucS}{{\underline{\calS}}}
\nc{\ucT}{{\underline{\calT}}}
\nc{\ucM}{{\underline{\calM}}}
\nc{\ucP}{{\underline{\calP}}}
\nc{\ucN}{{\underline{\calN}}}

\nc{\bF}{{\breve{F}}}

\nc{\oFl}{{\overline{\Fl}}} 
\nc{\bU}{{\overline{U}}}
\nc{\tGr}{{\tilde{\Gr}}}
\nc{\cGr}{\calG\! r}
\nc{\oGr}{\overline{\on{Gr}}} 
\nc{\ocGr}{\overline{\calG\! r}}
\nc{\co}{{\colon}}
\nc{\sch}[1]{(Sch/{#1})}
\nc{\HypLoc}[1]{HypLoc({#1})}

\nc{\ohtimes}{\stackrel{!}{\otimes}}
\nc{\boxtilde}{\widetilde{\boxtimes}}
\nc{\vstar}{{\varhexstar}}

\nc{\Div}{\on{Div}}

\nc{\bslash}{\backslash}
\nc{\algQl}{{\bar{\bbQ}_\ell}}
\nc{\sF}{{\bar{F}}}
\nc{\nF}{{\breve{F}}}
\nc{\nW}{{W^{\on{nr}}}}
\nc{\sk}{{\bar{k}}}
\nc{\cont}{\on{c}}
\nc{\Supp}{\on{Supp}}
\nc{\blt}{\bullet}  
\nc{\dom}{\on{dom}}
\nc{\scon}{{\on{sc}}} 
\nc{\Affine}{\on{Aff}} 
\nc{\nscrA}{\mathscr{A}^{\on{nr}}} 
\nc{\nfraka}{{\bbf^{\on{nr}}}}
\nc{\ran}{{\rangle}}
\nc{\lan}{{\langle}}
\nc{\bk}{{\bar{k}}}
\nc{\tF}{{\tilde{F}}}
\nc{\sS}{{\bar{S}}}
\nc{\LG}{{^\text{L}\hspace{-0.04cm}G}}
\nc{\LL}{{^\text{L}\hspace{-0.07cm}L}}

\nc{\pot}[1]{ [\hspace{-0,5mm}[ {#1} ]\hspace{-0,5mm}] }
\nc{\rpot}[1]{ (\hspace{-0,7mm}( {#1} )\hspace{-0,7mm}) }

\nc{\defined}{\hspace{0.1cm}\stackrel{\text{\tiny \rm def}}{=}\hspace{0.1cm}}

\begin{document}

\title[Test functions for Local Models of Weil-restricted groups]{The test function conjecture for Local Models \\ of Weil-restricted groups}
\author[T.\,J.\,Haines and T.\,Richarz]{by Thomas J. Haines and Timo Richarz}

\address{Department of Mathematics, University of Maryland, College Park, MD 20742-4015, DC, USA}
\email{tjh@math.umd.edu}

\address{Fachbereich Mathematik, TU Darmstadt, Schlossgartenstrasse 7, 64289 Darmstadt, Germany}
\email{richarz@mathematik.uni-darmstadt.de}

%\address{Institut de Math\'ematiques de Jussieu, 4 place Jussieu, 75252 Paris cedex 05, France}
%\email{timo.richarz@uni-due.de}

\thanks{ Research of T.H.~partially supported by NSF DMS-1406787 and by Simons Fellowship 399424, and research of T.R.~ funded by the Deutsche Forschungsgemeinschaft (DFG, German Research Foundation) - 394587809.}

\maketitle

\begin{abstract}
 We prove the test function conjecture of Kottwitz and the first named author for local models of Shimura varieties with parahoric level structure attached to Weil-restricted groups, as defined by B.\,Levin. 
Our result covers the (modified) local models attached to {\it all} connected reductive groups over $p$-adic local fields with $p\geq 5$.
 %{\cm This finishes the proof of the test function conjecture for all known parahoric local models, by handling the remaining cases. }
 In addition, we give a self-contained study of relative affine Grassmannians and loop groups formed using general relative effective Cartier divisors in a relative curve over an arbitrary Noetherian affine scheme.
%The manuscript is supplemented by a study of relative affine Grassmannians and loop groups formed using general Cartier divisors. %{\cm This unifies some favorable cases.} {\cg I think we should say what we mean, or omit}
\end{abstract}

\setcounter{tocdepth}{1}
\tableofcontents
\setcounter{section}{0}

\thispagestyle{empty}

\section{Introduction}

Building upon the work of Pappas and Zhu \cite{PZ13}, B. Levin defines in \cite{Lev16} candidates for parahoric local models of Shimura varieties for reductive groups of the form $\Res_{K/F}(G_0)$ where $G_0$ splits over a tamely ramified extension of $K$, and $K/F$ is a finite (possibly wildly ramified) extension. 
The present manuscript is a follow-up of \cite{HaRi}, in which we prove the test function conjecture for these local models. The method follows closely \cite{HaRi}, and we only explain new arguments in detail, but repeat as much as necessary for readability. 
For a detailed introduction and further references we refer the reader to the introduction of \cite{HaRi}.

Let us mention that the article is supplemented in \S\ref{Recollect_Loop_Group_Sec} by a general study of relative affine Grassmannians and loop groups formed using a general Cartier divisor as in the work of Beilinson and Drinfeld \cite{BD}. 
This unifies the frameworks of \cite{PZ13, Lev16} in mixed characteristic, of \cite{He10, Zhu14, Zhu15, Ri16b} in equal characteristic, and of the work of Fedorov and Panin \cite{FP15, Fe} on the Grothendieck-Serre conjecture, cf.~ Examples \ref{Special_Cases} below. 
As an application, we identify the torus fixed points and their attractor and repeller loci in the sense of Drinfeld \cite{Dr} (cf. also \cite{He80}) for these relative affine Grassmannians, cf.~ Theorem \ref{GCT_Geo_Thm}.  

\subsection{Formulation of the main result} Let $p$ be a prime number. Let $F/\bbQ_p$ be a finite extension with residue field $k_F$ of cardinality $q$. Let $\sF/F$ be a separable closure, and denote by $\Ga_F$ the Galois group with inertia subgroup $I_F$ and fixed geometric Frobenius lift $\Phi_F\in\Ga_F$. 

Let $K/F$ be a finite extension, and let $G_0$ be a (connected) reductive $K$-group which splits over a tamely ramified extension. 
We are interested in the group of Weil restrictions $G=\Res_{K/F}(G_0)$ which is a reductive $F$-group but now possibly wildly ramified depending on $K/F$. 
%When $p\geq 5$ then any adjoint reductive $F$-group is of this form, cf.~\cite[\S1.12; \S 4]{Ti77} (see also \cite[\S7.a]{PR08}). 

Let $\calG$ be a parahoric $\calO_F$-group scheme in the sense of Bruhat-Tits \cite{BT84} with generic fiber $G$. 
Note that $\calG=\Res_{\calO_K/\calO_F}(\calG_0)$ for a unique parahoric $\calO_K$-group scheme $\calG_0$ with generic fiber $G_0$, cf.~ Corollary \ref{Parahoric_Cor}. 
We fix $\{\mu\}$ a (not necessarily minuscule) conjugacy class of geometric cocharacters in $G$ defined over a finite (separable) extension $E/F$. 

Attached to the triple $(G,\{\mu\},\calG)$ is the (flat) \emph{local model}
\[
M_{\{\mu\}}=M_{(G,\{\mu\},\calG)},
\]
which is a flat projective $\calO_E$-scheme, cf.\,\cite{PZ13} if $K=F$ and \cite{Lev16} for general $K/F$ (cf.\,also Definition \ref{Local_Model_Dfn}). 
The generic fiber $M_{\{\mu\}, E}$ is naturally the Schubert variety in the affine Grassmannian of $G/ E$ associated with the class $\{\mu\}$. The special fiber $M_{\{\mu\},k_E}$ is equidimensional, but neither irreducible nor a divisor with normal crossings in general.

Fix a prime number $\ell\not= p$, and fix $q^{-\nicefrac{1}{2}}\in \algQl$ in order to define half Tate twists. 
Let $d_\mu$ be the dimension of the generic fiber $M_{\{\mu\}, E}$, and denote the normalized intersection complex by
\[
\IC_{\{\mu\}}\defined j_{!*}\algQl[d_\mu](\nicefrac{d_\mu}{2})\in D_c^b(M_{\{\mu\}, E},\algQl)
\]  
cf.\,\S \ref{Geo_Sat_Weil_Sec}. 
Under the geometric Satake equivalence \cite{Gi, Lu81, BD, MV07, Ri14a, RZ15, Zhu}, the complex $\IC_{\{\mu\}}$ corresponds to the $^LG_E=G^\vee\rtimes \Ga_{E}$-representation $V_{\{\mu\}}$ of highest weight $\{\mu\}$ defined in \cite[6.1]{Hai14}, cf.~\cite[Cor 3.12]{HaRi}. 
Note that we have $G^\vee=\Ind_{\Ga_K}^{\Ga_F}(G_0^\vee)$ as groups over $\bar{\mathbb Q}_\ell$ under which
%and 
$V_{\{\mu\}}=\boxtimes_\psi V_{\mu_\psi}$ (cf.~ Lemma \ref{Induced_Rep_Lem}).
%, and both are taken over $\bar{\mathbb Q}_\ell$.

Let $E_0/F$ be the maximal unramified subextension of $E/F$, and let $\Phi_E=\Phi_{E_0}=\Phi_F^{[E_0:F]}$ and $q_E = q_{E_0} = q^{[E_0:F]}$. 
The semi-simple trace of Frobenius function on the sheaf of nearby cycles 
\[
\tau^{\on{ss}}_{\{\mu\}}\co M_{\{\mu\}}(k_{E})\to \algQl, \;\;\; x\mapsto (-1)^{d_\mu}\on{tr}^{\on{ss}}(\Phi_{E}\,|\, \Psi_{M_{\{\mu\}}}(\IC_{\{\mu\}})_{\bar{x}}),
\]
is naturally a function in the center $\calZ(G({E_0}),\calG({\calO_{E_0}}))$ of the parahoric Hecke algebra, cf.~ \cite[Thm.~10.14]{PZ13}, \cite[Thm.~5.3.3]{Lev16} and $\S\ref{reduction_to_MT_subsec}$. We remark that $\tau^{\rm ss}_{\{\mu\}}$ lives in the center  of the $\bar{\mathbb Q}_\ell$-valued Hecke algebra attached to {\em function field} analogues of $(G_{E_0}, \calG_{\mathcal O_{E_0}}, E_0)$; we are implicitly identifying this with $\calZ(G(E_0), \calG(\calO_{E_0}))$ via Lemma \ref{hecke_identification}.

Our main result, {\em the test function conjecture for local models for Weil restricted groups}, characterizes the function $\tau^{\on{ss}}_{\{\mu\}}$, extending the main result of \cite{HaRi} to the Weil-restricted situation. It confirms that even for these local models, the local geometry of Shimura varieties at places of parahoric bad reduction can be related to automorphic-type data, as required by the Langlands-Kottwitz method.\bigskip

\noindent\textbf{Main Theorem.} %(The test function conjecture for local models of Weil restricted groups).{\cg I suggest we remove the parenthetical description, and put it as indicated above}
{\it  Let $(G,\{\mu\},\calG)$ be a triple as above. Let $E/F$ be a finite separable extension over which $\{\mu\}$ is defined, and let $E_0/F$ be the maximal unramified subextension. Then
$$
\tau^{\on{ss}}_{\{\mu\}} =  z^{\rm ss}_{\{\mu\}}
$$
where $z^{\rm ss}_{\{\mu\}} = z^{\rm ss}_{\calG, {\{\mu\}}} \in \calZ(G(E_0),\calG(\calO_{E_0}))$ is the unique function which acts on any $\calG(\calO_{E_0})$-spherical smooth irreducible $\bar{\mathbb Q}_\ell$-representation $\pi$ by the scalar
\[
 \on{tr}\Bigl(s(\pi)\;\big|\; \on{Ind}_{^L{G}_E}^{^L{G}_{E_0}}(V_{\{\mu\}})^{1\rtimes I_{E_0}}\Bigr),
\]
where $s(\pi)\in [(G^\vee)^{I_{E_0}}\rtimes \Phi_{E_0}]_{\on{ss}}/(G^\vee)^{I_{E_0}}$ is the Satake parameter for $\pi$ \cite{Hai15}. The function $q^{d_\mu/2}_{E_0}\tau^{\on{ss}}_{\{\mu\}}$ takes values in $\mathbb Z$ and is independent of $\ell \neq p$ and $q^{1/2} \in \bar{\mathbb Q}_\ell$.}

\bigskip

The construction of $s(\pi)$ is also reviewed in \cite[\S 7.2]{HaRi}, and the values of $z^{\rm ss}_{\{\mu\}}$ are studied in \cite[\S 7.7]{HaRi}, cf. \S\ref{Values_Sec}.  The definition of the local models $M_{\{\mu\}}$ depends on certain auxiliary choices (cf.\,Remark \ref{Unique_Model_dfn}), but the function $\tau^{\rm ss}_{\{\mu\}}$ depends canonically only on the data $(G,\{\mu\},\calG)$.

As an application of our main theorem we prove in \S\ref{modified_local_mod_sec} the test function conjecture for (modified) local models attached to {\it all} groups and prime numbers $p\geq 5$.
This relies on the fact that when $p\geq 5$ any adjoint reductive $F$-group is isomorphic to a product of Weil restrictions of scalars of tamely ramified groups, cf.~\eqref{weak_assumption} below. In Theorem \ref{Main_Thm_Phi} and $\S\ref{reduction_to_MT_subsec}$ we also show that the variant of the Main Theorem holds, where semisimple traces are replaced by traces with respect to any fixed lift $\Phi_{E}$ of geometric Frobenius.

\subsection{Other results} Our methods can be used to obtain results on the fixed point (resp.~attractor and repeller) locus of $\bbG_m$-actions on Fusion Grassmannians (cf.~Theorem A below), and the special fiber of local models (cf.~Theorem B below).

\subsubsection{Fusion Grassmannians} Let $F$ be any field, and let $G$ be a reductive $F$-group. For each $n\geq 0$, there is the fusion Grassmannian $\Gr_{G,n}\to \bbA^n_F$ defined in \cite{BD} which parametrizes isomorphism classes of $G$-bundles on the affine line together with a trivialization away from $n$ points. Given a cocharacter $\chi\co \bbG_{m,F}\to G$ we obtain a fiberwise $\bbG_m$-action on the family $\Gr_{G,n}\to \bbA^n_F$, and we are interested in determining the diagram on the fixed point ind-scheme and attractor (resp. repeller) ind-scheme 
\[
(\Gr_{G,n})^0\leftarrow (\Gr_{G,n})^\pm\rightarrow \Gr_{G,n},
\]
cf.\,\eqref{flow}. Let $M\subset G$ be the centralizer of $\chi$, which is a Levi subgroup. The dynamic method promulgated in \cite{CGP10} defines a pair of parabolic subgroups $(P^+,P^-)$ in $G$ such that $P^+\cap P^-=M$; see the formulation of Theorem \ref{GCT_Geo_Thm}. The natural maps $M\leftarrow P^\pm \to G$ induce maps of fusion Grassmannians
\[
\Gr_{M,n}\leftarrow \Gr_{P^\pm,n}\to \Gr_{G,n}.
\]
An extension of the method used in the proof of \cite[Prop.~3.4]{HaRi} allows us to prove the following result.

\bigskip

\noindent\textbf{Theorem A.} 
{\it For each $n\in\bbZ_{\geq 0}$, there is a commutative diagram of $\bbA^n_F$-ind-schemes
\[
\begin{tikzpicture}[baseline=(current  bounding  box.center)]
\matrix(a)[matrix of math nodes, 
row sep=1.5em, column sep=2em, 
text height=1.5ex, text depth=0.45ex] 
{\Gr_{M,n} & \Gr_{P^\pm,n} & \Gr_{G,n} \\ 
(\Gr_{G,n})^0& (\Gr_{G,n})^\pm& \Gr_{G,n}, \\}; 
\path[->](a-1-2) edge node[above] {}  (a-1-1);
\path[->](a-1-2) edge node[above] {}  (a-1-3);
\path[->](a-2-2) edge node[below] {}  (a-2-1);
\path[->](a-2-2) edge node[below] {} (a-2-3);
\path[->](a-1-1) edge node[left] {$\simeq$} (a-2-1);
\path[->](a-1-2) edge node[left] {$\simeq$} (a-2-2);
\path[->](a-1-3) edge node[left] {$\id$} (a-2-3);
\end{tikzpicture}
\]
where the vertical maps are isomorphisms.}

\bigskip

 Theorem A is a special case of Theorem \ref{GCT_Geo_Thm} which applies to general reductive group schemes over $\bbA^n_F$ which are not necessarily defined over $F$. Let us point out that \cite[Prop.~3.4]{HaRi} implies that Theorem A holds fiberwise. However, we do not know how to prove sufficiently good flatness properties of $\Gr_{G,n}\to \bbA^n_F$ in order to deduce the more general result from the fiberwise result. % in this way. 
 
The tensor structure on the constant term functors in geometric Langlands is constructed in \cite{BD, MV07}. In \cite{Ga07, Re12}, it is explained how to use the nearby cycles to define the fusion structure used in the geometric Satake isomorphism. Theorem A together with \cite[Thm.~3.3]{Ri19} gives another way of constructing the tensor structure on the constant term functors - even without passing to the underlying reduced ind-schemes, cf.~proof of \cite[Thm.~3.16]{HaRi}.

\subsubsection{Special fibers of local models} As in \cite[\S 6.3.1]{HaRi}, we use the commutation of nearby cycles with constant terms to determine the irreducible components of the geometric special fiber $M_{\{\mu\},\bar{k}}$ of the local models. Recall that by construction (cf.~Definition \ref{Local_Model_Dfn}), there is a closed embedding
\[
M_{\{\mu\},\bar k}\,\hookto\, \Fl_{\calG^\flat,\bar k},
\]
where $\Fl_{\calG^\flat}$ is the (partial) affine flag variety attached to the function field analogue $\calG^\flat/k_F\pot{u}$ of $\calG/\calO_F$, cf.~Theorem \ref{parahoric_group_thm} and Proposition \ref{Fiber_BDGrass} ii). As envisioned by Kottwitz and Rapoport, the geometric special fiber $M_{\{\mu\},\bar{k}}$ should be the union of the Schubert varieties ${\mathcal Fl}_{\calG^\flat,\bar{k}}^{\leq w}\subset \Fl_{\calG^\flat,\bar{k}}$ where $w$ ranges over the $\{\mu\}$-admissible set $\Adm_{\{\mu\}}^{\bf f}\subset W_\bbf\bslash W/W_\bbf$ where $\calG=\calG_\bbf$ and $W=W(G, F)$ denotes the Iwahori-Weyl group. Here we are identifying the Iwahori-Weyl groups attached to $G/F$ and $G^\flat/k_F\rpot{u}$ by Lemma \ref{Iwahorilem}. The following result verifies their prediction (cf.~Theorem \ref{Weil_specialfiber}).

\bigskip
\noindent\textbf{Theorem B.} {\it The smooth locus $(M_{\{\mu\}})^{\on{sm}}$ is fiberwise dense in $M_{\{\mu\}}$, and on reduced subschemes a union of the Schubert varieties
\[
(M_{\{\mu\},\bar{k}})_\red\,=\,\bigcup_{w\in\Adm_{\{\mu\}}^\bbf}{\mathcal Fl}_{\calG^\flat,\bar{k}}^{\leq w}.
\]
In particular, the geometric special fiber $M_{\{\mu\},\bar{k}}$ is generically reduced.}
\bigskip

If $p\nmid |\pi_1(G_\der)|$, then Theorem B is \cite[Thm.~9.3]{PZ13} for $K=F$, and \cite[Thm.~2.3.5]{Lev16} when $K \neq F$. We have removed this condition on $p$ and thereby conclude that the Kottwitz-Rapoport strata in the special fiber are enumerated by the $\{\mu\}$-admissible set for all local models constructed in \cite{PZ13, Lev16}.

%{\cg I will add Theorem B on the special fibers and the admissible set.}

\subsection{Overview} In \S \ref{Recollect_Gm_Act_Sec} we recall a few facts about $\bbG_m$-actions for convenience. The following \S\ref{Recollect_Loop_Group_Sec} studies relative affine Grassmannians formed using a general Cartier divisor. In \S \ref{Recollect_Weil_Restricted_Sec}, we recall the definition of Weil-restricted local models and results from \cite{Lev16} which are needed in the sequel. These results are applied in \S \ref{Gm_Act_Sec} to study $\bbG_m$-actions on Beilinson-Drinfeld affine Grassmannians  for Weil-restricted groups. In \S \ref{Test_Functions_Sec}, we formulate and prove the test function conjecture for Weil-restricted local models. 

\subsection{Acknowledgements} The authors thank Michael Rapoport for funding, the University of Maryland for logistical support which made this research possible, and the anonymous referee for useful suggestions leading to the results in \S\ref{modified_local_mod_sec}. 
The second named author thanks the DFG (German Research Foundation) for financial support during the academic year 2018.

\subsection{Conventions on Ind-Algebraic Spaces} Let $\calO$ be a ring, and denote by $\calO\on{-Alg}$ the category of $\calO$-algebras equipped with the fpqc topology. An $\calO$-space $X$ is a sheaf on the site $\calO\on{-Alg}$, and we denote the category of $\calO$-spaces by $\text{Sp}_\calO$. As each object in the site $\calO\on{-Alg}$ is quasi-compact, the pretopology on $\calO\on{-Alg}$ is generated by finite covering families, and hence filtered colimits exist in $\text{Sp}_\calO$ and can be computed in the category of presheaves. 

The category $\on{Sp}_\calO$ contains the category of $\calO$-schemes $\text{Sch}_\calO$ as a full subcategory. An $\calO$-algebraic space is a $\calO$-space $X$ such that $X\to X\times_\calO X$ is relatively representable, and such that there exists an \'etale surjective map from a scheme $U\to X$. By a Theorem of Gabber \cite[Tag 03W8]{StaPro} this agrees with the usual definition of algebraic spaces using \'etale or fppf sheaves.

The category of $\calO$-algebraic spaces is denoted $\text{AlgSp}_\calO$. There are full embeddings $\text{Sch}_\calO\subset \text{AlgSp}_\calO\subset \text{Sp}_\calO$. A map of $\calO$-spaces $X\to Y$ is called representable (resp. schematic) if for every scheme $T\to Y$ the fiber product $X\times_YT$ is representable by an algebraic space (resp. scheme). 

An $\calO$-ind-algebraic space (resp. $\calO$-ind-scheme) is a contravariant functor
\[
X\co  \calO\on{-Alg}\,\to\, \on{Sets} 
\]
such that there exists a presentation as presheaves $X=\text{colim}_iX_i$ where $\{X_i\}_{i\in I}$ is a filtered system of $\calO$-algebraic spaces (resp.\,$\calO$-schemes) $X_i$ with transition maps being (schematic) closed immersions. Since filtered colimits in $\on{Sp}_\calO$ can be computed in presheaves, every $\calO$-ind-algebraic space (resp. $\calO$-ind-scheme) is an $\calO$-space. The category of $\calO$-ind-algebraic spaces (resp. $\calO$-ind-schemes) $\on{IndAlgSp}_\calO$ (resp. $\on{IndSch}_\calO$) is the full subcategory of $\on{Sp}_\calO$ whose objects are $\calO$-ind-algebraic spaces (resp. $\calO$-ind-schemes). If $X=\on{colim}_iX_i$ and $Y=\text{colim}_jY_j$ are presentations of ind-algebraic spaces (resp. ind-schemes), and if each $X_i$ is quasi-compact, then as sets
\[
\Hom_{\on{Sp}_\calO}(X,Y)= \on{lim}_i\on{colim}_j\Hom_{\on{Sp}_\calO}(X_i,Y_j),
\]
because every map $X_i\to Y$ factors over some $Y_j$ by quasi-compactness of $X_i$. The category $\text{IndAlgSp}_\calO$ (resp. $\text{IndSch}_\calO$) is closed under fiber products, i.e.,\,$\on{colim_{(i,j)}}(X_i\times_\calO Y_j)$ is a presentation of $X\times_\calO Y$. If $\calP$ is a property of algebraic spaces (resp. schemes), then an $\calO$-ind-algebraic space (resp. $\calO$-ind-scheme) $X$ is said to have $\text{ind-}\calP$ if there exists a presentation $X=\text{colim}_iX_i$ where each $X_i$ has property $\calP$. A map $f\co X\to Y$ of $\calO$-ind-algebraic spaces (resp. $\calO$-ind-schemes) is said to have property $\calP$ if $f$ is representable and for all schemes $T\to Y$, the pullback $f\times_Y T$ has property $\calP$. Note that every representable quasi-compact map of $\calO$-ind-schemes is schematic. 

 %%%%%%%%%%%%%Recollection on \bbG_m-actions%%%%%%%%%%%%%
 \section{Actions of $\bbG_m$ on Ind-Algebraic Spaces}\label{Recollect_Gm_Act_Sec}
 %\section{Recollection on $\bbG_m$-actions on Ind-Algebraic Spaces}\label{Recollect_Gm_Act_Sec}
 We recall some set-up and notation from \cite{Dr} and \cite{Ri19}. Let $\calO$ be a ring, and let $X$ be an $\calO$-algebraic space (or $\calO$-ind-algebraic space) equipped with an action of $\bbG_m$ which is trivial on $\calO$. There are the following three functors on the category of $\calO$-algebras
\begin{equation}\label{flow}
\begin{aligned}
\hspace{1cm} X^0\co&\; R\longmapsto \Hom^{\bbG_{m}}_R(R, X)\\
\hspace{1cm} X^+\co& \; R\longmapsto \Hom^{\bbG_{m}}_R((\bbA_{R}^1)^+, X)\\
\hspace{1cm}X^-\co& \; R\longmapsto \Hom^{\bbG_{m}}_R((\bbA^1_{R})^-, X),
\end{aligned}
\end{equation}
where $(\bbA_R^1)^+$ (resp. $(\bbA_R^1)^-$) is $\bbA^1_R$ with the usual (resp. opposite) $\bbG_m$-action. The functor $X^0$ is the functor of $\bbG_m$-fixed points in $X$, and $X^+$ (resp. $X^-$) is called the attractor (resp. repeller). Informally speaking, $X^+$ (resp. $X^-$) is the space of points $x$ such that the limit $\lim_{\la\to 0}\la\cdot x$ (resp. $\lim_{\la\to \infty}\la\cdot x$) exists. The functors \eqref{flow} come equipped with natural maps
\begin{equation}\label{hyper_loc_maps}
X^0\leftarrow X^\pm \to X,
\end{equation}
where $X^\pm\to X^0$ (resp. $X^\pm\to X$) is given by evaluating a morphism at the zero section (resp. at the unit section). We say that the $\bbG_m$-action on an algebraic space $X$ is \'etale (resp. Zariski) locally linearizable
%, i.e.
if the $\bbG_m$-action lifts - necessarily uniquely - to an \'etale cover which is affine, cf.~\cite[Def.~1.6]{Ri19}. 
We say that an $\bbG_m$-action on an $S$-ind-algebraic space $X$ is \'etale (resp. Zariski) locally linearizable if there is an $\bbG_m$-stable presentation with equivariant transition maps $X=\on{colim}_iX_i$ where the $\bbG_m$-action on each $X_i$ is \'etale (resp. Zariski) locally linearizable. We use the following representability properties of the functors \eqref{flow}, cf. \cite[Thm.~2.1]{HaRi}. 

\begin{thm}\label{Gm_thm} Let $X=\on{colim}_iX_i$ be an $\calO$-ind-algebraic space equipped with an \'etale locally linearizable $\bbG_m$-action. \smallskip\\
i\textup{)} The subfunctor $X^0=\on{colim}_iX_i^0$ is representable by a closed sub-ind-algebraic space of $X$. \smallskip\\
ii\textup{)} The functor $X^\pm=\on{colim}_iX_i^\pm$ is representable, and the map $X^\pm \to X$ is representable and quasi-compact. The map $X^\pm\to X^0$ is ind-affine with geometrically connected fibers and induces a bijection on connected components $\pi_0(X^\pm)\simeq \pi_0(X^0)$ of the underlying topological spaces.\smallskip\\
iii\textup{)} If $X=\on{colim}_iX_i$ is of ind-finite presentation \textup{(}resp.\, an ind-scheme; \,resp.\,separated\textup{)}, so are $X^0$ and $X^\pm$.
\end{thm}  
 The proof is like that of \cite[Thm.\,2.1]{HaRi}, using the representability results of \cite[Thm.\,1.8]{Ri19}. We record the following lemma for later use.
 
 \begin{lem}\label{Gm_product_lem}
 For $n\in \bbZ_{>0}$, let $X_1,\ldots, X_n$ be $\calO$-algebraic spaces \textup{(}or $\calO$-ind-algebraic spaces\textup{)} equipped with an \'etale locally linearizable $\bbG_m$-action. Then the diagonal $\bbG_m$-action on the product $\prod_{i=1}^nX_i$ is \'etale locally linearizable, and the canonical maps
 \begin{align*}
 (\prod_{i=1}^nX_i)^0 \overset{\simeq}{\longto}  \prod_{i=1}^nX_i^0\;\;\;\;\text{and}\;\;\;\; (\prod_{i=1}^nX_i)^\pm \overset{\simeq}{\longto}  \prod_{i=1}^nX_i^\pm
 \end{align*}
are isomorphisms.
 \end{lem}
 \begin{proof} 
 If, for each $i$, the map $U_i\to X_i$ is an \'etale local linearization, then the product $\prod_{i=1}^nU_i\to \prod_{i=1}^nX_i$ is an \'etale local linearization. It is easy to check on the level of functors that the maps are isomorphisms. 
 \end{proof}

%%%%%%%%%%%%%%%%%%Recollection on loop groups and affine Grassmannians%%%%%%%%%%%%%%%%%%%%%%
\section{Affine Grassmannians for Cartier divisors}\label{Recollect_Loop_Group_Sec}
%\section{Loop groups and affine Grassmannians for Cartier divisors}\label{Recollect_Loop_Group_Sec}

In this section, we give a self-contained treatment of affine Grassmannians for non-constant group schemes over relative curves which are formed using a formal neighborhood of a general Cartier divisor. This extends the work of Beilinson-Drinfeld \cite{BD}, %{\cm don't they consider only the case of constant groups and curves over fields?  In that case, maybe we should say this `extends' rather than `follows' their work}
and is inspired by the work of Fedorov-Panin \cite{FP15, Fe} and Levin \cite{Lev16}.

\subsection{Definitions and Examples}
Let $\calO$ be a Noetherian ring. Let $X$ be a smooth $\calO$-curve, i.e., the structure map $X\to \Spec(\calO)$ is of finite presentation and smooth of pure dimension $1$. Let $D\subset X$ be a relative effective Cartier divisor which is finite and flat over $\calO$. Let $\calG$ be a smooth affine $X$-group scheme. 

To the triple $(X,\calG,D)$, we associate the functor $\Gr_\calG=\Gr_{(X,\calG,D)}$ on the category of $\calO$-algebras which assigns to every $R$ the set of isomorphism classes of tuples $(\calF,\al)$ with 
\begin{align}\label{dfnBD}
\begin{cases}
& \calF \; \text{a} \; \calG\text{-torsor on} \; X_R; \\
& \al: \calF|_{(X\bslash D)_R}\stackrel{\simeq}{\longto} \calF_0|_{(X\bslash D)_R} \;\text{a trivialization}, 
\end{cases}
\end{align}
where $\calF_0$ denotes the trivial $\calG$-torsor. Fpqc-descent for schemes affine over $X_R$ implies that $\Gr_{\calG}$ is an $\calO$-space. As $\calG$ is smooth affine and hence of finite presentation, the functor $\Gr_{\calG}$ commutes with filtered colimits of $\calO$-algebras. Further, if $R$ is a $\calO$-algebra, then as functors on $R\text{-Alg}$,
\begin{equation}\label{Base_Change}
\Gr_{\calG}\times_{\Spec(\calO)}\Spec(R)=\Gr_{\calG}|_{R\text{-Alg}}=\Gr_{(X_R, \calG_{X_R},D_R)}.
\end{equation}
If we replace $D$ by a positive multiple $nD$ for some $n\geq 1$, then $X\bslash D=X\bslash nD$, and hence as $\calO$-functors
\begin{equation}\label{Multiple_Divisor}
\Gr_{(X,\calG,D)}=\Gr_{(X,\calG,nD)}.
\end{equation}
The following examples are special cases of the general set-up. 

\begin{ex}\label{Special_Cases} 
i) {\it Affine Grassmannians/Flag Varieties.} Let $\calO=F$ be a field, and let $D=\{x\}$ for some point $x\in X(F)$. Then on completed local rings $\calO_{x}\simeq F\pot{t_x}$ where $t_x$ denotes a local parameter at $x\in X$. If $\calG=G\otimes_FX$ for some smooth affine $F$-group $G$, then $\Gr_G:=\Gr_\calG$ is (by the Beauville-Laszlo theorem \cite{BL95}) the ``affine Grassmannian'' formed using the local parameter $t_x$, i.e., the ind-scheme given by the \'etale sheafification of the functor $R\mapsto G(R\rpot{t_x})/G(R\pot{t_x})$. In general, the functor $\Gr_\calG$ is the ``twisted affine flag variety'' for the group scheme $\calG\otimes_XF\pot{t_x}$ in the sense of \cite{PR08}. \smallskip\\
ii) {\it Mixed characteristic.} Let $\calO=\calO_F$ be the valuation ring of a finite extension $F/\bbQ_p$. Let $K/F$ be a finite totally ramified extension with uniformizer $\varpi\in K$. Let $X=\bbA^1_{\calO_F}$ with global coordinate denoted $z$, and let $D=\{Q=0\}$, where $Q\in \calO_F[z]$ is the minimal polynomial of $\varpi$ over $F$ (an Eisenstein polynomial). Let $\calG$ be the $X$-group scheme constructed in \cite[Thm.~4.1]{PZ13} if $K=F$, and in \cite[Thm.~3.3.3]{Lev16} otherwise; here it is denoted $\underline{\mathcal G}$, see Theorem \ref{parahoric_group_thm}. Then $\Gr_\calG$ is the $\calO_F$-ind-scheme defined in \cite[Eq (6.11)]{PZ13} if $K=F$, and in \cite[Def 4.1.1]{Lev16} otherwise; here we denote it $\Gr_{\tilde{\calG}}$, see $\S \ref{affGrass_WR_def_subsec}$. \smallskip\\
iii) {\it Equal characteristic.} Let $F$ be a field, and let $C$ be a smooth affine $F$-curve. Let $\calO=\Ga(C,\calO_C)$ be the global sections, and let $X=C\times_FC=C_\calO$. Let $\calG_0$ be a smooth affine $\calO$-group scheme, and let $\calG=\calG_0\otimes_\calO X$. Let $D:=\Delta(C)$ be the diagonal divisor in $X$. If $C=\bbA^1_F$, then $\Gr_\calG$ is the ind-scheme defined in \cite[Eq (3.1.1)]{Zhu14}. If $x\in C(F)$ is a point, and $\calO_x\to \%calO$ denotes the completed local ring, then $\Gr_\calG\otimes_\calO\calO_x$ is the ind-scheme defined in \cite[Def 2.3]{Ri16b}. Let us remark that this is a special case of the general set-up in \cite[\S 2]{He10}.\smallskip\\
iv) {\it Fusion Grassmannians.} Let $F$ be a field, and let $C$ be an affine curve over $F$. The $d$-th symmetric product $C^{(d)}$ is by [SGA IV, Exp. XVII, Prop. 6.3.9] the moduli space of degree $d$ effective Cartier divisors on $C$. Let $\Spec(\calO):=C^{(d)}$, and we let $D:=C^{(d)}$ be the universal degree $d$ divisor on $X:=C\times_FC^{(d)}=C_\calO$. For a smooth affine $F$-group scheme $G$, we let $\calG=G\otimes_FX$. Then the ind-scheme $\Gr_\calG\times_{\Spec(\calO)} C^d$ is the fusion Grassmannian defined in \cite[5.3.11]{BD}.\smallskip\\
v) {\it Generically trivial bundles.} If $X=\bbA^1_\calO$ and $\calG$ is split reductive, then the functor $\Gr_\calG$ in \eqref{dfnBD} is the moduli space of objects used in \cite[Thm.~2]{Fe}.
\end{ex}

%Let us denote by $\Div_X$ the space of \underline{relative} effective Cartier divisors on $X$ which are finite over $R$. As $X$ is quasi-projective, the sheaf $\Div_X$ is representable by the disjoint union $\coprod_{d\geq 1}X^{(d)}$ of the symmetric powers of $X$, cf. [SGA IV, Exp. XVII, Prop. 6.3.9]. We consider $\Gr_G$ as given over $\Div_X$ via the forgetful map $\Gr_G\to\Div_X$, $(D,\calF,\be)\mapsto D$.

\subsubsection{Loop Groups} \label{Loop_Group_Sec}
The functor $\Gr_\calG$ is related to loop groups as follows. For an $\calO$-algebra $R$, let $(X_R/D_R)\,\, \widehat{}\,$ be the formal affine\footnote{One can show that a formal completion $(X/X')\,\,\widehat{}$ of a scheme $X$ along an {\em affine} closed subscheme $X' \subset X$ is of the form ${\rm Spf}(A)$ for an admissible topological ring $A$.} %{\cm This is implicit in \cite[2.12.2]{BD}.}
%{\cy This is 7.11.1 in BD. I think we need to reformulate the red colored text as this is definitely not what BD are claiming and we should not write that they do. I am totally fine with the rest of the footnote.}
scheme defined by $D_R$ in $X_R$, and denote by $R\pot{D}$ its ring of regular functions. Explicitly, if $\mathcal I_R \subset \mathcal O_{X_R}$ is the ideal sheaf for $D_R$, then $(D_R, \mathcal O_{X_R}/\mathcal I_R^n)$ is an affine scheme ${\rm Spec}(A_n)$ for all $n \geq 1$, and $R\pot{D} \overset{\rm def}{=} \varprojlim A_n = \varprojlim \Gamma(D_R, \mathcal O_{X_R}/\mathcal I_R^n)$. Let $\hat{D}_R=\Spec(R\pot{D})$ be the associated affine (true) scheme. The map $(X_R/D_R)\,\, \widehat{}\to X_R$ uniquely extends to a map $p\co \hat{D}_R\to X_R$ by \cite[Thm.~1.1]{Bha16}\footnote{When $X_R$ is quasi-projective, one can invoke the more elementary result of \cite[2.12.6]{BD}.}, and $p^{-1}(D_R)\simeq D_R$ defines a relative effective Cartier divisor on $\hat{D}_R$. 
Let $\hat{D}_R^o=\hat{D}_R\backslash D_R$. As $D_R$ is a Cartier divisor in $\hat{D}_R$, it is locally principal, and hence the complement $\hat{D}^o_R:=\Spec(R\rpot{D})$ is an affine scheme. The \emph{\textup{(}twisted\textup{)} loop group} $L\calG=L_D\calG$ is the functor on the category of $\calO$-algebras
\begin{equation}\label{Loop_Group}
L\calG\co R\mapsto \calG(R\rpot{D}).
\end{equation}
The \emph{positive \textup{(}twisted\textup{)} loop group} $L^+\calG=L_D^+\calG$ is the functor on the category of $\calO$-algebras
\begin{equation}\label{Pos_Loop_Group}
L^+ \calG\co R\mapsto \calG(R\pot{D}).
\end{equation}
As every Cartier divisor is locally defined by a single non-zero divisor, we see that $L^+ \calG\subset L \calG$ is a subgroup functor. Let us explain why these functors are representable in this generality. 

\begin{lem}\label{Loop_Group_Rep_Lem}
i\textup{)} The functor $L^+\calG$ \textup{(}resp. $L\calG$\textup{)} is representable by an affine scheme \textup{(}resp. ind-affine ind-scheme\textup{)}. In particular, $L^+\calG$ and $L\calG$ are $\calO$-spaces. \smallskip\\
ii\textup{)} The scheme $L^+\calG$ is a faithfully flat affine $\calO$-group scheme which is pro-smooth.
\end{lem}
\begin{proof}
Part i) is true for every affine scheme $\calG$ of finite presentation over $\calO$: Let $\calG=\bbA^1_\calO$ first. Denote by $I_D$ the invertible ideal defined by $D$ in $\calO\pot{D}$. By the preceding discussion, the ring $\calO\pot{D}/I_D$ is isomorphic to the global sections of $D$ and both are finite locally free $\calO$-modules, cf. \cite[Tag 0B9C]{StaPro}. For any $a\in \bbZ$, we form $I_D^a$ as an invertible $\calO\pot{D}$-module. For $a\leq b$, denote by $E_{[a,b]}$ the $\calO$-module $I_D^{a}/I_D^b$ which is also finite locally free (hence reflexive) by an induction argument. As $b$ varies, the set of $\calO$-modules $\{E_{[a,b]}\}_{b\geq a}$ forms an inverse system, and $\calO\pot{D}=\text{lim}_{b\geq 0}E_{[0,b]}$ by definition.
It follows that $I_D^{a}=\text{lim}_{b\geq a}E_{[a,b]}$ for any $a\in\bbZ$. In particular, we get $\calO\rpot{D}=\text{colim}_{a}\text{lim}_{b\geq a}E_{[a,b]}$. As $E_{[a,b]}$ is a reflexive $\calO$-module, we get for every $\calO$-algebra $R$ that
\begin{equation}\label{dual}
E_{[a,b]}\otimes_\calO R= \Hom_{\calO\text{-Mod}}((E_{[a,b]})^\ast,R)=\Hom_{\calO\text{-Sch}}(\Spec(R),\bbV_{[a,b]}),
\end{equation}
where $\bbV_{[a,b]}=\Spec(\text{Sym}^\otimes(E_{[a,b]})^\ast)$ for every pair $ b\geq a$. Taking limits shows that 
\[
\bbA^1_\calO(R\pot{D})=R\pot{D}=\text{lim}_{b\geq 0}(E_{[0,b]}\otimes_\calO R)
\] 
is identified with the $R$-points of the affine $\calO$-scheme $\text{lim}_{b\geq 0}\bbV_{[0,b]}$. The same argument shows that $R\mapsto \bbA^1_\calO(R\rpot{D})$ is representable by the ind-affine ind-scheme $\text{colim}_{a}\text{lim}_{b\geq a}\bbV_{[a,b]}$. This gives part i) in the case $\calG=\bbA^1_\calO$. For the general case, one verifies that the $L^+$-construction (resp. $L$-construction) commutes with taking finite products and equalizers, and that finite products and equalizers are constructed termwise in the category of ind-schemes. Hence, the lemma follows for $L^+\bbA^n_\calO$ (resp. $L\bbA^n_\calO$). A finite presentation $\calG=\Spec(\calO[t_1,\ldots,t_n]/(f_1,\ldots,f_m))$ realizes $\calG$ as the equalizer of the two maps $\varphi, \psi\co \bbA^n_\calO\to \bbA^m_\calO$ where $\varphi$ is given by the functions $f_1,\ldots,f_m$ and $\psi$ is the composition of the structure map with the zero section. Hence, $L^+\calG$ (resp. $L\calG$) is the equalizer of $L^+\varphi$ and $L^+\psi$ (resp. $L\varphi$ and $L\psi$) in the category of schemes (resp. ind-schemes). As equalizers define closed subschemes and $L^+\bbA^n_\calO$ is affine (resp. $L\bbA^n_\calO$ ind-affine), i) follows. \smallskip\\
Part ii) is true for every smooth affine $\calO$-scheme $\calG$, necessarily of finite presentation: For $n\geq 0$, let $D_n=\Spec(\calO\pot{D}/I_{D}^{n+1})$ be the $n$-th infinitesimal neighborhood of $D$ in $X$. The Weil restriction of scalars $\calG_n:=\Res_{D_n/\calO}(\calG\times_XD_n)$ is a smooth affine $\calO$-group scheme, cf.~ \cite[\S7.6, Thm.~4, Prop.~5]{BLR90}. For varying $n$, these groups fit into an inverse system $\calG_m\to \calG_n$ for $m\geq n$, and the natural map of functors 
\begin{equation}\label{Inv_Limit_Loop}
L^+\calG\,\overset{\simeq}{\longto}\, \on{lim}_{n\geq 0}\calG_n
\end{equation}
is an isomorphism. This proves ii), and the lemma follows.
\end{proof}

\begin{rmk} \label{Multiple_Rmk}
If $nD$ is a positive multiple of $D$, then there is a canonical isomorphism $\calO\pot{D}\overset{\simeq\;}{\to} \calO\pot{nD}$ (resp. $\calO\rpot{D}\overset{\simeq\;}{\to} \calO\rpot{nD}$). Indeed, as $I_{nD}=I_D^n\subset I_D$, the ring $\calO\pot{D}$ is complete with respect to the $I_{nD}$-adic topology, and hence $\calO\pot{D}\simeq \on{lim}_{k\geq 0}R\pot{D}/I_{nD}^k=R\pot{nD}$. 
\end{rmk}

\begin{lem} \label{BL_Lem}
i\textup{)} The loop group $L \calG$ represents the functor on the category of $\calO$-algebras which assigns to every $R$ the set of isomorphism classes of triples $(\calF,\al,\be)$, where $\calF$ is a $\calG$-torsor on $X_R$, $\al\co\calF|_{X_R\bslash D_R}\stackrel{\simeq}{\longto}\calF_0$ \textup{(}resp. $\be\co\calF_0\stackrel{\simeq}{\longto}\calF|_{\hat{D}_R}$\textup{)} is a trivialization over $X_R\bslash D_R$ \textup{(}resp. $\hat{D}_R$\textup{)}. \smallskip \\
ii\textup{)} The projection $L \calG\to\Gr_\calG$, $(\calF,\al,\be)\to(\calF,\al)$ is a right $L^+\calG$-torsor in the \'etale topology, and induces an isomorphism of sheaves $L \calG/L^+\calG\stackrel{\simeq}{\longto}\; \Gr_\calG$.
\end{lem}
\begin{proof} Part i) is deduced from the Beauville-Laszlo theorem \cite{BL95}, cf. \cite[\S 2.12]{BD} for a further discussion (cf. also \cite[Lem.~6.1]{PZ13}). For ii), it is enough to prove that the projection $L \calG\to\Gr_\calG$ admits sections \'etale locally. 

Let $R$ be an $\calO$-algebra, and let $\calF\to\hat{D}_R$ be a $\calG$-torsor.
We have to show that $\calF$ is trivial \'etale locally on $R$, i.e., admits a $\hat{D}_R$-section \'etale locally on $R$.
By applying the lifting criterion for smoothness and an algebraization result for sections (algebraization is easy because $\calF$ is affine), it is enough to show that the restriction $\calF|_{D_R}\to D_R$ admits a section \'etale locally on $R$.
Since the functor $\calF|_{D_R}\co R\text{-Alg}\to \text{Sets}, B \mapsto \calF(D_B)$ commutes with filtered colimits (because $\calF$ is a scheme locally of finite presentation \cite[01ZC]{StaPro}), we may asssume without loss of generality that $R$ is a strictly Henselian local $\calO$-algebra. 
Now by assumption on $D$ the $R$-algebra $R':=\Ga(D_R,\calO_{D_R})$ is finite, and hence a direct product of Henselian local rings $R'=R_1\times\ldots\times R_n$, cf.~\cite[04GH]{StaPro}. 
As $R$ is strictly Henselian, each $R_i$ is strictly Henselian as well (because a finite extension of a separably closed field is separably closed).
But each non-empty smooth scheme over a finite product of strictly Henselian local rings admits a section by Hensel's lemma. 
This finishes the proof. 
\end{proof}

Lemma \ref{BL_Lem} ii) shows that there is a transitive action map
\begin{equation}\label{Action_Map}
L\calG\times_\calO\Gr_\calG\,\longto\, \Gr_\calG.
\end{equation}
Let us look at the fibers of \eqref{Action_Map} over $\calO$.

\begin{cor}\label{Fibers_Action_Map_Cor} i\textup{)} Let $F$ be a field, and let $\calO\to F$ be a ring morphism. The underlying reduced subscheme $D_{F,\red}\subset D_F$ is an effective Cartier divisor on $X_F$, and we write $D_{F,\red}=\sum_{i=1}^nD_i$ where $D_i$ are distinct irreducible, i.e., the $D_i$ are closed points of $X_F$. There is a canonical isomorphism of $F$-spaces
\[
\Gr_{(X,\calG,D)}\otimes_\calO F \,\overset{\simeq}{\longto}\, \prod_{i=1}^n \Gr_{(X_F, \calG_F, D_{i})},
\]  
compatible with the action of $L\calG_{(X,\calG,D)}\otimes_\calO F\simeq \prod_{i=1}^n L\calG_{(X_F, \calG_F, D_{i})}$.\smallskip\\
ii\textup{)} Let $\calO=F$ be a field, and let $D=[x]$ be the divisor on $X$ defined by a closed point $x\in X$. The residue field $K:=\kappa(x)$ is a finite field extension, and we assume that $K/F$ is separable. There is a canonical isomorphism of $F$-spaces
\[
\Gr_{(X,\calG,D)}\,\overset{\simeq}{\longto}\, \Res_{K/F}(\Gr_{(X_K,\calG_{X_K},D)})
\]
compatible with the action of $L\calG_{(X,\calG,D)}\simeq \Res_{K/F}(L\calG_{(X_K,\calG_{X_K},D)})$.

\end{cor}
\begin{proof} For i), we may by \eqref{Base_Change} assume $\calO=F$. It is immediate from Remark \ref{Multiple_Rmk} that for any $\calO$-algebra $R$, we have $R\pot{D_\red}\simeq R\pot{D}$ (resp. $R\rpot{D_\red}\simeq R\rpot{D}$). Further, there is a canonical isomorphism
\[
R\pot{D_\red} \,\overset{\simeq}{\longto}\, \prod_{i=1}^nR\pot{D_i}\;\;\;\; \text{(resp. $R\rpot{D_\red} \,\overset{\simeq}{\longto}\, \prod_{i=1}^nR\rpot{D_i}$)}
\]
because $X$ is of dimension $1$, and hence $D_i\cap D_j=\varnothing$ for $i\not = j$. Part i) follows from Lemma \ref{BL_Lem} ii). For ii), first note that if we consider $D$ as the divisor on $X_K$ defined by the $K$-point $x$, then $\Gr_{(X_K,\calG_{X_K},D)}$ is the twisted affine Grassmannian over $K$, cf. Example \ref{Special_Cases} i). Let $\tilde{K}/F$ be the splitting field of $K$ which is a finite Galois extension with Galois group $\tilde{\Ga}$. There is a canonical isomorphism of $\tilde{K}$-algebras
\[
K\otimes_F\tilde{K}\,\overset{\simeq}{\longto}\, \prod_{\psi\co K\hookto \tilde{K}}\!\!\tilde{K},\;\;\;\;\;\; a\otimes b\,\longmapsto\, (\psi(a)\cdot b)_\psi,
\]
which is $\tilde{\Ga}$-equivariant for the action $\ga*(c_\psi)_\psi\mapsto (\ga(c_\psi))_{\ga\psi}$ on the target. Applying this isomorphism to $D\otimes_F\tilde{K}$, we obtain by i) a $\tilde{\Ga}$-equivariant isomorphism
\begin{equation}\label{DescentDatum}
\Gr_{(X,\calG,D)}\otimes_F \tilde{K} \,\overset{\simeq}{\longto}\, \prod_{\psi} \Gr_{(X_{K}, \calG_{X_{K}}, D)}\otimes_{K,\psi}\tilde{K},
\end{equation}
compatible with the actions of the loop groups. The canonical descent datum on the source in \eqref{DescentDatum} induces a descent datum on the target of \eqref{DescentDatum} which implies ii). 
\end{proof}

Let us point out some useful compatibility with Weil restriction of scalars.

\begin{cor}\label{cor_weil_restriction}
Let $X'\to X$ be a finite flat surjective map of smooth quasi-projective $\calO$-curves, and assume $\calG=\Res_{X'/X}(\calG')$ for a smooth affine $X'$-group scheme $\calG'$. If $D':=D\times _XX'$, then the natural map is an isomorphism of $\calO$-spaces
\begin{equation}\label{cor_res_map}
\Gr_{(X',\calG',D')}\overset{\simeq\;}{\to}\Gr_{(X,\calG, D)},\;\; (\calF',\al')\mapsto (\Res_{X'/X}(\calF'),\Res_{X'/X}(\al')).
\end{equation}
\end{cor}
\begin{proof} Since $X'\to X$ is finite flat surjective, the closed subscheme $D'\subset X'$ is a relative effective Cartier divisor which is finite flat over $\calO$. Hence, the functor $\Gr_{(X',\calG',D')}$ is well defined. Using Lemma \ref{BL_Lem} ii), the map \eqref{cor_res_map} is induced for any $\calO$-algebra $R$ by the canonical map of $R$-algebras
\[ 
\Spec(R\pot{D'}) \,\to\, \Spec(R\pot{D})\times_XX' \;\;\;\; \text{(resp. $\Spec(R\rpot{D'}) \,\to\, \Spec(R\rpot{D})\times_XX'$)}.
\]
If $R$ is Noetherian, then the first map (hence the second map) is an isomorphism by \cite[00MA]{StaPro} because $X'\to X$ is finite. In particular, \eqref{cor_res_map} is an isomorphism for any Noetherian $\calO$-algebra $R$. As both functors in \eqref{cor_res_map} commute with filtered colimits of $\calO$-algebras, the corollary follows.
\end{proof}

\begin{lem}\label{tomato_lem}
Let $\calO'\to \calO$ be a finite \'etale map of Noetherian rings. 
Then the composition $X\to \Spec(\calO)\to \Spec(\calO')$ is a smooth curve as well, and there is a canonical isomorphism of functors
\[
\Gr_{(X/\calO', \calG, D)}\;\simeq\;\Res_{\calO/\calO'}\big(\Gr_{(X/\calO,\calG, D)}\big).
\]
 \end{lem}
\begin{proof} 
If $T\to \Spec(\calO')$ is a test scheme, then $X\times_{\Spec(\calO)}(\Spec(\calO)\times_{\Spec(\calO')}T)=X\times_{\Spec(\calO')}T$.
The lemma follows immediately from the definitions. 
\end{proof}

\subsubsection{Basic representability properties} The starting point is the following lemma, and we sketch its proof.

\begin{lem} \label{Aff_Grass_Rep_Lem}
 If $\calG=\Gl_{n,X}$, then the functor $\Gr_{\calG}$ is representable by an ind-projective $\calO$-ind-scheme.
\end{lem}
\begin{proof}
%{\cg I don't think we should assume $R$ is Noetherian now, but we need to reduce to that case later.} {\cm Let $R$ be a Noetherian $\calO$-algebra.} 
 Let $R$ be an $\calO$-algebra. If $\calG=\Gl_{n,X}$, then $\Gr_{\calG}(R)$ classifies rank $n$ vector bundles $\calE$ on $X_R$ together with an isomorphism $\calE|_{U_R}\simeq \calO_{U_R}^n$ where $U_R:=(X\bslash D)_R$. Let $\calI_{D_R}\subset \calO_{X_R}$ be the invertible ideal sheaf defined by $D_R\subset X_R$. For $N\geq 1$, let $\Gr_{\calG,N}$ be the $\calO$-space whose $R$-valued points are rank $n$ vector bundles $\calE$ on $X_R$ such that as $\calO_{X_R}$-modules
\[
\left(\calI_{D_R}^N\right)^n\,\subset\,\calE\,\subset\, \left( \calI_{D_R}^{-N}\right)^n.
\]
Every vector bundle is locally free and by bounding the poles (resp. zeros) of basis elements, one gets as $\calO$-spaces
\[
\on{colim}_{N\geq 1}\Gr_{\calG,N}\,\overset{\simeq}{\longto}\, \Gr_{\calG}.
\]
We claim that each $\Gr_{\calG,N}$ is representable by a $\on{Quot}$-scheme as follows. The $\calO_{X_R}$-module $\calE_{N,R}:=(\calI_{D}^{-N}/\calI_{D}^N)^n\otimes_\calO R$ is coherent and locally free over $R$. Let $\on{Quot}_N$ be the $\calO$-space whose $R$-points are coherent $\calO_{X_R}$-module quotients $\calE_{N,R}\onto \calQ$ which are locally free $R$-modules. The functor $\on{Quot}_N$ is representable by a projective $\calO$-scheme by the theory of $\on{Quot}$-schemes applied to the finite flat $\calO$-scheme $2ND$, and the coherent $\calO_{2ND}=\calO_{X}/\calI_D^{2N}$-module $\calE_{N,\calO}$. More precisely, in the notation of \cite[\S5.1.4]{FGA}, one has a finite disjoint union
\[
\on{Quot}_N\,=\,\coprod_{r\in\bbZ_{\geq 0}} \on{Quot}_{\calE_{N,\calO}/{ 2ND}/\Spec(\calO)}^{r, \calO_{{2ND}}},
\]
and the representability result is then a theorem of Grothendieck \cite[\S5.5.2, Thm.~5.14]{FGA}. Note that the structure sheaf $\calO_{{ 2ND}}$ is relatively ample for ${ 2ND} \to \Spec(\calO)$ because the map is finite {(cf.\, \cite[Tag 01VG, 28.35.6]{StaPro})}. Concretely, $\on{Quot}_N$ is the closed subscheme of the Grassmannian
\[
\on{Quot}_N\,\hookto\, \on{Grass}(\calE_{N,\calO}),
\]
which is cut out by the condition that the quotients are stable under the finitely many nilpotent operators $u_1,\ldots, u_n$ on $\calE_{N,\calO}$ induced by some presentation ${ 2ND}=\Spec(\calO[u_1,\ldots,u_n]/J)$. Hence, to prove the lemma it is enough to show that as functors
\begin{equation}\label{key_step}
\Gr_{\calG,N}\,\overset{\simeq}{\longto}\, \on{Quot}_N,\;\;\;\calE \longmapsto \left(\calI_{D_R}^{-N}\right)^n/ \calE.
\end{equation}
We need to check that $\calQ:= \left(\calI_{D_R}^{-N}\right)^n/ \calE$ is a locally free $R$-module. This follows from the isomorphism as $R$-modules $\calO_{U_R}^n/\calE\simeq \oplus_{k\geq 0} \calI_{D_R}^{-k-1}\calE/\calI_{D_R}^{-k}\calE$, and the short exact sequence  
\[
0\,\to\, \left(\calI_{D_R}^{-N}\right)^n/\calE \,\to\, \calO_{U_R}^n/\calE \,\to\, \calO_{U_R}^n/ \left(\calI_{D_R}^{-N}\right)^n \,\to\, 0,
\]
cf.~ also the argument in \cite[Lem.~1.1.5]{Zhu}. Conversely, let $\calQ\in \on{Quot}_N(R)$, and define the coherent $\calO_{X_R}$-module
\[
\calE\defined \ker\left(\left( \calI_{D_R}^{-N}\right)^n\to \calE_{N,R}\to \calQ\right).
\]
We need to show that $\calE$ is a rank $n$ vector bundle on $X_R$. Covering $X_R$ with affine schemes, we may assume $X_R=\Spec(S)$ is affine. Let $\frakp\subset  S$ be a prime ideal lying over a prime ideal $\frakm:=\frakp\cap R\subset R$. By \cite[Tag 00M]{StaPro} applied to the map of local rings $R_\frakm\to S_\frakp$ and the module $\calE_\frakp$ (note that $\calE_\frakp$ is still $R_\frakm$-flat), to prove $\calE_{\frakp}$ is free over $S_{\frakp}$ we are reduced to the case where $R$ is a field. In the  case where $R$ is a field, $\calE \subset \left( \calI_{D_R}^{-N}\right)^n$ is a torsion-free rank $n$ submodule, and since $X_R\to \Spec(R)$ is a smooth curve, $\calE$ is a vector bundle.
\end{proof}

\begin{rmk} Using Lemma \ref{BL_Lem} ii), the set $\Gr_{\Gl_{n,X}}(\calO)$ can be identified with the set of $\calO\pot{D}$-lattices in $\calO\rpot{D}$, i.e., in the notation of Lemma \ref{Loop_Group_Rep_Lem}, the set of $\calO\pot{D}$-submodules $M\subset \calO\rpot{D}$ such that for some $N>\!\!>0$, $\left(I_{D}^N\right)^n\subset M\subset \left(I_{D}^{-N}\right)^n$ and $\left(I_{D}^{-N}\right)^n/M$ is a locally free $\calO$-module. 
\end{rmk}

\begin{prop} \label{Aff_Grass_Rep_Prop}
If $\calG\hookto G$ is a monomorphism of smooth $X$-affine $X$-group schemes such that the fppf-quotient $G/\calG$ is a $X$-quasi-affine scheme \textup{(}resp.~$X$-affine scheme\textup{)}, then the map $\Gr_\calG\to \Gr_{G}$ is representable by a quasi-compact immersion \textup{(}resp.~closed immersion\textup{)}.
\end{prop}
\begin{proof}  Following the proof of \cite[Prop.~1.2.6]{Zhu}, it is enough to establish the analogue of \cite[Lem.~1.2.7]{Zhu}. Let $R$ an $\calO$-algebra, and let $p\co V\to \hat{D}_R$ be an affine scheme of finite presentation. Let $s$ be a section of $p$ over $\hat{D}_R^o$. We need to prove that the presheaf assigning to any $R$-algebra $R'$, the set of sections $s'$ of $p$ over $\hat{D}_{R'}$ such that $s'|_{\hat{D}_{R'}^o}=s|_{\hat{D}_{R'}^o}$ is representable by a closed subscheme of $\Spec(R)$. Indeed, choosing a closed embedding $V\subset \bbA^n_{\hat{D}_R}$ for some $n>\!\!>0$ and using that $R\pot{D}\subset R\rpot{D}$ is injective, we reduce to the case $V=\bbA^n_{\hat{D}_R}$. The presheaf in question is representable by the locus on $\Spec(R)$ where the class $\bar{s}$ of the section $s\in V(\hat{D}_R^o)=R\rpot{D}^n$ in $(R\rpot{D}/R\pot{D})^n$ vanishes. With the notation of Lemma \ref{Loop_Group_Rep_Lem}, we have $\bar{s}\in E_{[-N,0]}\otimes_\calO R$ for some $N>\!\!>0$. As $E_{[-N,0]}$ is a reflexive $\calO$-module, we see that giving an element of $E_{[-N,0]}\otimes_\calO R$ is equivalent to giving a map of $R$-schemes $\Spec(R)\to \bbV(E_{[-N,0]}\otimes_\calO R)$. Then the presheaf in question is representable by the equalizer of the two maps corresponding to the elements $\bar{s},0\in E_{[-N,0]}\otimes_\calO R$ which is a closed subscheme of $\Spec(R)$.
\end{proof}

\begin{cor}\label{Aff_Grass_Rep_Cor}
i\textup{)} If there exists a monomorphism $\calG\hookto \Gl_{n,X}$ such that the fppf-quotient is a $X$-quasi-affine scheme \textup{(}resp. an $X$-affine scheme\textup{)}, then $\Gr_\calG=\on{colim}_i\Gr_{\calG,i}$ is representable by a separated $\calO$-ind-scheme of ind-finite type \textup{(}resp. separated ind-proper $\calO$-ind-scheme\textup{)}. Each $\Gr_{\calG,i}$ can be chosen to be $L^+\calG$-stable. \smallskip\\ 
ii\textup{)} If in i\textup{)} the representation $\calG\hookto\Gl_{n,X}$ exists \'etale locally on $\calO$, then $\Gr_\calG=\on{colim}_i\Gr_{\calG,i}$ is a separated $\calO$-ind-algebraic space of ind-finite presentation \textup{(}resp. separated ind-proper $\calO$-ind-algebraic space\textup{)}. Each $\Gr_{\calG,i}$ can be chosen to be $L^+\calG$-stable.\smallskip\\  
iii\textup{)} If $\calG=G\otimes_\calO X$ is constant and $G$ is a reductive $\calO$-group scheme, then $\Gr_\calG$ is representable by an ind-proper $\calO$-ind-algebraic space.
\end{cor}
\begin{proof} Part i) is immediate from Lemma \ref{Aff_Grass_Rep_Lem} and Proposition \ref{Aff_Grass_Rep_Prop}. For ii), we use part i) together with Lemma \ref{Tech_Lem} below. Note that the diagonal of $\Gr_{\calG}$ being representable by a closed immersion follows from the same property of $\Gr_{\Gl_{n,\calO}}$ and the effectivity of descent for closed immersions. Further, if $\calO\to \calO'$ is \'etale, then the method of Lemma \ref{Tech_Lem} shows that an $L^+\calG\otimes_\calO{\calO'}$-stable presentation of $\Gr_{\calG}\otimes_\calO{\calO'}$ induces an $L^+\calG$-stable presentation of $\Gr_\calG$ (because $L^+\calG$ is affine and flat, and taking the scheme theoretic image commutes with flat base change). For iii), note that after an \'etale cover $\calO\to \calO'$, the group scheme $G_{\calO'}:=G\otimes_\calO\calO'$ is split reductive, and in particular linearly reductive. If we choose a closed immersion $G_{\calO'}\hookto \Gl_{n,\calO'}$, then the quotient $\Gl_{n,\calO'}/G_{\calO'}$ is representable by an affine scheme by \cite[Cor 9.7.7]{Al14}, and iii) follows from ii).
\end{proof}

\begin{lem}\label{Tech_Lem}
Let $X$ be an $\calO$-space with schematic diagonal, and such that there exists a \'etale surjective \textup{(}as sheaves\textup{)} map of $\calO$-spaces $U\to X$ with $U$ an $\calO$-ind-scheme. If either $U\to X$ is quasi-compact or $U$ is quasi-separated, then $X$ is an $\calO$-ind-algebraic space.
\end{lem}
\begin{proof} Given a presentation $U=\on{colim}_{i\in I}U_i$ with $U_i$ being schemes, we need to construct a presentation $X=\on{colim}_{i\in I}X_i$ with $X_i$ being algebraic spaces. For each $i$, consider $U_i\subset U\to X$. We define $X_i'$ to be the scheme theoretic image of the map 
\begin{equation}\label{Tech_Lem_Map}
U_i\times_XU\,\subset\, U\times_XU\,\overset{p_2\;}{\longto}\, U. 
\end{equation}
This well defined for the following reason: Since $U_i\times_XU$ is a quasi-compact scheme, the map \eqref{Tech_Lem_Map} factors through $U_j\subset U$ for some $j>\!\!>i$. In either case, $U\to X$ quasi-compact or $U$ quasi-separated, the map \eqref{Tech_Lem_Map} is quasi-compact. By \cite[01R8]{StaPro}, the scheme theoretic image behaves well for quasi-compact maps, and $X_i'\subset U_j$ is a quasi-compact closed subscheme. As scheme theoretic images of quasi-compact maps commute with flat base change \cite[Tag 081I]{StaPro}, the scheme $X_i'$ is equipped with a descent datum relative to $U\to X$, and defines a closed $\calO$-subspace $X_i\subset X$ together with an \'etale surjective map $X_i'\to X_i$. As $X_i\subset X$ is closed, the diagonal of $X_i$ is schematic, and $X_i$ is a quasi-compact algebraic space. By construction the $X_i$ form a filtered direct system indexed by the poset $I$, and the canonical map $\on{colim}_{i\in I}X_i\to X$ is an isomorphism (because $U\to X$ is a sheaf surjection, and $\on{colim}_iX_i'=U$ by construction). 
\end{proof}

\begin{rmk} 
It would be nice to give a proof of representability of $\Gr_\calG$ which does not refer to the choice of an embedding $\calG\hookto \Gl_{n,X}$. 
\end{rmk}

%%%%%%%%%The open cell%%%%%%%%%%%%%%
\subsection{The open cell}\label{Open_Cell_Sec}
In the following two subsections, we apply our methods to prove Theorem \ref{GCT_Geo_Thm}, a generalization of Theorem A from the introduction. The results are not used in the proof of our Main Theorem.

We specialize to the case where $X=\bbA^1_\calO$, and where $\calG=G\otimes_\calO X$ is constant, i.e., the base change of a smooth affine $\calO$-group scheme $G$ of finite presentation. In this case, we denote $L_D\calG$ (resp. $L^+_D\calG$; resp. $\Gr_{(X,\calG,D)}$) by $LG=L_DG$ (resp. $L^+G=L^+_DG$; resp. $\Gr_G=\Gr_{(X,G,D)}$).

Since $D\subset \bbA^1_\calO$ is assumed to be finite over $\calO$, the subscheme $D\subset \bbP^1_\calO$ is closed and defines a relative effective Cartier divisor. In particular, Lemma \ref{BL_Lem} ii) (the Beauville-Laszlo lemma) implies that $\Gr_{(\bbA^1_\calO,G,D)}=\Gr_{(\bbP^1_\calO,G,D)}$ by extending torsors trivially to $\infty$. 

The \emph{negative loop group} is the functor on the category of $\calO$-algebras
\begin{equation}\label{Neg_Loop_Group}
L^- G\co R\mapsto G(\bbP^1_R\bslash D_R).
\end{equation}
Then $L^-G$ is an $\calO$-space which is a subgroup functor $L^- G\subset L G$. 

\begin{lem}\label{Neg_Loop_Rep_Lem}
The functor $L^- G$ is representable by an ind-affine ind-scheme locally of ind-finite presentation over $\calO$. 
\end{lem}
\begin{proof}
That the affine schemes are of finite presentation follows from the fact that $L^-G$ commutes with filtered colimits (because $G$ is of finite presentation). One verifies that $L^-$ commutes with finite products and equalizers, and hence the proof of representability is reduced to the case $G=\bbA^1_\calO$, cf. the proof of Lemma \ref{Loop_Group_Rep_Lem}. We have to show that the functor on the category of $\calO$-algebras $R$ given by the global sections $R\mapsto \Ga(\calO_{\bbP^1_R\bslash D_R})$ is representable by an ind-affine ind-scheme. But as $R$-modules $\Ga(\calO_{\bbP^1_R\bslash D_R})=\on{colim}_n\Ga(\calO_{\bbP^1_R}(nD_R))$, and we claim that $\Gamma(\calO_{\bbP^1_R}(nD_R))$ is finite locally free: Indeed, this follows from the short exact sequence
\[
0\to \calO_{\bbP^1_R}\to \calO_{\bbP^1_R}(nD_R)\to \calI^{-1}_{nD_R}/\calO_{\bbP^1_R}\to 0,
\]
and the vanishing of $H^1_{\text{Zar}}(\bbP^1_R, \calO_{\bbP^1_R})$. This proves the lemma.
\end{proof}

Now define $L^{--}G=\ker(L^-G\to G)$ for $g\mapsto g(\infty)$. Then the intersection $L^{--}G\cap L^+G$ is trivial inside $L G$, and we consider the orbit map
\begin{equation}\label{Open_Cell}
L^{--}G\,\longto\,\Gr_G,\;\;\; g^-\longmapsto g^-\cdot e_0,
\end{equation}
where $e_0\in\Gr_G$ denotes the base point.

\begin{lem}\label{Open_Cell_Lem}
The map \eqref{Open_Cell} is representable by an open immersion, and identifies $L^{--}G$ with those pairs $(\calF,\al)$ where $\calF$ is the trivial torsor.
\end{lem}
\begin{proof} The argument is the same as the deformation argument given in \cite[Lem.~3.1]{HaRi}, and we do not repeat it here.
\end{proof}

\iffalse
\begin{ex} Let $G=\Gl_{n,R}$, and let $D$ be defined by the zero section of $\bbA^1_R$. Then $R\pot{D}=R\pot{t}$ is the ring of formal power series, and $R\rpot{D}=R\rpot{t}$ is the ring of Laurent series. We like to describe $L^{--}G(R)$ as a subset of $\Gr_G(R)$. We have
\[
\Gr_G(R)=\{\text{$\calF\subset R\rpot{t}^n$ loc. free $R\pot{t}$-submodule of rank $n$ with $\calF[t^{-1}]=R\rpot{t}^n$}\}.
\] 
Consider the decomposition $R\rpot{t}=t^{-1}R[t^{-1}]\oplus R\pot{t}$ as $R$-modules. Then $L^{--}G(R)$ is the subset of $\calF\in \Gr_G(R)$ such that $\calF\oplus (t^{-1}R[t^{-1}])^n=R\rpot{t}^n$ (which is an open condition in the affine Grassmannian). Such an $\calF$ is necessarily trivial: writing the standard basis $(e_1,\ldots,e_n)$ of $R\rpot{t}^n$ as $e_i=f_i+h_i$ with $f_i\in\calF$, $h_i\in (t^{-1}R[t^{-1}])^n$, then an approximation argument shows that $\langle f_1,\ldots,f_n\rangle=\calF$ as $R\pot{t}$-modules. 
\end{ex}
\fi

%%%%%%%%%Gm-actions on \Gr_G%%%%%%%%%%%%%%%%%%%
\subsection{Geometry of $\bbG_m$-actions on $\Gr_G$} We assume $X=\bbA^1_\calO$, and $\calG=G\otimes_\calO X$ with $G$ being a reductive $\calO$-group scheme with connected (and hence geometrically connected) fibers. Let $\chi\co \bbG_{m,\calO}\to G$ be an $\calO$-rational cocharacter. The cocharacter $\chi$ induces via the composition
\begin{equation}\label{Gm_Action}
\bbG_{m,\calO}\subset L^+\bbG_{m,\calO}\overset{L^+\!\chi\phantom{h}}{\longto}L^+G\subset LG
\end{equation}
a left $\bbG_m$-action on the affine Grassmannian $\Gr_G\to \Spec(\calO)$. As in \eqref{hyper_loc_maps}, we obtain maps of $\calO$-spaces
\begin{equation}\label{Hyper_Loc_Grass}
(\Gr_G)^0\leftarrow (\Gr_G)^\pm\to \Gr_G.
\end{equation}
Let us mention the following lemma which implies the ind-representability of the spaces \eqref{Hyper_Loc_Grass}, in light of Theorem \ref{Gm_thm} and Corollary \ref{Aff_Grass_Rep_Cor}.

 \begin{lem}\label{Loc_Lin_Lem}
The $\bbG_m$-action on $\Gr_G$ is \'etale locally linearizable.
\end{lem}
\begin{proof} 
After an \'etale cover $\calO\to \calO'$, there exists a closed immersion $\Gr_{G_\calO'}\to \Gr_{\Gl_{n,\calO'}}$ (cf. Proposition \ref{Aff_Grass_Rep_Cor} iii)) which is $\bbG_m$-equivariant with respect to the action on $\Gr_{\Gl_{n,\calO'}}$ given by the cocharacter $\bbG_{m,\calO'}\overset{\chi\,}{\to} G_{\calO'}\to \Gl_{n,\calO'}$. The proof of Lemma \ref{Tech_Lem} shows that an $L^+G_{\calO'}$-stable presentation of $\Gr_{G_{\calO'}}$ by quasi-compact schemes induces an $L^+G$-stable presentation of $\Gr_G$ by quasi-compact algebraic spaces. To prove the lemma it is enough to show that the $\bbG_m$-action on $ \Gr_{\Gl_{n,\calO'}}$ is Zariski locally linearizable, and we reduce to the case $\calO=\calO'$, $G=\Gl_{n,\calO}$. By \cite[Prop.~6.2.11; Prop.~3.1.9]{Co14}, Zariski locally on $\calO$ the cocharacter $\chi$ lies in a split maximal torus in $\Gl_{n,\calO}$ which is $\calO$-conjugate to the diagonal matrices in $\Gl_{n,\calO}$, \iffalse {\cg I don't understand the next sentence} {\cm Then there exists a product decomposition as rings $\calO=\calO_1\times\ldots\times \calO_k$ for some $k\in\bbZ_{\geq 1}$ such that $\chi|_{\calO_i}$ is defined over $\bbZ$} {\cy That's a little of road. It's about the structure of $X:=\Hom(\bbG_{m,\calO}, \bbG_{m,\calO})$ for general rings. Of course, if $\Spec(\calO)$ is connected, then $X=\bbZ$ and every map comes by base change from $\bbZ$, i.e., is ``defined over $\bbZ$''. If $\Spec(\calO)$ is the disjoint union of several connected components, then $X$ is the direct sum of several copies of $\bbZ$. We can avoid doing this, by localizing on $\calO$, and then assume that $\chi$ is a standard cocharacter coming by base change from $\bbZ$.}\fi and hence is after conjugation with a permutation matrix dominant. In this way, we reduce to the case where $\chi$ is a standard dominant cocharacter given by $\la\mapsto \on{diag}(\la^{a_1},\ldots,\la^{a_n})$ for some integers $a_1\geq \ldots\geq a_n$. With the notation of Lemma \ref{Aff_Grass_Rep_Lem}, it is now immediate that the $\bbG_m$-action on $\on{Quot}_N\subset \on{Grass}(\calE_{N,\calO})$ is linear, and compatible with the transition maps for varying $N$. The lemma follows.
\end{proof}

Our aim is to express \eqref{Hyper_Loc_Grass} in terms of group theoretical data related to the cocharacter $\chi$, cf. Theorem \ref{GCT_Geo_Thm} below. 

Let $\chi$ act on $G$ via conjugation $(\la,g)\mapsto \chi(\la)\cdot g\cdot\chi(g)^{-1}$. The fixed points $M=G^0$ (resp. the attractor $P^+=G^+$; resp. the repeller $P^-=G^-$) defines a closed subgroup of $G$ which is smooth of finite presentation over $\calO$, cf. \cite{Mar15}. The group $M$ is the centralizer of $\chi$, and is by the classical theory over a field a reductive $\calO$-group scheme which is fiberwise connected (hence fiberwise geometrically connected). By \eqref{hyper_loc_maps} we have natural maps of $\calO$-groups
\begin{equation}\label{hyperlocgroup}
M\leftarrow P^\pm \to G.
\end{equation}

\begin{thm} \label{GCT_Geo_Thm}
The maps \eqref{hyperlocgroup} induce a commutative diagram of $\calO$-ind-algebraic spaces
\begin{equation}\label{gctgeo_diag}
\begin{tikzpicture}[baseline=(current  bounding  box.center)]
\matrix(a)[matrix of math nodes, 
row sep=1.5em, column sep=2em, 
text height=1.5ex, text depth=0.45ex] 
{\Gr_M & \Gr_{P^\pm} & \Gr_G \\ 
(\Gr_{G})^0& (\Gr_{G})^\pm& \Gr_{G}, \\}; 
\path[->](a-1-2) edge node[above] {}  (a-1-1);
\path[->](a-1-2) edge node[above] {}  (a-1-3);
\path[->](a-2-2) edge node[below] {}  (a-2-1);
\path[->](a-2-2) edge node[below] {} (a-2-3);
\path[->](a-1-1) edge node[left] {$\iota^0$} (a-2-1);
\path[->](a-1-2) edge node[left] {$\iota^\pm$} (a-2-2);
\path[->](a-1-3) edge node[left] {$\id$} (a-2-3);
\end{tikzpicture}
\end{equation}
where the vertical maps $\iota^0$ and $\iota^\pm$ are isomorphisms.
\end{thm}

\begin{rmk} i) An interesting example to which Theorem \ref{GCT_Geo_Thm} applies is the case of fusion Grassmannians $\Gr_G\to \bbA^n_F$, cf. Example \ref{Special_Cases} iv) with\footnote{The case of general smooth $F$-curves $C$ can be reduced to the special case of $\bbA^1_F$, but we do not need this in the present article.} $C=\bbA^1_F$. Hence, Theorem \ref{GCT_Geo_Thm} implies Theorem A from the introduction. Note that the group $G$ need not be defined over $F$, but can be a general reductive group scheme over the $n$-th symmetric power $(\bbA^1_F)^{(n)}$. Changing the set up slightly, the group $G$ could even be a general reductive group scheme over $\bbA^n_F$ (take $D=\Spec(\calO)=\bbA^n_F$, $X=\bbA^n_F\times_F\bbA^1_F $ and consider the divisor $\bbA^n_F\to \bbA^n_F\times \bbA^1_F$, $(x_i)_i\mapsto ((x_i)_i,\sum_i x_i)$ for $i=1,\ldots, n$). %{\cm This does not seem obvious from the set-up in 3.3 -- there we assume $G$ is defined over $\calO$; if we start in the set-up of Example \ref{Special_Cases} iv), then $G$ is defined over $C^{(n)}$ not $C^n$...we have to comment how the set-up can be altered, still getting 3.15 for general groups over $C^n$. For example we can start with $G$ defined over $C^{(n)} $ and then define it's base-change which would be over $C^{n}$.  But not every group over $C^n$ should come from a group over $C^{(n)}$...}.  \smallskip\\
ii) Note that Theorem \ref{GCT_Geo_Thm} also generalizes \cite[Lem.~3.6]{HaRi} and justifies \cite[sentence containing (3.33)]{HaRi}.
\end{rmk}

%The data $(X,G,D,\chi)$ is defined over a finitely generated $\bbZ$-algebra, and the diagram \eqref{gctgeo_diag} is compatible with base change. Without loss of generality we reduce to the case where $\calO$ is a finitely generated $\bbZ$-algebra, and in particular Noetherian with a locally connected spectrum.

\subsubsection{Construction of $\iota^0$ and $\iota^\pm$}\label{Construction_Sec} The strategy of construction is the same as in \cite{HaRi} which we recall for readability.

As the $\bbG_m$-action on $\Gr_M$ is trivial, the natural map $\Gr_M\to \Gr_G$ factors as $\Gr_M\to (\Gr_G)^0\to \Gr_G$ which defines $\iota^0$. For the construction of the map $\iota^\pm$, we use the Rees construction explained in Heinloth \cite[1.6.2]{He18}. The $\bbG_m$-action $P^\pm\times \bbG_{m,\calO}\to P^\pm, (p,\la)\mapsto \chi(\la^\pm)\cdot p\cdot \chi(\la^\pm)^{-1}$ via conjugation extends via the monoid action of $\bbA^1$ on $(\bbA_\calO^1)^\pm$ in \eqref{flow} to a monoid action 
\begin{equation}\label{monoidaction}
m_\chi\co P^\pm\times \bbA^1_\calO\longto P^\pm
\end{equation}
such that $m_\chi(p,0)\in M$. We let $\on{gr}_\chi\co P^\pm\times \bbA^1_\calO\to P^\pm\times \bbA^1_\calO$, $(p, \la)\mapsto (m_\chi(p,\la),\la)$ viewed as an $\bbA^1_\calO$-group homomorphism. Then the restriction $\on{gr}_\chi|_{\{1\}}$ is the identity whereas $\on{gr}_\chi|_{\{0\}}$ is the composition $P^\pm\to M\to P^\pm$. For a point $(\calF^\pm,\al^\pm)\in \Gr_{P^\pm}(R)$, the Rees bundle is
\begin{equation}\label{Reesbundle}
\on{Rees}_\chi(\calF^\pm,\al^\pm)\defined \on{gr}_{\chi , *}(\calF^\pm_{ \bbA^1_R},\al^\pm_{ \bbA^1_R}) \in \Gr_{P^\pm}(\bbA^1_R),
\end{equation}
where $\on{gr}_{\chi , *}$ denotes the push forward under the $\bbA^1$-group homomorphism. Then the restriction $\on{Rees}_\chi(\calF^\pm,\al^\pm)|_{\{1\}_R}$ is equal to $(\calF^\pm,\al^\pm)$ whereas $\on{Rees}_\chi(\calF^\pm,\al^\pm)|_{\{0\}_R}$ is the image of $(\calF^\pm,\al^\pm)$ under the composition $\Gr_{P^\pm}\to \Gr_M\to \Gr_{P^\pm}$. One checks that $\on{Rees}_\chi(\calF^\pm,\al^\pm)$ is $\bbG_m$-equivariant, and hence defines an $R$-point of $(\Gr_{P^\pm})^\pm$. As the Rees construction is functorial, we obtain a map of $\calO$-spaces
\begin{equation}\label{Reesmap}
\on{Rees}_\chi\co \Gr_{P^\pm}\to (\Gr_{P^\pm})^\pm,
\end{equation}
which is inverse to the map $(\Gr_{P^\pm})^\pm\to \Gr_{P^\pm}$ given by evaluating at the unit section. We define the map $\Gr_{P^\pm}\to (\Gr_G)^\pm$ to be the composition $\Gr_{P^\pm}\simeq (\Gr_{P^\pm})^\pm\to (\Gr_G)^\pm$ where the latter map is deduced from the natural map $\Gr_{P^\pm}\to \Gr_G$. This constructs the commutative diagram \eqref{gctgeo_diag}.

We claim that the map $\iota^0$ (resp. $\iota^\pm$) is representable by a quasi-compact immersion. By \cite[Thm.~2.4.1]{Co14}, the fppf quotient $G/M$ is quasi-affine, and hence $\iota^0$ is representable by a quasi-compact immersion by Proposition \ref{Aff_Grass_Rep_Prop}. Note that since $M$ is reductive, the space $\Gr_M$ is ind-proper and hence $\iota^0$ is even a closed immersion. For $\iota^\pm$, we use that quasi-compact immersions are of effective descent (cf. \cite[Tag 0247, 02JR]{StaPro}), and after passing to an \'etale ring extension of $\calO$, we reduce to the case where $G$ is linearly reductive. As in the proof of Corollary \ref{Aff_Grass_Rep_Cor}, we choose $G\hookto \Gl_{n,\calO}$ such that $\Gl_{n,\calO}/G$ is quasi-affine (or even affine). Let $Q^+\subset \Gl_{n,\calO}$ (resp. $Q^-\subset \Gl_{n,\calO}$) be the attractor (resp. repeller) subgroup defined by the cocharacter $\bbG_{m,\calO}\overset{\chi}{\to} G\to \Gl_{n,\calO}$. Then we have $P^\pm=Q^\pm\times_{\Gl_{n,\calO}}G$. The quotient $Q^\pm/P^\pm$ is an algebraic space of finite presentation over $\calO$, and the map $i\co Q^\pm/P^\pm\hookto \Gl_{n,\calO}/G$ is a monomorphism of finite type (hence separated and quasi-finite, { by \cite[Tag 0463, 59.27.10]{StaPro}}). Thus, $Q^\pm/P^\pm$ is a scheme, and the map $i$ is quasi-affine by Zariski's main theorem. In particular, $Q^\pm/P^\pm$ is quasi-affine as well. Now there is a commutative diagram of $\calO$-spaces 
 \begin{equation}\label{commsquareBD}
\begin{tikzpicture}[baseline=(current  bounding  box.center)]
\matrix(a)[matrix of math nodes, 
row sep=1.5em, column sep=2em, 
text height=1.5ex, text depth=0.45ex] 
{\Gr_{P^\pm} & & \\
&(\Gr_{G})^\pm& \Gr_G  \\ 
\Gr_{Q^\pm}&(\Gr_{\Gl_{n,\calO}})^\pm& \Gr_{\Gl_{n,\calO}}, \\}; 
\path[->](a-1-1) edge[dotted] node[above] {}  (a-2-2);
\path[->](a-1-1) edge[bend left=15] node[above] {}  (a-2-3);
\path[->](a-1-1) edge node[above] {}  (a-3-1);
\path[->](a-3-1) edge node[above] {$\simeq$}  (a-3-2);
\path[->](a-2-2) edge node[above] {}  (a-2-3);
\path[->](a-2-2) edge node[left] {}  (a-3-2);
\path[->](a-3-2) edge node[below] {}  (a-3-3);
\path[->](a-2-3) edge node[right] {}  (a-3-3);
\end{tikzpicture}
\end{equation}
constructed as follows. The map $\Gr_G\to \Gr_{\Gl_{n,\calO}}$ is a quasi-compact immersion by Proposition \ref{Aff_Grass_Rep_Prop}, and as $\Gr_G$ is ind-proper, it is a closed immersion. Hence, the square is Cartesian by general properties of attractor (resp. repeller) ind-schemes. This also constructs the dotted arrow in \eqref{commsquareBD} which is the map $\iota^\pm$. Further, the map $\Gr_{Q^\pm}\to(\Gr_{\Gl_{n,\calO}})^\pm$ is an isomorphism by Lemma \ref{Linear_Attractor_Lem} below. The map $\Gr_{P^\pm}\to\Gr_{Q^\pm}$ is a quasi-compact immersion because $Q^\pm/P^\pm$ is quasi-affine. Since $(\Gr_{G})^\pm\to (\Gr_{\Gl_{n,\calO}})^\pm$ is a closed immersion, the map $\iota^\pm$ is a quasi-compact immersion.

\begin{lem} \label{Linear_Attractor_Lem}
If $G=\Gl_{n,\calO}$, then the maps $\iota^0$ and $\iota^\pm$ are isomorphisms.
\end{lem}
\begin{proof} \iffalse {\cg To make this even clearer to the reader, we might spell out what we mean by `compatible with the grading/filtration' -- it is clearest when the underlying bundle is trivial (ie when we are in the open cell)}\fi As in the proof of Lemma \ref{Loc_Lin_Lem}, we reduce to the case where $\chi$ is a standard dominant cocharacter. Then $\chi$ corresponds to a $\bbZ$-grading on $V:=\calO^n$, say $V=\oplus_{i\in\bbZ}V_i$, compatible with the standard $\calO$-basis of $V$. The group $M$ (resp. $P^+$/$P^-$) is a standard Levi (resp. standard parabolic) of automorphisms of $V$ preserving the grading (resp. the ascending/descending filtration induced from the grading). In the description of Lemma \ref{Aff_Grass_Rep_Lem}, the subfunctor $\Gr_M$ (resp. $\Gr_{P^\pm}$) are those vector bundles $\calE\in \Gr_G(R)$ compatible with the grading (resp. filtration induced by the grading) on $V\otimes_\calO \calO_{U_R}$. Likewise, the grading on $V$ induces in the notation of Lemma \ref{Aff_Grass_Rep_Lem} gradings on $\calE_{N,\calO}=V\otimes_\calO (\calI_{D}^{-N}/\calI_{D}^N)$ for each $N\geq 1$. As in Lemma \ref{Loc_Lin_Lem}, we have a closed $\bbG_m$-equivariant immersion, and hence the diagram of $\calO$-schemes 

\[
\begin{tikzpicture}[baseline=(current  bounding  box.center)]
\matrix(a)[matrix of math nodes, 
row sep=1.5em, column sep=2.5em, 
text height=1.5ex, text depth=0.45ex] 
{ \on{Quot}_N^0& \on{Grass}(\calE_{N,\calO})^0 \\ 
\on{Quot}_N& \on{Grass}(\calE_{N,\calO}), \\}; 
\path[->](a-1-1) edge node[above] {}  (a-1-2);
\path[->](a-2-1) edge node[below] {}  (a-2-2);
\path[->](a-1-1) edge node[left] {} (a-2-1);
\path[->](a-1-2) edge node[left] {} (a-2-2);
\end{tikzpicture}
\]
is cartesian, and likewise on attractor (resp. repeller) schemes. The equality $\on{Grass}(\calE_{N,\calO})^0=\prod_{i\in \bbZ} \on{Grass}(V_i\otimes_\calO (\calI_{D}^{-N}/\calI_{D}^N))$ is immediate, and one checks that $\on{Grass}(V\otimes_\calO (\calI_{D}^{-N}/\calI_{D}^N))^\pm$ is the subfunctor of those subspaces in $\calE_{N,\calO}$ compatible with the filtration. The lemma follows.\iffalse {\cb (Alternatively, use the Iwasawa decomposition argument of \cite[pf.\,of Prop.\,3.4]{HaRi}  to prove that $\iota^0(F)$ and $\iota^{\pm}(F)$ are bijections for every $\mathcal O$-field $F$, then use the argument below proving (\ref{groups:eq1}) and (\ref{groups:eq2}) are isomorphisms to prove $\iota^0$ and $\iota^{\pm}$ are isomorphisms in an open neighborhood of the base point, then use the $L{\rm GL}_n$-action as in \cite[pf.\,of Lem.3.6]{HaRi}; see {\em Second case} below.) {\cg I am not convinced yet...}}\fi
\end{proof}

\subsubsection{Proof of Theorem \ref{GCT_Geo_Thm}} We need a lemma first. By functoriality of the loop group construction, the $\bbG_m$-action on $G$ via $\chi$-conjugation gives an $\bbG_m$ on $LG$ (resp. $L^+G$; resp. $L^-G$). There are natural monomorphisms on negative loop groups
\begin{align}\label{groups:eq1}
 L^{-}M&\,\longto\, (L^{-}G)^0;\\
\label{groups:eq2}
 L^{-}P^\pm &\,\longto\, (L^{-}G)^\pm.
\end{align}

\begin{lem} \label{Iso_Neg_Loop}
The maps \eqref{groups:eq1} and \eqref{groups:eq2} are isomorphisms.
\end{lem}
\begin{proof}
Replacing $\calO$ by an \'etale cover, we may assume that there exists a closed embedding $G\hookto\Gl_{n,\calO}$. By the proof of Lemma \ref{Neg_Loop_Rep_Lem} (resp. Lemma \ref{Loop_Group_Rep_Lem} i)), the induced map $L^-G\to L^-\Gl_{n,\calO}$ is a closed immersion.

Let $\chi'\co \bbG_{m,\calO}\overset{\chi\,}{\to} G\to \Gl_{n,\calO}$, and denote the fixed point group (resp. attractor/repeller group) by $L$ (resp. $Q^\pm$). It is straight forward to check $L^-M=L^-G\cap L^-L$ (resp. $L^-P^\pm=L^-G\cap L^-Q^\pm$) and $(L^-G)^0=L^-G\cap (L^-\Gl_{n,\calO})^0$ (resp. $(L^-G)^\pm=L^-G\cap (L^-\Gl_{n,\calO})^\pm$). Hence, we may assume $G=\Gl_{n,\calO}$.

After passing to a Zariski cover of $\calO$, we may assume that $\chi$ is a standard dominant cocharacter, cf. proof of Lemma \ref{Loc_Lin_Lem}. We have for every $\calO$-algebra $R$,
\[
(L^-\Gl_{n,\calO})^0(R)\,=\,\{g\in G(\bbP^1_R\bslash D_R)\;|\; \forall S\in \on{(R-Alg)}, \la\in \bbG_m(S)\co\, \chi(\la)\cdot g\cdot \chi(\la)^{-1}=g\}.
\]   
Let $g\in (L^-\Gl_{n,\calO})^0(R)$. To show $g\in (L^-M)(R)$, we can take $S=R[t,t^{-1}]$ to see that the desired entries in the matrix $g$ vanish. The case of $(L^-G)^\pm$ is similar, and the lemma follows.
\end{proof}

\noindent{\it First case.} Let $\calO=F$ be a field. By fpqc-descent, we may assume that $F$ is algebraically closed. Then $D_\red=\sum_{i=1}^d[x_i]$ for pairwise distinct points $x_i\in X(F)$. If $d=1$, the maps $\iota^0$ and $\iota^\pm$ are isomorphisms in light of Example \ref{Special_Cases} i) and \cite[Prop.~3.4]{HaRi}. In general, by Corollary \ref{Fibers_Action_Map_Cor} each ind-scheme in \eqref{gctgeo_diag} is a direct product of $d$ copies (compatible with the maps) of classical affine Grassmannians formed using local parameters at $x_i$. The $\bbG_m$-action on the product via
\[
\bbG_m\,\subset\, L_D^+\bbG_m\,\simeq\,L^+_{[x_1]}\bbG_m\times_F\ldots\times_FL^+_{[x_n]}\bbG_m 
\]
is the diagonal action, and we conclude using Lemma \ref{Gm_product_lem} and the case $d=1$.\smallskip\\
{\it Second case.} Let $\calO$ be an Artinian local ring with maximal ideal $\frakm$, and residue field $F$. Passing to the strict Heselization, we may assume that $F$ is separably closed. The restriction of $\iota^0$ (resp. $\iota^\pm$) to the open cell $L^{--}M$ (resp. $L^{--}P^\pm$) is an isomorphism by Lemma \ref{Iso_Neg_Loop}. By Lemma \ref{Open_Cell_Lem}, there is the open subset 
\[
V_M\defined \bigcup_m\,m\cdot L^{--}M\cdot e_0\;\;\;\;\; \textup{(}\text{resp.}\,\; V_{P^\pm}\defined \bigcup_p\,p\cdot L^{--}P^\pm\cdot e_0\textup{)},
\]
of $\Gr_M$ (resp. $\Gr_{P^\pm}$), where the union runs over all $m\in LM(\calO)$ (resp. $p\in LP^\pm(\calO)$). The $LM$-equivariance (resp. $LP^\pm$-equivariance) of $\iota^0$ (resp. $\iota^\pm$) implies that $\iota^0|_{V_M}$ (resp. $\iota^\pm|_{V_{P^\pm}}$) is an isomorphism. As $\Gr_M$ (resp. $\Gr_{P^\pm}$) is a nilpotent thickening of $\Gr_M\otimes_\calO F$ (resp. $\Gr_{P^\pm}\otimes_\calO F$), it is enough to show that $V_M$ (resp. $V_{P^\pm}$) contains the special fiber. As $G$ splits over $F$ (because separably closed), the points $\Gr_M(F)\subset \Gr_M$ (resp. $\Gr_{P^\pm}(F)\subset \Gr_{P^\pm}$) are dense which follows from the density of $\bbA^n_F(F)\subset \bbA^n_F$ and the cellular structure of these spaces. Thus, it suffices to show that $\Gr_M(F)\subset V_M$ (resp. $\Gr_{P^\pm}(F)\subset V_{P^\pm}$). In view of Lemma \ref{BL_Lem} ii), it suffices to show that the reduction map $LM(\calO)\to LM(F)$ (resp. $LP^\pm(\calO)\to LP^\pm(F)$) is surjective. As $\calO$ is Artinian, the ring $\calO\rpot{D}$ is (semi-local) Artinian, and the reduction map $\calO\rpot{D}\to F\rpot{D}$ is surjective with nilpotent kernel $\frakm\rpot{D}$. Hence, the desired surjectivity follows from the formal lifting criterion using the smoothness of $M$ (resp. $P^\pm$). This handles the second case.\smallskip\\
{\it The general case.} Passing to an \'etale extension of $\calO$, we may assume that \eqref{gctgeo_diag} is a diagram of ind-schemes, cf.\,Corollary \ref{Aff_Grass_Rep_Cor}. In view of \eqref{Base_Change}, the closed immersion $\iota^0$ (resp. quasi-compact immersion $\iota^\pm$) is fiberwise bijective, and hence bijective. Now Theorem \ref{GCT_Geo_Thm} follows from Lemma \ref{Isom_Lem} below using the second case.

\begin{lem} \label{Isom_Lem}
Let $\calO$ be a Noetherian ring, and let $\iota\co Y\to Z$ be a quasi-compact immersion of finite type $\calO$-schemes. If $\iota$ is set-theoretically bijective, and if for every maximal ideal $\frakm\subset \calO$ and every $n\geq 1$, the reduction $\iota\otimes \calO/\frakm^n$ is an isomorphism, then $\iota$ is an isomorphism.
\end{lem}
\begin{proof} By \cite[Tag 01QV]{StaPro}, the map $\iota$ factors as an open immersion followed by a closed immersion: $Y\to\bar{Y}\to Z$. As $\iota$ is bijective, we have $Y=\bar{Y}$ and $\iota$ is a bijective closed immersion. Being an isomorphism is local on the target, and we may assume that $Z=\Spec(A)$ and hence $Y=\Spec(B)$ are affine. The map of $\calO$-algebras $\iota^{\#}\co A\to B$ is surjective (because closed immersion), and each element in $I:=\ker(\iota^{\#})$ is nilpotent (because $\iota^{\#}$ is bijective on spectra). It is enough to show that for the localization $I_{\frakm}=0$ for all maximal ideals $\frakm\subset \calO$. Without loss of generality, we may assume that $\calO$ is local with maximal ideal $\frakm$. If $\frakm A=A$, i.e., the fiber of $Z$ over $\frakm$ is empty, there is nothing to prove, and we may assume that $\frakm A\subset A$ is a proper ideal. As $A/\frakm^nA\to B/\frakm^nB$ is an isomorphism for all $n\geq 1$, we have $I\subset \cap_{n\geq 1}\frakm^nA$. But since $\frakm^nA=(\frakm A)^n$ and $\frakm A\subset A$ is a proper ideal in a Noetherian ring, we have $\cap_{n\geq 1}\frakm^nA=0$ by Krull's intersection theorem. The lemma follows. 
\end{proof}

 %%%%%%%%%%%%%%%%Recollection on local models for Weil-restricted groups%%%%%%%%%%%%%%%%%%%%%%%%
 \section{Local models for Weil-restricted groups}\label{Recollect_Weil_Restricted_Sec}
% \section{Recollection on local models for Weil-restricted groups}\label{Recollect_Weil_Restricted_Sec}
 In this section, we collect a few properties of the Weil-restricted affine Grassmannians as constructed in \cite{Lev16}. We provide proofs for several statements which appear to be well-known but for which we could not find proofs in the literature.

 \subsection{Notation}\label{notation_sec} Let $F/\bbQ_p$ be a finite extension with ring of integers $\calO_F$, and residue field $k$ with $q$ elements. 
 Let $K/F$ be a finite extension with ring of integers $\calO_K$ with residue field $k_0/k$.
Let $K_0/F$ denote the maximal unramified subextension of $K/F$ with the same residue field $k_0/k$.
 Fix a uniformizer $\varpi$ of $K$, and denote by $Q\in K_0[u]$ the minimal polynomial, i.e. $Q$ is the unique irreducible normalized polynomial with $Q(\varpi)=0$. Note that $Q\in \calO_{K_0}[u]$, and that $Q\equiv u^{[K:K_0]} \mod \varpi$.
 
Let $\breve{F}$ (resp.~$\breve{K}$; $\breve{K}_0$) denote the completion of the maximal unramified extension of $F$ (resp.~$K$; $K_0$) inside a fixed algebraic closure $\bar{F}$, and let $\sigma \in {\rm Aut}(\breve{F}/F)$ denote the  Frobenius generator.
We note that $\breve F=\breve K_0$.

 In \S \ref{Local_Models_Sec} below, we specialize the general set-up of \S \ref{Recollect_Loop_Group_Sec} to the case where $\calO = \calO_F$, $X = \bbA^1_{\calO_{K_0}}$ is viewed as a smooth curve over $\calO$, and $D$ is defined by $\{Q = 0\}$. 
 We first summarize some properties of parahoric groups for Weil-restricted groups (cf. \S\ref{Parahoric_Group_Sec}), and the group schemes $\underline{\calG}_0$ over $X=\bbA^1_{\calO_{K_0}}$ constructed in \cite{PZ13, Lev16} (cf. \S\ref{Group_Sec}).
 
 %Parahoric groups for Weil-restrictions
 \subsection{Parahoric Group Schemes for Weil-restricted groups}\label{Parahoric_Group_Sec}
 
 Let $G_0$ be a reductive $K$-group. 
 Fix a maximal $K$-split torus $A_0$, a maximal $\breve{K}$-split torus $S_0$ containing $A_0$ and defined over $K$. 
 Let $M_0=Z_{G_0}(A_0)$ denote the centralizer of $A_0$ which is a minimal $K$-Levi subgroup of $G_0$, and let $T_0=Z_{G_0}(S_0)$ be the centralizer of $S_0$. 
 Then $T_0$ is a maximal torus because $G_{0,\breve{K}}$ is quasi-split by Steinberg's theorem. 
 
We are interested in parahoric subgroups of the Weil restriction of scalars $G:=\Res_{K/F}(G_0)$. We will first need to classify the maximal $F$-split tori in $G$.

\begin{lem} \label{cochar_Res_lem} Suppose $T_0$ is any $K$-torus, so that $T = \Res_{K/F}(T_0)$ is an $F$-torus.  Then there is a canonical isomorphism of groups
\begin{equation} \label{cochar_Res_eq}
X_*(T)_{\Gamma_F} = X_*(T_0)_{\Gamma_K}.
\end{equation}
In particular, the $F$-split rank of $T$ is the $K$-split rank of $T_0$.
\end{lem}

\begin{proof}
Recall that $T$ represents the functor on $F$-tori which sends the $F$-torus $T'$ to
$$
{\rm Hom}_{K\text{-tori}}(T' \otimes_F K, T_0) = {\rm Hom}_{\Gamma_K\text{-Mod}}(X_*(T'), X_*(T_0)) = {\rm Hom}_{\Gamma_F\text{-Mod}}\big(X_*(T'), {\rm Ind}_{\Gamma_K}^{\Gamma_F}(X_*(T_0)\big).
$$
We deduce that $X_*(\Res_{K/F}(T_0)) = {\rm Ind}_{\Gamma_K}^{\Gamma_F}(X_*(T_0)) \cong X_*(T_0) \otimes_{\mathbb Z[\Gamma_K]} \mathbb Z[\Gamma_F]$ (since $[\Gamma_F: \Gamma_K] < \infty$).  Then the $H_0$-version of Shapiro's lemma gives $(X_*(T_0) \otimes_{\mathbb Z[\Gamma_K]} \mathbb Z[\Gamma_F])_{\Gamma_F}  = X_*(T_0)_{\Gamma_K}$, which implies the lemma. 
\end{proof}
Under the bijection
\begin{equation} \label{Res_univ_eq}
{\rm Hom}_{F}(T', \Res_{K/F}(G_0)) = {\rm Hom}_K(T'_{K}, G_0),
\end{equation}
$T' \to \Res_{K/F}(G_0)$ is injective if and only if the corresponding morphism $T'_K \rightarrow G_0$ is injective.  
Since any $K$-split torus is of the form $T'_K$ for a unique $F$-split torus $T'$, this shows that the rank of a maximal $F$-split torus in $\Res_{K/F}(G_0)$ is the same as the rank of a maximal $K$-split torus in $G_0$. 
For the maximal $K$-split torus $A_0 \subset G_0$, we write $A_0 = A_{K}$ for a unique $F$-split torus $A$. 
Using the canonical embedding $A \hookrightarrow \Res_{K/F} (A_{K}) = \Res_{K/F}(A_0)$, we see that $A$ is the $F$-split component of $\Res_{K/F}(A_0)$ and also a maximal $F$-split torus in $G$. 

From now on, we will abuse notation and denote by $A$ the image of $A \hookrightarrow \Res_{K/F}(A_0) \hookrightarrow \Res_{K/F}(G_0)=G$ (even though $A$ is not a Weil restriction of a torus). 
The discussion following (\ref{Res_univ_eq}) shows that $A_0 \mapsto A$ gives a bijection between maximal $K$-split tori in $G_0$ and maximal $F$-split tori in $G$.

Let us note that since $S_0$ is $\breve{K}$-split (and using $\breve{K}_0 \otimes_{K_0} K = \breve{K}$) there exists a unique subtorus $S_0'\hookto \Res_{K/K_0}(S_0)$ which is a maximal $\breve{K}_0$-split torus in $\Res_{K/K_0}(G_0)$ defined over $K_0$ and of the same rank as $S_0$.
We let $S$ denote the image of $\Res_{K_0/F}(S_0')\hookto \Res_{K_0/F}(\Res_{K/K_0}(S_0))=\Res_{K/F}S_0$ which is a maximal $\breve{F}$-split torus in $G$ defined over $F$.

\iffalse
there exists a unique diagonal $F$-torus of the same rank
\begin{equation}\label{Diag_Torus}
\tilde{A}\,\hookto\, \Res_{K/F}(A) \;\;\;\;\; \on{(resp.}\, \tilde{S} \,\hookto\,\Res_{K/F}(S)\,\text{)},
\end{equation}
which is a maximal $F$-split torus (resp. maximal $\bF$-split torus defined over $F$) inside $\tilde{G}$. Let $\tilde{M}=Z_{\tilde{G}}(\tilde{A})$ (resp. $\tilde{T}=Z_{\tilde{G}}(\tilde{S})$) denote the centralizer which is a minimal $F$-Levi subgroup (resp. maximal torus) in $\tilde{G}$. 
\fi

\begin{lem}\label{Torus_Weil_Lem}
Letting $M=\Res_{K/F}(M_0)$ and $T=\Res_{K/F}(T_0)$, we have $M = Z_{G}(A)$ and $T = Z_{G}(S)$ as subgroups of $G=\Res_{K/F}(G_0)$. 
\end{lem}
\begin{proof} 
Both containments `$\subseteq$' are obvious.
%{\cm For `$\supseteq$' we note that the torus $A_\sF$ (resp. $S_\sF$) is (resp.~contains) the diagonal torus inside $\prod_{K\hookto \sF}A_0\otimes_{K,\psi}\sF$ (resp.\,$\prod_{K\hookto \sF}S_0\otimes_{K,\psi}\sF$). By considering their centralizers inside $\prod_{K\hookto \sF}G_0\otimes_{K,\psi}\sF$, the lemma is obvious.}
For `$\supseteq$' we note that the torus $A_\sF$ is the diagonal torus inside $\prod_{K\hookto \sF}A_0\otimes_{K,\psi}\sF$. Considering its centralizer inside $\prod_{K\hookto \sF}G_0\otimes_{K,\psi}\sF$ proves $M = Z_G(A)$. Since $Z_G(S)$ is necessarily a maximal torus, the inclusion $T \subseteq Z_G(S)$ is also an equality.
\end{proof}

The correspondence $A \leftrightarrow A_0$ induces a correspondence between the apartments in the (extended) Bruhat-Tits buildings $\scrB(G,F)$ and $\scrB(G_0, K)$. We will show that there is a canonical isomorphism
\begin{equation}\label{Building_Iso}
\scrB(G,F) \,\simeq\, \scrB(G_0,K),
\end{equation}
equivariant for the action of $G(F)=G_0(K)$, and compatible with an identification of apartments $\scrA(G,A,F)=\scrA(G_0,A_0,K)$.

The \emph{Iwahori-Weyl group} $W=W(G,A,F)$ is the group 
\begin{equation}\label{Iwahori_Weyl}
W\defined \on{Norm}_{G}(A)(F)/M(F)_1,
\end{equation}
where $M(F)_1$ is the unique parahoric subgroup of the minimal Levi $M$, cf.~ \cite{HR08, Ri16a}. 
(By Lemma \ref{Torus_Weil_Lem}, $M$ is a minimal $F$-Levi subgroup of $G$.)  
We define $\breve{W}= W(G, S, \breve{F})$ analogously.
In the following we will use the identification $\breve F\otimes_FK=\prod_{[K_0:F]}\breve K$ which is $\sig$-equivariant for the action $\sig(a_j)_j=(\sig a_{j-1})_j$ on the product.

\begin{lem}\label{IW_lem}
There is a canonical identification of Iwahori-Weyl groups 
\[
W(G,A,F)=W(G_0,A_0,K) \hspace{.25in} \mbox{and} \hspace{.25in} W(G, S, \breve{F}) = W(G_0, S_0, \breve{F}\otimes_FK)=\prod_{[K_0:F]}W(G_0,S_0,\breve K).
\]
\end{lem}
\begin{proof}
As in Lemma \ref{Torus_Weil_Lem}, one shows $\on{Norm}_{G}(A)=\Res_{K/F}(\on{Norm}_{G_0}(A_0))$, and hence $\on{Norm}_{G}(A)(F)=\on{Norm}_{G_0}(A_0)(K)$. 
By Lemma \ref{Kottwitz_ker_lem} below, $M(F)_1 = M_0(K)_1$.  
The first equality follows and the second is similar.
\iffalse
Let $\calM$ (resp. $\tilde{\calM}$) denote the unique parahoric $\calO_K$-group scheme of $M$ (resp. parahoric $\calO_F$-group scheme of $\tilde{M}$).
By Proposition \ref{Parahoric_Weil_Prop} (resp. Corollary \ref{Parahoric_Cor}), we have $\tilde{\calM}=\Res_{\calO_K/\calO_F}(\calM)$, and hence $\tilde{M}_1:=\tilde{\calM}(\calO_F)=\calM(\calO_K)=:M_1$. Further, $\calG_{\tilde{\bbf}}(\calO_F)=\calG_\bbf(\calO_K)$ which implies the corollary.
\fi
\end{proof}

\begin{lem} \label{Kottwitz_ker_lem}
Let $G(F)_1 \subset G(F)$ denote the Kottwitz kernel, i.e.,\,$ G(F)_1 = G(F) \cap G(\breve{F})_1$ with 
$$
G(\breve{F})_1 = {\rm ker}\big(\kappa_{G}\co G(\breve{F}) \rightarrow X^*(Z_{G^\vee}^{I_F})\big),
$$
where $\kappa_{G}$ is the Kottwitz homomorphism of \cite[$\S7$]{Ko97}.
Then $G(\breve{F})_1 = \prod_{[K_0:F]}G_0(\breve{K})_1$ and $G(F)_1 = G_0(K)_1$.
\end{lem}

\begin{proof}
The result is clear when $G_0$ is 
%{\cg Timo: the following seems confusing for induced tori -- we should just do it all at once for tori (the same comment could have been made in our earlier version)} 
a torus: $G(\breve{F})_1$ and $\prod_{[K_0:F]}G_0(\breve{K})_1$ coincide with the unique maximal bounded subgroup of 
\[
G(\breve{F}) = G_0(\breve{F}\otimes_FK)=\prod_{[K_0:F]}G_0(\breve K).
\]
Thus, by a variation of Lemma \ref{cochar_Res_lem}, $\kappa_{G}\co G(\breve{F}) \rightarrow X_*(G)_{I_F}$ is the direct product over $[K_0:K]$-many copies of $\kappa_{G_0}\co G_0(\breve{K}) \to X_*(G_0)_{I_K}$. 
%{\cm If $G_0$ is any torus, then taking a presentation by induced tori as in the construction of $\kappa_{G_0}$ (cf.~ \cite[$\S7.2$]{Ko97}), the same assertion holds for $G_0$.} 
Clearly the result holds for $G_0 = G_{0, \rm sc}$ and hence for $G_{0, \rm der} = G_{0, \rm sc}$ by reduction to the torus case. 
Finally the general case follows by the method of $z$-extensions as in the construction of $\kappa_{G_0}$ (\cite[$\S7.4$]{Ko97}).
\end{proof}

\begin{lem} \label{apt_lem}
There is a canonical isomorphism of apartments $\scrA(G,S,\breve{F})=\scrA(G_0,S_0,\breve{F}\otimes_FK)$ compatible with the action of the Iwahori-Weyl groups $W(G, S, \breve{F}) = W(G_0, S_0, \breve{F}\otimes_FK)$ and the action of the Frobenius $\sig$. 
%{\cm of a geometric Frobenius element $\Phi \in \Gamma_F$}{\cg Tom: This was fixed in the notation above.}
\end{lem}

\begin{proof}
Let $\breve{\Sigma}_{G}$ (resp.\,$\breve{\Sigma}_{G_0}$) denote the Bruhat-Tits \'{e}chelonnage root system attached to $(G, S)$ (resp.\,$(G_0,S_0)$). Taking $T_0 = T_{0,\rm sc}$ in Lemma \ref{cochar_Res_lem} and using \cite[Lem.\,15]{HR08}, we obtain an equality of coroot lattices
$$
Q^\vee(\breve{\Sigma}_{G}) = X_*(T_{\rm sc})_{I_F} =  \prod_{[K_0:F]}X_*(T_{0,\rm sc})_{I_K} = \prod_{[K_0:F]}Q^\vee(\breve{\Sigma}_{G_0}).
$$
By considering minimal positive generators of these lattices, we deduce that $\breve{\Sigma}_{G} = \prod_{[K_0:F]}\breve{\Sigma}_{G_0}$. 
As all identifications are canonical this isomorphism is compatible with the action of $\sig$ on both sides. 
%{\cm, noting $\Phi$ is a common geometric Frobenius element in $\Gamma_F$ and in $\Gamma_K$.} 
This gives the identification of affine root hyperplanes needed to prove the isomorphism of apartments
$$
\scrA(G, S,\breve{F}) = \scrA(G_0, S_0, \breve{F}\otimes_FK).
$$
The isomorphism is equivariant for $W(G, S, \breve{F}) = W(G_0, S_0, \breve{F}\otimes_FK)$ and $\sig$.
\end{proof}

\begin{prop} \label{building_prop}
There is a canonical isomorphism $\scrB(G,F) \,\simeq\, \scrB(G_0,K)$,
equivariant for the action of $G(F)=G_0(K)$, and compatible with an identification of apartments $\scrA(G,A,F)=\scrA(G_0,A_0,K)$. 
\end{prop}

\begin{proof}
By construction $\scrB(G_0,\breve{K}) = (G_0(\breve{K}) \times \scrA(G_0,S_0,\breve{K}))/ \sim$, where $(g, x) \sim (g', x')$ if there exists $n \in {\rm Norm}_{G_0}(S_0)(\breve{K})$ such that $n \cdot x = x'$ and $g^{-1}g' n \in U_x$.  Here $U_x$ is the subgroup of $G_0(\breve{K})$ generated by the affine root groups $U_{\alpha + r}$ associated to $\alpha + r$ with $\alpha(x) + r \geq 0$, for $(\alpha, r) \in \breve{\Sigma}_{G_0} \times \mathbb Z$.  
Because $\breve{\Sigma}_{G} = \prod_{[K_0:F]}\breve{\Sigma}_{G_0}$, the equivalence relation for $G$ is the $[K_0:F]$-fold product of the equivalence relation for $G_0$.
%{\cm $U_x$ is the same for $\tilde{G}$ and $G$, and so the equivalence relation is the same for $\tilde{G}$ and $G$. } 
Using Lemma \ref{apt_lem}, this proves $\scrB(G, \breve{F}) =\prod_{[K_0:F]}\scrB(G_0,\breve K)= \scrB(G, \breve{F}\otimes_FK)$, equivariantly for $\sig$, and the proposition follows by \'{e}tale descent, cf.~\cite[\S 5.1]{BT84}.
\end{proof}

Let $\bbf$ be a facet of $\scrA(G,A,F)$, and denote by $\bbf_0$ the corresponding facet in $\scrA(G_0,A_0,K)$. 
Let $\calG_{\bbf}$ (resp. $\calG_{\bbf_0}$) be the associated parahoric group scheme over $\calO_F$ (resp.\,over $\calO_K$). 

\begin{prop}\label{Parahoric_Weil_Prop}
There is a canonical isomorphism of $\calO_F$-group schemes $\calG_{\bbf}\simeq\Res_{\calO_K/\calO_F}(\calG_{\bbf_0})$ inducing the identity on generic fibers.
\end{prop}
\begin{proof} By the defining property of parahoric group schemes, it suffices to check that the group $\calH:=\Res_{\calO_K/\calO_F}(\calG_{\bbf_0})$ is a smooth affine $\calO_F$-group scheme of finite type with (geometrically) connected special fiber, with the property that $\calH(\calO_{\breve{F}})$ is the intersection of the Kottwitz kernel $G(\breve{F})_1$ with the pointwise fixer in $G(\breve{F})$ of $\bbf$ which we view as a subset of the building over $\breve{F}$; cf.\,\cite{HR08}.
The $\calO_F$-group $\calH$ is smooth affine and of finite type by general properties of Weil restriction of scalars, cf.~ \cite[\S7.6, Thm.~4, Prop.~5]{BLR90}. 
If $R=\calO_K/\varpi^{[K:K_0]}$, then the special fiber is given by 
\[
\calH\otimes_{\calO_F}k\,=\,\Res_{R/k}(\calG_{\bbf_0}\otimes_{\calO_K}R),
 \]
which is a successive extension of smooth (geometrically) connected groups, and hence (geometrically) connected. 
As $\prod_{[K_0:F]}\breve{K} = K\otimes_F\breve{F}$ we have $\calH(\calO_{\breve{F}})=\prod_{[K_0:F]}\calG_{\bbf_0}(\calO_{\breve{K}})$ 
%{\cm But $\calG_{\bbf}(\calO_{\breve{K}})$} 
which is the intersection of $\prod_{[K_0:F]}G_0(\breve{K})_1 = G(\breve{F})_1$ (Lemma \ref{Kottwitz_ker_lem}) with the pointwise fixer in $\prod_{[K_0:F]}G_0(\breve{K})$ of $\bbf$, by %{\cm Proposition \ref{building_prop} applied over the field extension $\breve{K}/\breve{F}$} 
Lemma \ref{apt_lem}. 
The proposition follows.
\end{proof}

\begin{cor}\label{Parahoric_Cor} 
Every parahoric $\calO_F$-group scheme of $G$ is of the form $\Res_{\calO_K/\calO_F}(\calG_{\bbf_0})$ for a unique facet $\bbf_0\subset \scrB(G,K)$.
\end{cor}
\hfill\ensuremath{\Box} 

The subgroup $W_{\bbf}=W_{\bbf}(G,A,F)$ of $W$ associated with $\bbf$ is the group
\[
W_{\bbf}\defined \big(\on{Norm}_{G}(A)(F)\cap \calG_{\bbf}(\calO_F)\big)/M(F)_1.
\]
The isomorphism $W(G,A,F)=W(G_0,A_0,K)$ induces $W_{\bbf}(G,A,F)=W_{\bbf_0}(G_0,A_0,K)$.
\iffalse
\begin{cor}\label{IW_f_Cor}
There is a canonical identification of Iwahori-Weyl groups 
\[
W(\tilde{G},\tilde{A},F)=W(G,A,K),
\]
compatible with $W_{\tilde{\bbf}}(\tilde{G},\tilde{A},F)=W_{\bbf}(G,A,K)$.
\end{cor}
\begin{proof}
As in Lemma \ref{Torus_Weil_Lem}, one shows $\on{Norm}_{\tilde{G}}(\tilde{A})=\Res_{K/F}(\on{Norm}_{G}(A))$, and hence $\on{Norm}_{\tilde{G}}(\tilde{A})(F)=\on{Norm}_{G}(A)(K)$. Let $\calM$ (resp. $\tilde{\calM}$) denote the unique parahoric $\calO_K$-group scheme of $M$ (resp. parahoric $\calO_F$-group scheme of $\tilde{M}$). By Proposition \ref{Parahoric_Weil_Prop} (resp. Corollary \ref{Parahoric_Cor}), we have $\tilde{\calM}=\Res_{\calO_K/\calO_F}(\calM)$, and hence $\tilde{M}_1:=\tilde{\calM}(\calO_F)=\calM(\calO_K)=:M_1$. Further, $\calG_{\tilde{\bbf}}(\calO_F)=\calG_\bbf(\calO_K)$ which implies the corollary.
\end{proof}

\fi
Let us point out a consequence of Proposition \ref{building_prop}  which is used later.

\begin{cor}\label{Center_Identify}
There is a canonical identification $\calZ(G(F),\calG_{\bbf}(\calO_F))=\calZ(G_0(K),\calG_{\bbf_0}(\calO_{K}))$ of centers of parahoric Hecke algebras compatible with the Bernstein isomorphism of \cite[Thm.~11.10.1]{Hai14}, where the Haar measures are normalized to give $\calG_{\bbf}(\calO_F)=\calG_{\bbf_0}(\calO_{K})$ volume $1$.
\end{cor}
\begin{proof}
In view of $\calG_{\bbf}(\calO_F)=\calG_{\bbf_0}(\calO_K)$, the equality of the centers is clear, and it remains to show the compatibility with the Bernstein isomorphism. 
This follows from the equality
\[
\La_{M}:= M(F)/M(F)_1=M_0(K)/M_0(K)_1=:\La_{M_0},
\] 
combined with the definition of Bernstein isomorphisms given by the integration formula (e.g.\,\cite[11.11]{Hai14}) and the isomorphism of finite relative Weyl groups $W_0(G, A, F) = W_0(G_0, A_0, K)$ consistent with Lemma \ref{IW_lem}.
\end{proof}

 %Group schemes
 \subsection{Group schemes over $\bbA^1_{\calO_{K_0}}$} \label{Group_Sec}
Let $G_0$ be a reductive $K$-group which splits over a tamely ramified extension, and let $G:=\Res_{K/F}(G_0)$.
Fix a chain of subgroups $A_0\subset S_0\subset T_0\subset M_0$ in $G_0$ as in \S\ref{Parahoric_Group_Sec} with corresponding chain of subgroups $A\subset S\subset T\subset M$ in $G$.
Further, fix a parabolic $K$-subgroup $P_0$ containing $M_0$ in $G_0$, and let $P:=\Res_{K/F}(P_0)$ in $G$. 
 
 In \cite[\S 3]{PZ13}, a reductive $\calO_K[u^\pm]$-group scheme $\uG_0$ admitting a maximal torus, and with connected fibers is constructed.
 As observed in \cite[\S3.1; Prop.~3.3]{Lev16}, the group scheme $\uG_0$ is defined over $\calO_{K_0}[u^\pm]$ in the following sense.

\begin{prop} \label{extension_prop}
i\textup{)} There exists a reductive $\calO_{K_0}[u^\pm]$-group $\uG_0$ together with a tuple of smooth closed $\calO_{K_0}[u^\pm]$-subgroups $(\uA_0,\uS_0,\uT_0,\uM_0,\uP_0)$ and an isomorphism of $K$-groups
\[
(\uG_0, \uA_0,\uS_0,\uT_0,\uM_0,\uP_0)\otimes_{\calO_{K_0}[u^\pm], u\mapsto\varpi}K\;\simeq\; (G_0,A_0,S_0,T_0,M_0,P_0),
\]
where $\uA_0$ is a maximal $\calO_{K_0}[u^\pm]$-split torus, $\uS_0$ a maximal $\calO_{\breve K_0}[u^\pm]$-split torus defined over $\calO_{K_0}[u^\pm]$, $\uT_0$ its centralizer, $\uM_0$ the centralizer of $\uA_0$ \textup{(}a minimal Levi\textup{)}, and $\uP_0$ a parabolic $\calO_{K_0}[u^\pm]$-subgroup with Levi $\uM_0$.\smallskip\\
ii\textup{)} The base change $\uG_{\calO_{\breve K_0}[u^\pm]}$ is quasi-split. In particular, $\uT_0$ is a maximal torus.
\end{prop}
\begin{proof} 
%The result in \cite[Prop.~3.3]{Lev16} is slightly more general where $K/F$ is not assumed to be totally ramified, but we do not need this more general version in the manuscript. 
Let us recall some elements of the construction as needed later. 
Let $\tilde{K}/K$ be a tamely ramified Galois extension which splits $G_0$. After possibly enlarging $\tilde{K}$, we may assume: 
\begin{enumerate}
\item[1)] the group $G_0$ is quasi-split over the maximal unramified subextension $\tilde{K}_0$ of $\tilde{K}/F$; 
\item[2)] there is a uniformizer $\tilde{\varpi} \in \tilde{K}$ and an integer $\tilde{e} \geq 1$ such that $\varpi = \tilde{\varpi}^{\tilde{e}}$, and therefore $\tilde{K} \overset{\sim}{\leftarrow} \tilde{K}_0[v]/Q(v^{\tilde{e}})$ via $\tilde{\varpi} \mapsfrom v$; 
\item[3)] $\tilde{K}_0$ contains a primitive $\tilde{e}$-th root of unity, cf.~ \cite[\S3.1]{PZ13}.
\end{enumerate}

\iffalse
write $\tilde{K}=\tilde{K}_0[v]/(v^{\tilde{e}}-\varpi), \tilde{\varpi}\mapsto v$ with $\tilde{\varpi}\in \tilde{K}$ a uniformizer, cf. \cite[\S3.1]{PZ13}. 
\fi

There is a cocartesian diagram\footnote{This differs from \cite[\S3.1]{Lev16} in that Levin uses instead of $\tilde{K}_0$ the maximal unramified  subxtension of $\tilde{K}/K$; this seems to be a mistake, e.g.,\,the diagram corresponding to (\ref{ring:diag}) is not cocartesian.} of $\calO_{K_0}$-algebras
\begin{equation}\label{ring:diag}
\begin{tikzpicture}[baseline=(current  bounding  box.center)]
\matrix(a)[matrix of math nodes, 
row sep=1.5em, column sep=4em, 
text height=1.5ex, text depth=0.45ex] 
{\calO_{\tilde{K}_0}[v] & \tilde{K} \\ 
\calO_{K_0}[u]& K \\}; 
\path[->](a-1-1) edge node[above] {$v\mapsto \tilde{\varpi}$}  (a-1-2);
\path[->](a-2-1) edge node[below] {$u\mapsto \varpi$}  (a-2-2);
\path[->](a-2-1) edge node[left] {$u\mapsto v^{\tilde{e}}$} (a-1-1);
\path[->](a-2-2) edge node[left] {} (a-1-2);
\end{tikzpicture}
\end{equation}
One can prove that $\mathcal O_{\tilde{K}_0}[v]/\mathcal O_{K_0}[u]$ is a ramified Galois cover with group isomorphic to $\tilde{\Gamma} :=\Gal(\tilde{K}/K)$; for this we use that $\tilde{K}_0$ contains a primitive $\tilde{e}$-th root of unity. 
As in \cite[\S 3]{PZ13}, the $\calO_{K_0}[u^\pm]$-group scheme $\uG_0$ is constructed in \cite[\S3.1]{Lev16} by descending a suitable choice of Chevalley model for $G_{0,\tilde{K}}$ along the \'etale ring extension $\calO_{\tilde{K}_0}[v^\pm]/\calO_{K_0}[u^\pm]$, cf. \cite[\S 3]{PZ13} and \cite[\S3.1]{Lev16} for details. See also Example \ref{torus_eg}.
\end{proof}
 
Let us denote
\begin{equation}\label{analogues0}
(G^\flat_0,A^\flat_0,S^\flat_0,T^\flat_0,M^\flat_0,P^\flat_0)\defined (\uG_0,\uA_0,\uS_0,\uT_0,\uM_0,\uP_0)\otimes_{\calO_{K_0}[u^{\pm}]}k_0\rpot{u}.
\end{equation}
Then $G^\flat_0$ is a reductive group over $K^\flat_0 := k_0\rpot{u}$, and $(A^\flat_0,S^\flat_0,T^\flat_0,M^\flat_0,P^\flat_0)$ are analogous to the corresponding groups above, cf.\,the discussion in \cite[4.1.2; 4.1.3]{PZ13}, \cite[3.3]{Lev16}. 
Further, we obtain a canonical identification of the apartments 
\begin{equation}\label{apartments}
\scrA(G,A,F)=\scrA(G_0,A_0,K) =\scrA(\uG_0,\uA_0,\kappa\rpot{u}),
\end{equation}
for both $\kappa=K_0,k_0$, cf.\,\cite[4.1.3]{PZ13}, \cite[Prop.~3.3.1 ff.]{Lev16}. 
In particular we have $\scrA(G_0,A_0,K_0)= \scrA(G^\flat_0,A^\flat_0,K^\flat_0)$ for $\kappa=k_0$.
Thus, we may think about $G_0^\flat$ as an equal characteristic analogue of $G_0$ of the same Dynkin type.

We now introduce the equal characteristic analogue $G^\flat$ of $G$ by restriction of scalars along the unramified extension $K_0/F$:
we define the sextuple of $k\rpot{u}$-groups
\begin{equation}\label{analogues}
(G^\flat,A^\flat,S^\flat,T^\flat,M^\flat,P^\flat),
\end{equation}
where $G^\flat=\Res_{K^\flat_0/F^\flat}(G_0^\flat)$ is a reductive group over $F^\flat := k\rpot{u}$, and likewise $(S^\flat,T^\flat,M^\flat,P^\flat)$ are obtained from $(S^\flat_0,T^\flat_0,M^\flat_0,P^\flat_0)$ by restriction of scalars along the unramified extension $K^\flat_0/F^\flat$. 
Here $A^\flat$ is the maximal $F^\flat$-split subtorus inside $\Res_{K^\flat_0/F^\flat}(A_0^\flat)$.

Combining \eqref{apartments} with Proposition \ref{building_prop}, we obtain a canonical identification of the apartments 
\begin{equation}\label{apartments_general}
\scrA(G,A,F) = \scrA(G^\flat,A^\flat,F^\flat).
\end{equation}
%for both $\kappa=F,k$. In particular we have $\scrA(G,A,F)= \scrA(G^\flat,A^\flat,F^\flat)$ for $\kappa=k$.

We shall use the following two results in \S \ref{Test_Functions_Sec} below.

\begin{lem} \label{Iwahorilem}
There is an identification of Iwahori-Weyl groups $W(G,A,F)= W(G^\flat,A^\flat,F^\flat)$ which is compatible with the action on the apartments under the identification \eqref{apartments_general}.
\end{lem}
\begin{proof} Over $\bF$ we obtain a $\sig$-equivariant isomorphism according to \cite[4.1.2]{PZ13}, \cite[3.3.0.1]{Lev16} compatible with the action on the apartments. The general case follows by taking $\sig$-fixed points from \cite[\S 1.2]{Ri16a} (cf.\,also \cite[4.1.3]{PZ13}, \cite[Prop.~3.3.1 ii)]{Lev16}).
\end{proof}

Now let $\calG=\calG_{\bbf}$ be a parahoric $\calO_F$-group scheme of $G$ whose facet $\bbf$ is contained in $\scrA(G,A,F)$. 
Then under \eqref{apartments_general} we obtain a unique facet $\bbf^\flat\subset \scrA(G^\flat,A^\flat,F^\flat)$, and hence a parahoric $\calO_{F^\flat}$-group scheme $\calG^\flat:=\calG_{\bbf^\flat}$ of $G^\flat$.

\begin{lem} \label{hecke_identification}
There is a canonical identification $\calZ(G(F),\calG(\calO_F))=\calZ(G^\flat(F^\flat),\calG^\flat(\calO_{F^\flat}))$ of centers of parahoric Hecke algebras, where the Haar measures are normalized to give $\calG(\calO_F)$ \textup{(}resp. $\calG^\flat(\calO_{F^\flat})$\textup{)} volume $1$.
\end{lem}
\begin{proof} Applying Lemma \ref{Iwahorilem} for $M$, we obtain an identification of abelian groups
\begin{equation}\label{Lattice_Identify}
\La_M:= M(F)/M(F)_1=M^\flat(F^\flat)/M^\flat(F^\flat)_1=:\La_{M^\flat},
\end{equation}
where $M(F)_1$ (resp. $M^\flat(F^\flat)_1$) is the unique parahoric group scheme of $M(F)$ (resp. $M^\flat(F^\flat)$).  The result follows via the Bernstein isomorphisms \cite[Thm.~11.10.1]{Hai14}
$$
\calZ(G^\flat(F^\flat),\calG^\flat(\calO_{F^\flat})) \simeq\bar{\mathbb Q}_\ell[\Lambda_{M^\flat}]^{W_0(G^\flat,A^\flat,F^\flat)} = \bar{\mathbb Q}_\ell[\Lambda_M]^{W_0(G,A,F)} \simeq \calZ(G(F),\calG(\calO_F)),
$$
noting that the finite relative Weyl groups of $(G,A,F)$ and $(G^\flat,A^\flat,F^\flat)$ are isomorphic compatible with the action on $\La_M=\La_{M^\flat}$.
%, and that $k_K=k_F$ because $K/F$ is totally ramified. 
\end{proof}

\iffalse
{\cg Timo:There is a group $\ucG$ over $\calO_F[u]$, defined to be the Weil restriction of $\ucG_0$ along $\calO_F[u] \to \calO_{K_0}[u]$, which has all of the corresponding good properties below.  Do we want to mention this?  I need to think whether it will be needed in the latter sections.}
{\cy Tom: Maybe we can mention this. However, I am not sure whether this Weil restricted group is of any use. It seems to give the wrong local models.}
\fi
 
\begin{thm} \label{parahoric_group_thm} 
 Fix $(\uG_0,\uA_0,\uS_0,\uT_0)$ and $\calG_\bbf$ with $\bbf\subset\scrA(G,A,F)$ as above. There exists a tuple of smooth affine $\calO_{K_0}[u]$-group schemes $(\ucG_0,\ucA_0,\ucS_0,\ucT_0)$ with geometrically connected fibers satisfying the following properties:\smallskip\\
i\textup{)} The restriction $(\ucG_0,\ucA_0,\ucS_0,\ucT_0)|_{\calO_{K_0}[u^\pm]}$ is $(\uG_0,\uA_0,\uS_0,\uT_0)$ as $\calO_{K_0}[u^\pm]$-groups. \smallskip\\
ii\textup{)} The base change of $\ucG_0$ under $\calO_{K_0}[u]\to \calO_K$, $u\mapsto \varpi$ is the parahoric group $\calG=\calG_\bbf$.\smallskip\\
iii\textup{)} The base change of $\ucG_0$ under $\calO_{K_0}[u]\to \kappa\pot{u}$, $u\mapsto u$ for both $\kappa=K_0,k_0$ is the parahoric group scheme for $\uG_{0,\kappa\rpot{u}}$ attached to $\bbf$ under \eqref{apartments}.\smallskip\\
iv\textup{)} The group $\ucA_0$ is a split $\calO_{K_0}[u]$-torus, $\ucS$ a $\calO_{K_0}[u]$-torus which splits over $\calO_{\bF}[u]$ and $\ucT$ is a smooth affine $\calO_{K_0}[u]$-group scheme such that $\ucT_0\otimes \kappa\pot{u}$ is the neutral component of the lft N\'eron model of $\uT_{0,\kappa\rpot{u}}$, for $\kappa=K_0,k_0$.\smallskip\\
The group $\ucG_0$ is uniquely determined \textup{(}up to unique isomorphism\textup{)} by properties i\textup{)} and iii\textup{)} for $\kappa=K_0$, and so is the tuple $(\ucA_0,\ucS_0,\ucT_0)$ using iv\textup{)}.
\end{thm}
 \begin{proof} 
 This is \cite[Thm.~3.3.3, Prop.~3.3.4]{Lev16}, cf. also \cite[Thm.~4.1]{PZ13}, esp.\,4.2.1, for the uniqueness assertion.
 \end{proof}

\begin{ex} \label{torus_eg} 
Suppose $G_0 = T_0$ is a tamely ramified torus over $K$.  Let $T_H$ be the split torus over $\mathcal O_{K_0}$ such that $T_0$ is given by a 1-cocycle 
$$
[\tau] \in H^1(\tilde{\Gamma}, {\rm Aut}(T_H \otimes_{\mathcal O_{K_0}} \tilde{K})).
$$
Explicitly,
$$
T_0 = \Big({\rm Res}_{\tilde{K}/K} (T_H \otimes_{\mathcal O_{K_0}} \tilde{K})\Big)^{\tilde{\Gamma}}.
$$

We let $T_H \otimes_{\mathcal O_{K_0}} \tilde{\mathcal O}_{0}[v]$ be the split torus over $\tilde{\mathcal O}_0[v] :=  \mathcal O_{ \tilde K_0}[v]$ (cf.~\eqref{ring:diag}), which is endowed with Galois actions $\tau(\gamma) \otimes \gamma$ for $\gamma \in \tilde{\Gamma}$ which we view as Galois descent data used to give a torus over $\mathcal O_{K_0}[u]$. Explicitly, we define $\underline{T}_0/\mathcal O_{K_0}[u^\pm]$ and $\underline{\mathcal T}_0/ \mathcal O_{K_0}[u]$ by
$$
\underline{T}_0 = \Big({\rm Res}_{\tilde{\mathcal O}_0[v^\pm]/\mathcal O_{K_0}[u^\pm]} (T_H \otimes_{\mathcal O_{K_0}} \tilde{\mathcal O}_0[v^\pm])\Big)^{\tilde{\Gamma}}.
$$
and $\underline{\mathcal T}_0$ as the (fiberwise) neutral component of
$$
\Big({\rm Res}_{\tilde{\mathcal O}_0[v]/\mathcal O_{K_0}[u]} (T_H \otimes_{\mathcal O_{K_0}} \tilde{\mathcal O}_0[v])\Big)^{\tilde{\Gamma}}.
$$
Write ${\rm Gal}(\tilde{K}/K) = \langle \gamma \rangle \rtimes \langle \sigma \rangle$ where $\gamma$ generates the inertia subgroup and $\sigma$ is a lift of a generator of ${\rm Gal}(\tilde{K}^{\rm un}/K)$ for $\tilde{K}^{\rm un}/K$ the maximal unramified subextension of $\tilde{K}/K$.  Then $\uT_0$ is realized as a ${\rm Gal}(\breve{K}_0/K_0)$-descent of $$\big({\rm Res}_{\calO_{\breve{K}_0}[v^\pm]/\calO_{\breve{K}_0}[u^\pm]} \, (T_H \otimes_{\calO_{K_0}} \calO_{\breve{K}_0}[v^\pm]) \big)^{\gamma}.$$ This shows that the formation of $\uT_0$ commutes with base change $\mathbb A^1_{E_0} \rightarrow \mathbb A^1_{K_0}$, where $E_0/K_0$ is any unramified extension. Similar remarks apply to $\ucT_0$. 
\end{ex}

 %Local models
 \subsection{Affine Grassmannians and Local Models}\label{Local_Models_Sec} 
We continue with the notation as in \S \ref{notation_sec}.
Recall that we fix a uniformizer $\varpi\in K$ with Eisenstein polynomial $Q\in\calO_{K_0}[u]$. 
Let $(G_0,A_0,S_0,T_0)$ be tamely ramified over $K$, and fix a spreading $(\uG_0,\uA_0,\uS_0,\uT_0)$ defined over $\calO_{K_0}[u^\pm]$ as in Proposition \ref{extension_prop}. 
Let $(G,A,S,T)$ be constructed from $(G_0,A_0,S_0,T_0)$ by Weil restriction of scalars along $K/F$ as in \S\ref{Parahoric_Group_Sec}.
Choose a facet $\bbf\in\scrA(G,A,F)$, and let $\calG:=\calG_\bbf$ be the corresponding parahoric $\calO_F$-group scheme for $G$. 
Associated with these data, we have the tuple $(\ucG_0,\ucA_0,\ucS_0,\ucT_0)$ of smooth affine group schemes over $X:=\Spec(\calO_{K_0}[u])$ constructed in Theorem \ref{parahoric_group_thm}. 
Since $\calO_{K_0}/\calO_F$ is finite \'etale, we can view $X$ as a smooth curve over $\calO_F$.
Let $D\subset X$ be the closed subscheme defined by $\{Q=0\}$ viewed as a relative effective Cartier divisor over $\calO_{F}$.
We are interested in local models for the group $G=\Res_{K/F}(G_0)$ with level structure given by the parahoric $\calO_F$-group $\calG=\calG_{\bbf}=\Res_{\calO_K/\calO_F}(\calG_{\bbf_0})$, cf. Proposition \ref{Parahoric_Weil_Prop}. 

\subsubsection{Affine Grassmannians for Weil-restricted groups} \label{affGrass_WR_def_subsec} 
The Beilinson-Drinfeld Grassmannian 
\begin{equation}\label{BD_Grass_notation}
\Gr_{\calG}\defined \Gr_{(X,\ucG_0, D)}
\end{equation}
from \eqref{dfnBD} specializes to \cite[Def.~4.1.1]{Lev16} for $K_0=F$. 
By Lemma \ref{tomato_lem}, we have 
\[
\Gr_{\calG}=\Gr_{(X/\calO_{F},\ucG_0, D)}=\Res_{\calO_{K_0}/\calO_F}\big(\Gr_{(X/\calO_{K_0},\ucG_0, D)}\big).
\]
Hence, our definition of $\Gr_{\calG}$ agrees with \cite[Prop.~4.1.8 ff.]{Lev16}.

We think about \eqref{BD_Grass_notation} as being the Beilinson-Drinfeld Grassmannian associated with the parahoric $\calO_F$-group scheme $\calG$. 
Explicitly, $\Gr_{\calG}$ is the functor on the category of $\calO_F$-algebras $R$ given by the isomorphism classes of tuples $(\calF,\al)$ with
\begin{equation}\label{BD_Grass_dfn}
\begin{cases}
\text{$\calF$ a $\underline{\calG}_0$-torsor on $\Spec\big((R\otimes_{\calO_F}\calO_{K_0})[u]\big)$};\\
\text{$\al\co \calF|_{\Spec\big((R\otimes_{\calO_F}\calO_{K_0})[u][\nicefrac{1}{Q}]\big)}\simeq \calF^0|_{\Spec\big((R\otimes_{\calO_F}\calO_{K_0})[u][\nicefrac{1}{Q}]\big)}$ a trivialization},
\end{cases}
\end{equation}
where $\calF^0$ denotes the trivial torsor. If $Q=u-\varpi$, i.e., $K=F$, then $\Gr_{{\calG}}$ is the BD-Grassmannian defined in \cite[6.2.3; (6.11)]{PZ13}.

For an $\calO_F$-algebra $R$, we have the regular functions on the completion of $X_R$ along $D_R$, namely the $\calO_{K_0}[u]$-algebra $R\pot{D}=\on{lim}_N(R\otimes_{\calO_F}\calO_{K_0})[u]/(Q^N)$, and likewise $R\rpot{D}=R\pot{Q}[\nicefrac{1}{Q}]$. With the notation of \S\ref{Loop_Group_Sec}, we have the loop group 
\[
L{\calG}(R)\defined L_D\ucG_0(R)=\underline{\calG}_0(R\rpot{D}),
\] 
and the positive loop group 
\[
L^+{\calG}(R) \defined L^+_D\underline{\calG}_0(R)=\underline{\calG}_0(R\pot{D}).
\] 
By Lemma \ref{BL_Lem}, there is a natural isomorphism $L\calG/L^+\calG\simeq \Gr_{\calG}$, and thus a transitive action morphism
\begin{equation}\label{global_act}
L{{\calG}}\times_{\calO_F} \Gr_{{{\calG}}}\longto \Gr_{{{\calG}}}.
\end{equation}
The following proposition is \cite[Prop.~4.1.6, 4.1.8]{Lev16}.

\begin{prop}\label{Fiber_BDGrass}
i\textup{)} The generic fiber of \eqref{global_act} is isomorphic to
\begin{equation}\label{global_act_generic}
L_z{G}\times_{F} \Gr_{{G}}\longto \Gr_{{G}},
\end{equation}
where $L_z{G}(R)={G}(R\rpot{z})=G((K\otimes_FR)\rpot{z})$ is the loop group for ${G}=\Res_{K/F}(G)$ formed using the parameter $z:=u-\varpi\in K[u]$, and $\Gr_{{G}}$ is as in Example \ref{Special_Cases} i\textup{)} the affine Grassmannian for the group ${G}\otimes_FF\pot{z}$, i.e., the \'etale-sheaf associated with the functor on $F$-algebras $R\mapsto {G}(R\rpot{z})/{G}(R\pot{z})$.\smallskip\\ 
ii\textup{)} The special fiber of \eqref{global_act} is canonically isomorphic to
\begin{equation}\label{global_act_special}
L\calG^\flat \times_{k_F} \Fl_{\calG^\flat}\longto \Fl_{\calG^\flat},
\end{equation}
where $L\calG^\flat(R)=\calG^\flat(R\rpot{u})$ is the twisted affine loop group for the parahoric $k_F\pot{u}$-group scheme $\calG^\flat=\calG_{\bbf^\flat}$ of $G^\flat$ as in \eqref{apartments_general}, and $\Fl_{\calG^\flat}$ is the twisted affine flag variety for $\calG^\flat/k_F\pot{u}$ defined in \cite{PR08}, i.e., the \'etale-sheaf associated with the functor on $k_F$-algebras $R\mapsto \calG^\flat(R\rpot{u})/\calG^\flat(R\pot{u})$. 
\end{prop}
\begin{proof}  
Let $L^+_{D/\calO_{K_0}} \ucG_0$ denote the positive loop group attached to $D$ and the curve $\bbA^1_{\calO_{K_0}}$, and let $L_{D/\calO_{K_0}}\ucG_0$ be the corresponding loop group. We then have
$$
L \calG = L_D \ucG_0 = \Res_{\calO_{K_0}/\calO_F}\big(L_{D/\calO_{K_0}} \ucG_0\big),
$$
and likewise for the positive loop group. We may use Theorem \ref{parahoric_group_thm} to compute the generic and special fibers of the right hand side. 
For example, if $\calG^\flat_0:=\ucG_0\otimes k_0\pot{u}$, then the special fiber of $L\calG$ is
\begin{equation} \label{Res_calG_flat}
\Res_{k_0/k}(L\calG^\flat_0) = L\big(\Res_{k_0\pot{u}/k\pot{u}}(\calG^\flat_0)\big) = L \calG^\flat,
\end{equation}
and likewise for the positive loop group. 
This together with Lemma \ref{tomato_lem} reduces us to the case that $K_0=F$.
Then part ii) is Corollary \ref{Fibers_Action_Map_Cor} i). For i), note the natural maps $\Res_{K/F}(L_zG_0)\to L_z\Res_{K/F}(G_0)$ and $\Res_{K/F}(\Gr_{G_0})\to \Gr_{\Res_{K/F}(G_0)}$ are isomorphisms, cf. \cite[(1.2)]{PR08} and \cite[\S 2.6]{Lev}. Note that $Q(z+\varpi) \in zK[z]$. Hence by induction on $n \geq 1$, the map $u \mapsto z + \varpi$ sets up an isomorphism $F[u]/(Q^n) \overset{\sim}{\rightarrow} K[z]/(z^n)$, and hence $F[\![u]\!] \overset{\sim}{\rightarrow} K[\![z]\!]$. Similarly, we remark that for any $F$-algebra $R$ $u \mapsto z + \varpi$ gives an isomorphism $R[\![u]\!] \cong (R \otimes_FK)[\![z]\!]$. Let $\underline{\calG}_{K\pot{z}}:=\underline{\calG}_0\otimes_{\calO_F[u]} K\pot{z}$, and denote by $\Gr_{\underline{\calG}_{K\pot{z}}}$ the twisted affine Grassmannian for $\underline{\calG}_{K\pot{z}}$, cf. Example \ref{Special_Cases} i). In view of Corollary \ref{Fibers_Action_Map_Cor}, or the above remark, the generic fiber of \eqref{global_act} is canonically isomorphic to the action morphism
\[
\Res_{K/F}(L\underline{\calG}_{K\pot{z}})\times_F\Res_{K/F}(\Gr_{\underline{\calG}_{K\pot{z}}})\,\longto\,\Res_{K/F}(\Gr_{\underline{\calG}_{K\pot{z}}}).
\]
Hence, as in \cite[\S 6.2.6]{PZ13} and \cite[Prop.~4.1.6]{Lev16} it suffices to give an isomorphism of $K\pot{z}$-groups $\underline{\calG}_{K\pot{z}} \simeq G_0\otimes_KK\pot{z}$. But as $u$ is invertible in $K\pot{z}$, we have $\underline{\calG}_{K\pot{z}}=\uG_0\otimes_{\calO_F[u^\pm]}K\pot{z}$. With the notation of \eqref{ring:diag}, the group scheme $\uG_0$ is constructed by descent from $\calO_{\tilde{K}_0}[v^\pm]$ where it is a constant Chevalley group scheme. As in \cite[(6.9)]{PZ13}, it is enough to give a commutative diagram of $\tilde{\Ga}$-covers
\begin{equation}\label{Iso_Alg}
\begin{tikzpicture}[baseline=(current  bounding  box.center)]
\matrix(a)[matrix of math nodes, 
row sep=1.5em, column sep=1.5em, 
text height=1.5ex, text depth=0.45ex] 
{\Spec(\calO_{\tilde{K}_0}[v^\pm]\otimes_{\calO_F[u^\pm]}K\pot{z}) & & \Spec(\tilde{K}\pot{z}) \\ 
& \Spec(K\pot{z})&  \\}; 
\path[->](a-1-3) edge node[above] {$\simeq$}  (a-1-1);
\path[->](a-1-1) edge node[below] {$\on{pr}_2$}  (a-2-2);
\path[->](a-1-3) edge node[below] {$\tau$} (a-2-2);
\end{tikzpicture}
\end{equation}
which matches the $\tilde{\Ga}$-action on $\calO_{\tilde{K}_0}[v^\pm]/\calO_F[u^\pm]$ via \eqref{ring:diag} with the $\tilde{\Ga}$-action on the coefficients in $\tilde{K}\pot{z}$ (see below for why this is enough). As in \cite[(6.9)]{PZ13}, the isomorphism is given on rings by $v\mapsto \tilde{\varpi}\cdot (1+z)$ and $z\mapsto b\cdot z$ with  
\[
b\,:=\, {{\varpi\cdot (1+z)^{\tilde{e}}-\varpi}\over z}\,\in\, K\pot{z}^\times.
\]
The map $\tau$ is the $K$-algebra morphism given by $z\mapsto b\cdot z$. (To see that the horizontal morphism is an isomorphism, observe that $K[\![z]\!] = K[\![bz]\!]$, and let $f(z) \in K[\![z]\!]$ be such that $f(bz) = (1 + z)^{-1}$; then $v \otimes f(z) \mapsto \tilde{\varpi}$ and the morphism is surjective.  One sees it is injective using an $\mathcal O_F[u]$-basis for $\mathcal O_{\tilde{K}_0}[v]$ of the form $a_i v^j$ for $a_i \in \mathcal O_{\tilde{F}_0}$ to write any element in the source uniquely in the form $\sum_{i,j} a_i v^j \otimes f_{ij}$ for $f_{ij} \in K[\![z]\!]$. To see that diagram (\ref{Iso_Alg}) suffices, note that the right oblique arrow is isomorphic via $\tilde{K}[\![z]\!] \overset{\sim}{\rightarrow} \tilde{K}[\![z]\!]$, $z \mapsto b \cdot z$, to the arrow
${\rm Spec}(\tilde{K}[\![z]\!]) \rightarrow {\rm Spec}(K[\![z]\!])$ induced by the inclusion $K[\![z]\!] \hookrightarrow \tilde{K}[\![z]\!]$.) 
%Since we fixed $\uG_K\simeq G$ in the beginning, the isomorphism $\underline{\calG}_{K\pot{z}}\simeq G\otimes_K K\pot{z}$ is canonical.
\end{proof}

Recall from \cite[Cor 11.7]{PZ13} that there exists a closed immersion of $X$-groups $\ucG_0\hookto \Gl_{n,X}$ such that the quotient $\Gl_{n,X}/\ucG_0$ is quasi-affine. Thus, the $\calO_F$-space $\Gr_{\calG}=\Gr_{(X,\ucG_0, D)}$ is representable by a separated $\calO_F$-ind-scheme of ind-finite type, cf.~Corollary \ref{Aff_Grass_Rep_Cor}. We need the following stronger statement.

\begin{thm}\label{BD_rep_thm} The BD-Grassmannian $\Gr_{{\calG}}=\on{colim}_i\Gr_{{{\calG}},i}$ is representable by an ind-projective $\calO_F$-ind-scheme, and for each $i$, the projective $\calO_F$-scheme $\Gr_{{{\calG}},i}$ can be choosen to be $L^+{{\calG}}$-stable compatible with the transition maps.
\end{thm}
\begin{proof} 
By Lemma \ref{tomato_lem} we reduce to the case that $K_0=F$.
Then the ind-projectivity is proven in \cite[Thm.~4.2.11, Prop.~5.1.5]{Lev16}. If $G$ is unramified, the proof is considerably simpler, cf.~ \cite[Prop.~2.2.8]{Lev16}.  The proof relies on the existence and properties of specialization morphisms 
\[
{\rm sp}\co {\rm Gr}_{{\calT}}(\bar{F}) \longrightarrow {\rm Gr}_{{\calT}}(\bar{k}), 
\]
where $\Gr_\calT\subset\Gr_\calG$ is the part induced from the maximal torus, cf.~Lemma \ref{Spec_Tor_Lem} below. Levin constructs this map ``by hand'' in \cite[Prop.\,4.2.8]{Lev16}.  We will follow a more conceptual approach which avoids constructing ${\rm sp}$ ahead of time and the calculations that entails. Our outline is the following:
\begin{enumerate}
\item[(a)] Prove ${\rm Gr}_{{\calT}}\to \Spec(\calO_F)$ is ind-finite, using the method of \cite[Lem.\,2.20]{Ri16b}, cf.~\S \ref{subsect_ind_finite}.
\item[(b)] Deduce existence of the specialization maps for ${\calT}$ via the valuative criterion of properness, and prove the required compatibility with Kottwitz homomorphisms, cf.~\S\ref{Spec_Map_Torus}.
\item[(c)] Use (b) to  show that each local model has non-empty special fiber and deduce by \cite[Lem.\,2.22]{Ri16b} that each local model is proper, cf.~\S\ref{sect_nonempty_special_fiber}.
\item[(d)] Conclude that ${\rm Gr}_{{\calG}}\to \Spec(\calO_F)$ is ind-proper, cf.~\S\ref{sect_conclusion}.
\end{enumerate}
In view of Lemma \ref{Aff_Grass_Rep_Lem} and Corollary \ref{Aff_Grass_Rep_Cor}, the ind-properness of $\Gr_{ \calG}$ implies the theorem. The steps (a)-(d) are explicated in the next several subsections, and with them the proof is concluded.
\end{proof}

\iffalse
The general case is proven using the method from \cite[\S1.5]{Ri16b} which relies on the ind-properness of $\Gr_{{\calG}}$ in the case where $G$ is a torus. The latter follows from the existence of specializations \cite[Prop.~4.2.8]{Lev16} which are recalled in \S\ref{Spec_Map_Torus} below. 
\fi

%Sketch of another method: First prove ind-properness in case $\calG=\Res_{\calO_L/\calO_K}(\bbG_m)$ with $L/K$ tame, then conclude the general case of tori by using a partial resolution by induced tori as in Kottwitz, Isocrystals 2. 

\subsubsection{${\rm Gr}_{{\calT}}$ is ind-finite} \label{subsect_ind_finite}
Recall that we already reduced to the case $K_0=F$ so that $K/F$ is totally ramified.
Without loss of generality, we further assume that $F=\bF$, $\calO_F=\calO_{\bF}$. 
Here we use that the formation of the affine Grassmannian \eqref{BD_Grass_notation} and the group scheme $\ucT_0$ from Theorem \ref{parahoric_group_thm} is compatible with unramified base change, cf.~also Example \ref{torus_eg}. 
We show that $\Gr_{ \calT}:=\Gr_{(X,\ucT_0,D)}$ is ind-proper over $\calO_F$ where $X=\bbA^1_{\calO_F}$ and $D=\{Q=0\}$. 
It is then ind-finite, since this holds fiberwise by Proposition \ref{Fiber_BDGrass}. We proceed in two steps as follows.\smallskip\\
Step 1): First assume that ${T}=\Res_{K/F}(T_0)$ where $T_0$ is an {\it induced} $K$-torus which splits over a tamely ramified extension. Then $T_0$ is isomorphic to a finite product of $K$-tori of the form $T_1:=\Res_{ K_1/K}(\bbG_m)$ where $K_1/K$ is a \iffalse {\cy We do not need Galois here, but I think we can if we want to. But then we need to assume it in the next paragraph also by taking the induced torus being a product along Galois extensions. }\fi tamely ramified finite field extension. Note that $K_1/K$ is totally ramified by our assumption $F=\bF$. Accordingly, the $\bbA^1_{\calO_F}$-group scheme $\ucT_0$ is isomorphic to a finite product of $\bbA^1_{\calO_F}$-group schemes of the form 
\[
\ucT_1:=\Res_{\calO_F[v]/\calO_F[u]}(\bbG_m),
\]
where $v^{[K_1:K]}=u$. After fixing a uniformizer $\varpi_1\in K_1$ with $(\varpi_1)^{[K_1:K]}=\varpi$ (possible because $F=\bF$), this can be verified using Example \ref{torus_eg} (use that, in this case, $T_H \otimes \tilde{\calO}_0[v] \cong \mathbb ({\mathbb G}_{m,\tilde{\calO}_0[v]})^{[K_1:K]}$ with ${\rm Gal}(K_1/K)$ acting via the permutation of the factors). \iffalse {\cy I think it is also ok if $K_1/K$ is not Galois. Then we need to choose $\tilde K$ etc., but the group scheme we produce is just $\ucT_1$. This can be verified by choosing uniformizers $\tilde \varpi$ and $\varpi_1$ compatibly.}\fi Likewise, the affine Grassmannian $\Gr_{ \calT}$ is a finite $\calO_F$-product of the affine Grassmannians $\Gr_{(X,\ucT_1, D)}$, where $X=\bbA^1_{\calO_F}$ and $D=\{Q(u)=0\}$. Hence, we reduce to the case where $\ucT_0=\ucT_1$, i.e., $T=\Res_{K_1/K}(\bbG_m)$. By Corollary \ref{cor_weil_restriction}, there is an equality of ind-schemes
\[
\Gr_{(X,\ucT_0, D)}\,=\, \Gr_{(X',\bbG_m,D')},
\]
where $X'=\bbA^1_{\calO_F}=\Spec(\calO_F[v])$ and $D'=\{Q(v^{[K_1:K]})=0\}$. We reduce to the case $X=X'$, $\ucT_0=\bbG_m$ and $D=D'$. Then $\Gr_{(X,\bbG_m,D)}$ is ind-projective (hence ind-proper) by Lemma \ref{Aff_Grass_Rep_Lem}. \smallskip\\
Step 2): Now let $T=\Res_{K/F}(T_0)$ where $T_0$ is an $K$-torus which splits over a tamely ramified extension. As in \cite[\S 7]{Ko97}, we choose a surjection of $K$-tori $T_1\to T_0$ where $T_1$ is induced, and where the kernel $T_2:=\ker(T_1\to T)$ is a 
%{\cm flasque -- Kottwitz doesn't show we can choose $T_2$ to be flasque (that was done by Colliot-Thelene and Sansuc) but it does not seem to be needed anyway ; we just need $T_2$ connected}  
$K$-torus. Note that $T_1$ can be chosen to split over a tamely ramified extension (and so does $T_2$ as well). 
The proof of \cite[Prop.~2.2.2]{KP18} adapts to our set-up, and the map $T_1\to T_0$ extends to a map of $X$-groups $\ucT_1\to\ucT_0$ with kernel $\ucT_2$ an $X$-group scheme extending $T_2$. (Instead of using \cite{KP18}, one can also deduce this making use of the prescription given in Example \ref{torus_eg}.) We claim that the resulting map of $\calO_F$-ind-schemes
\begin{equation}\label{BD_surjective}
\Gr_{\calT_1}=\Gr_{(X,\ucT_1,D)}\,\longto\, \Gr_{(X,\ucT_0,D)}=\Gr_{\calT}
\end{equation}
is surjective on the underlying topological spaces. Clearly, this can be tested on the fibers of \eqref{BD_surjective} over $\calO_F$ which are determined by Proposition \ref{Fiber_BDGrass}. The geometric generic fiber of \eqref{BD_surjective} is isomorphic (on the underlying topological spaces) to the map of discrete groups $X_*(\Res_{K/F}(T_1))\to X_*(\Res_{K/F}(T_0))$ which is surjective because $T_1\to T_0$ is surjective and its kernel $T_2$ is a torus (i.e.,\,connected). The geometric special fiber of \eqref{BD_surjective} is under the Kottwitz map isomorphic to $X_*(T^\flat_1)_{I_{k\rpot{u}}}\to X_*(T^\flat_0)_{I_{k\rpot{u}}}$ which is induced by $T^\flat_1:=\ucT_1 \otimes k\rpot{u}\to\ucT_0 \otimes k\rpot{u}=: T^\flat_0$. This map is isomorphic to $X_*(T_1)_{I_K}\to X_*(T_0)_{I_K}$ which follows by applying the Kottwitz map to the identification \eqref{Lattice_Identify}. As in \cite[\S7 (7.2.5)]{Ko97} the desired surjectivity now follows from $T_2$ being a $K$-torus. By Step 1), the $\calO_F$-scheme $\Gr_{\calT_1}$ is ind-proper and maps surjectively onto the separated ind-scheme $\Gr_{\calT}$ which is therefore ind-proper as well. This concludes \S \ref{subsect_ind_finite}.

\subsubsection{The specialization map}\label{Spec_Map_Torus} Once ${\rm Gr}_{{\calG}}$ is known to be ind-proper, by the valuative criterion for properness there exists a specialization map 
\begin{equation}\label{Spec_Map}
\on{sp}\co \Gr_{{G}}(\sF)=\Gr_{{\calG}}(\sF)\,\longto\, \Gr_{{\calG}}(\bar{k})=\Fl_{\calG^\flat}(\bar{k}).
\end{equation}
In case $G_0=T_0$ is a maximal torus, and hence $\underline{\calG}_0=\underline{\calT}_0$ is as in Theorem \ref{parahoric_group_thm} iv), we therefore know the existence of the specialization map. It is made explicit in \cite[Lem.~9.8]{PZ13}, \cite[Prop.~4.2.8]{Lev16}. Recall the following result for later use (which compared to {\em loc.\,cit.} is proved in a more conceptual way here).

\begin{lem} \label{Spec_Tor_Lem}
Let $\Ga_F$ denote the Galois group of $F$, and likewise $\Ga_k$ and $\Ga_{k_0}$.
There is a commutative diagram of abelian groups
\begin{equation}\label{Tor_Diag}
\begin{tikzpicture}[baseline=(current  bounding  box.center)]
\matrix(a)[matrix of math nodes, 
row sep=1.5em, column sep=2em, 
text height=1.5ex, text depth=0.45ex] 
{\Gr_{\Res_{K/F}(T_0)}(\sF) & X_*(\Res_{K/F}(T_0)) \\ 
\Fl_{\calT^\flat}(\bar{k})& X_*(\Res_{K/F}(T_0))_{I_F}, \\}; 
\path[->](a-1-1) edge node[above] {$\simeq$}  (a-1-2);
\path[->](a-1-2) edge node[left] {$\on{pr}$}  (a-2-2);
\path[->](a-2-1) edge node[above] {$\simeq$}  (a-2-2);
\path[->](a-1-1) edge node[left] {$\on{sp}$} (a-2-1);
\end{tikzpicture}
\end{equation}
which is Galois equivariant for the $\Ga_F$-action on the top covering the $\Ga_k$-action on the bottom.
\end{lem}
\begin{proof}  
The top arrow is the natural isomorphism, and the map $\on{pr}$ is the canonical projection to the coinvariants. 
Let us construct the bottom arrow.
Note that $X_*(\Res_{K/F}(T_0))={\rm Ind}_{\Gamma_K}^{\Gamma_F}(X_*(T_0))$ is an induced Galois module by the proof of Lemma \ref{cochar_Res_lem}. 
Shapiro's lemma induces a $\Ga_k$-equivariant isomorphism
\begin{equation}\label{Spec_Tor_Lem:eq1}
{\rm Ind}_{\Gamma_{k_0}}^{\Gamma_k}(X_*(T_0)_{I_K})\overset{\simeq}{\longto} X_*(\Res_{K/F}(T_0))_{I_F}.
\end{equation}
(For any $\mathbb Z[\Gamma_K]$-module $M$, we have $(M \otimes_{\mathbb Z[\Gamma_K]} \mathbb Z[\Gamma_F])_{I_F} = M_{I_K} \otimes_{\mathbb Z[\Gamma_{k_0}]} \mathbb Z[\Gamma_{k}]$ canonically.)
\iffalse
Details: Abbreviate $X:=X_*(T_0)$ viewed as $\Ga_K$-module. Since the residue fields of $K$ and $K_0$ agree, the coinvariants $X_{I_K}$ are naturally a $\Ga_{k_0}$-module. We now calculate 
$$X_*(\Res_{K/F}(T_0))_{I_F}=(X\otimes_{ \bbZ[\Ga_K]} \bbZ[\Ga_F])_{I_F}=X\otimes_{ \bbZ[\Ga_K]} \bbZ[I_F\backslash \Ga_F].$$
Using $I_F=I_{K_0}$ and $X\otimes_{ \bbZ[\Ga_K]}\bbZ[I_{K_0}\backslash\Ga_{K_0}]=X_{I_{K}}$, this can be rewritten as $X\otimes_{\bbZ[\Ga_{k_0}]}\bbZ[\Ga_k]$. 
Hence the desired conclusion. 
I would say it this way: for any $\mathbb Z[\Gamma_K]$-module $M$, there is an obvious isomorphism 
$(M \otimes_{\mathbb Z[\Gamma_K]} \mathbb Z[\Gamma_F])_{I_F} = M_{I_K} \otimes_{\mathbb Z[\Gamma_{k_0}]} \mathbb Z[\Gamma_{k}]$.
\fi
Further, the Kottwitz map (cf.~ \cite[\S 7]{Ko97}) applied to \eqref{Lattice_Identify} in the case of $T_0(\breve{K})$ (resp. $T^\flat_0(\bar{k}\rpot{u})$) induces a $\Ga_{k_0}$-equivariant isomorphism $X_*(T^\flat_0)_{I_{K_0^\flat}}\simeq X_*(T_0)_{I_K}$. 
Applying the induction functor we deduce a $\Ga_{k}$-equivariant isomorphism 
\begin{equation}\label{Spec_Tor_Lem:eq2}
{\rm Ind}_{\Gamma_{k_0}}^{\Gamma_k}\big(X_*(T^\flat_0)_{I_{K_0^\flat}}\big)\overset{\simeq}{\longto}{\rm Ind}_{\Gamma_{k_0}}^{\Gamma_k}\big(X_*(T_0)_{I_K}\big).
\end{equation}
Finally, we use $\Fl_{\calT^\flat}=\Res_{k_0/k}(\Fl_{\calT^\flat_0})$ (cf.\,(\ref{Res_calG_flat})) together with the isomorphism induced by the Kottwitz map
\[
\Fl_{\calT^\flat_0}(\bar{k})\,=\, T^\flat_0(\bar{k}\rpot{u})/\calT^\flat_0(\bar{k}\pot{u})\,\overset{\simeq}{\longto}\, X_*(T^\flat_0)_{I_{K_0^\flat}},
\]
which is $\Ga_{k_0}$-equivariant as well.
This induces the $\Ga_k$-equivariant isomorphism 
\begin{equation}\label{Spec_Tor_Lem:eq3}
\Fl_{\calT^\flat}(\bar k)\overset{\simeq}{\longto} {\rm Ind}_{\Gamma_{k_0}}^{\Gamma_k}(X_*(T^\flat_0)_{I_{K_0^\flat}}).
\end{equation}
The bottom arrow in \eqref{Tor_Diag} is defined to be the composition of \eqref{Spec_Tor_Lem:eq1}, \eqref{Spec_Tor_Lem:eq2} and \eqref{Spec_Tor_Lem:eq3}.

It remains to prove the commutativity which is a reformulation of \cite[Prop.~4.2.8]{Lev16}: the composition $\on{pr}$ with the inverse of \eqref{Spec_Tor_Lem:eq1} is the map given by $\mu'\mapsto \bar{\la}_{\mu'}$ in the notation of {\it loc.~cit.}. We show the commutativity as follows. 
The diagram is comaptible with unramified extensions, and we reduce to the case $K_0=F$.
Changing notation, we may now assume that $F=\bF$, $k=\bar{k}$. 
The diagram \eqref{Tor_Diag} is functorial in the tamely ramified $K$-torus $T_0$. Arguing as in \S \ref{subsect_ind_finite} Step 2), we choose an induced tamely ramified $K$-torus $T_1\twoheadrightarrow T_0$ with kernel being a torus. Each item in the diagram for $T_1$ maps surjectively onto each item in the diagram for $T_0$, and we reduce to the case where $T_0=T_1$ is an induced tamely ramified $K$-torus. Arguing as in \S \ref{subsect_ind_finite} Step 1), the torus $T_0$ is a product of $K$-tori of the form $\Res_{K_1/K}(\bbG_m)$ with $K_1/K$ being totally (tamely) ramified. Accordingly, each item in the diagram \eqref{Tor_Diag} splits as a product compatible with the maps, and we reduce to the case where $T_0=\Res_{K_1/K}(\bbG_m)$. Replacing the pair $(X,D)$ with the pair $(X',D')$ as in \S\ref{subsect_ind_finite} Step 1), we reduce further to the case where $T_0=\bbG_m$.  In this case, we have for the (global) loop group
\[
L\bbG_m(\calO_{\sF})=(L\bbG_m)_{(X,\bbG_m,D)}(\calO_{\sF})\,=\, \calO_{\sF}\rpot{Q}^\times,
\]
where $Q\in \calO_F[u]$ is the minimal polynomial of $\varpi\in K$ over $F$. Writing $Q=(u-a_1)\cdot \ldots\cdot (u-a_d)$ for $d=[K:F]$ and pairwise distinct elements $a_1,\ldots,a_d\in\calO_{\sF}$, we compute for the generic fiber
\[
(L\bbG_m)_{(X,\bbG_m,D)}(\sF)\,=\,\prod_{i=1,\ldots,n}\bar{F}\rpot{u-a_i}^\times.
\]
For $i=1,\ldots,d$, let $v_i$ be the $(u-a_i)$-adic valuation of $\sF\rpot{u-a_i}$. The specialization map \eqref{Spec_Map} is explicitly given by the map
\[
\prod_{i=1,\ldots,d }\sF\rpot{u-a_i}^\times/\sF\pot{u-a_i}^\times \to k\rpot{u}^\times/k\pot{u}^\times,\;\;\;(x_1,\ldots,x_d)\mapsto u^{\sum_{i=1}^dv_i(x_i)},
\]
where we use that $Q\equiv u^{[K:F]} \mod \varpi$. 
One checks that \eqref{Tor_Diag} commutes for $T_0=\bbG_m$ which finishes the proof of the lemma.
\end{proof}

%Another proof can be obtained as in \cite[Lem.~2.21]{Ri16b} by reduction to an induced torus, and then treating the case $T=\Res_{L/K}(\bbG_m)$ with $L/K$ tamely ramified: Using the behavior of bundles for Weil restriction of scalars, one reduces to the case where $L/K$ is unramified, and further to the case $L=K$. The lemma boils down to an explicit calculation for $T=\bbG_m$.

\subsubsection{Local Models for Weil-restricted groups} \label{Dfn_Loc_Mod_Sec} We now recall the definition of local models for the pair $(G,\calG)=(\Res_{K/F}(G_0),\Res_{{\calO_K}/{\calO_F}}(\calG_{\bbf_0}))$. Let $\{\mu\}$ be a $G(\sF)$-conjugacy class of geometric cocharacters with reflex field $E/F$. For a representative $\mu\in \{\mu\}$, the associated {\em Schubert variety} is the reduced $L^+_zG_\sF$-orbit closure
\begin{equation}\label{Schubert_Dfn}
\Gr_{{G}}^{\leq \{\mu\}}\defined \overline{L^+_z{G}_\sF\cdot z^\mu\cdot e_0} \,\subset\, \Gr_{{G},\sF}.
\end{equation}
The $\sF$-scheme $\Gr_{{G}}^{\leq \{\mu\}}$ is defined over the reflex field $E=E(\{\mu\})$, i.e., the field of definition of $\{\mu\}$ which is a finite extension of $F$, and is a (geometrically irreducible) projective $E$-variety.

The following definition is \cite[Def 7.1]{PZ13} if $K/F$ is tamely ramified, and  \cite[Def 4.2.1]{Lev16} in general, cf. \cite[Prop.~4.2.4]{Lev16}).

\begin{dfn}\label{Local_Model_Dfn} The local model $M_{\{\mu\}}=M(\uG_0,\calG_\bbf,\{\mu\},\varpi)$ is the scheme theoretic closure of the locally closed subscheme
\[
\Gr_{{G}}^{\leq\{\mu\}}\,\hookto\, \Gr_{{G}}\otimes_F E\,\hookto\, \Gr_{{\calG}}\otimes_{\calO_F}\calO_E,
\]
%\[\Gr_{{G}}^{\leq\{\mu\}}\,\hookto\, (\Gr_{{G}}\otimes_F E)_{\red}\,\hookto\, (\Gr_{{\calG}}\otimes_{\calO_F}\calO_E)_{\red},\]
where $\Gr_{{G}}^{\leq \{\mu\}}$ is as in \eqref{Schubert_Dfn}.
\end{dfn} 

By definition, the local model $M_{\{\mu\}}$ is a closed flat $L^+{\calG}_{\calO_E}$-invariant subscheme of $(\Gr_{{\calG}}\otimes_{\calO_F}\calO_E)_{\red}$ which is uniquely determined up to unique isomorphism by the data $(\uG_0,\calG_\bbf,\{\mu\},\varpi)$. 
Its generic fiber $M_{\{\mu\}}\otimes E= \Gr_{{G},E}^{\leq\{\mu\}}$ is a (geometrically irreducible) variety, and the special fiber $M_{\{\mu\}}\otimes k_E$ is equidimensional, cf. \cite[Thm.~14.114]{GW10}. 
By Proposition \ref{Fiber_BDGrass}, the map $\Gr_{{\calG}}\to \Spec(\calO_F)$ is fiberwise ind-proper, and hence the map $M_{\{\mu\}}\to \Spec(\calO_E)$ is fiberwise proper. 
Note that there is a closed embedding into the flag variety
\begin{equation}\label{Schubert_union_flag}
M_{\{\mu\}}\otimes k_E\,\hookto\, \Gr_{{\calG}}\otimes_{\calO_F}k_E\,=\,\Fl_{\calG^\flat, k_E},
\end{equation}
which identifies the reduced locus $(M_{\{\mu\}}\otimes k_E)_\red$ with a union of Schubert varieties in $\Fl_{\calG^\flat, k_E}$. 

\begin{rmk} \label{Unique_Model_dfn}
The local model $M_{\{\mu\}}$ should up to unique isomorphism only depend on the data $({G},{{\calG}},\{\mu\})$. The uniqueness of $M_{\{\mu\}}$ is a separate question, and not of importance for the present article. We refer the reader to \cite[Rmk 3.2]{PZ13} for remarks on the uniqueness of $\uG_0$, and to \cite[Rmk 4.2.5]{Lev16} for remarks on the independence of $M_{\{\mu\}}$ on the choice of the uniformizer $\varpi\in K$. In the recent preprint \cite[Thm.~2.7]{HPR}, it is shown the ind-scheme $\Gr_{ \calG}$ for $K=F$ depends up to equivariant isomorphism only on the data $( G, \calG)$. So $M_{\{\mu\}}$ for $K = F$ depends up to equivariant isomorphism only on the data $( G, \calG,\{\mu\})$. Note that \cite[Conj 2.12]{HPR} uniquely characterizes $M_{\{\mu\}}$ for $K = F$ in the case where $\{\mu\}$ is minuscule.

%{\cm also add reference to the recent preprint of He-Pappas-Rapoport, which discusses this question}
\end{rmk}

%%%%%%Local models are proper
\subsubsection{Each local model is proper}\label{sect_nonempty_special_fiber}
For every conjugacy class $\{\mu\}$, we need to show that the local model $M_{\{\mu\}}$ is proper over $\calO_E$ where $E={E(\{\mu\})}$ is the reflex field. In view of \cite[Lem.~2.20]{Ri16b} and the discussion after Definition \ref{Local_Model_Dfn}, it remains to show that the special fiber of $M_{\{\mu\}}$ is non-empty. The inclusion $\ucT_0\subset \ucG_0$ induces a map of $\calO_F$-ind-schemes
\begin{equation}\label{Closed_Torus_Imm}
\Gr_{\calT}=\Gr_{(X,\ucT_0,D)}\,\to\,\Gr_{(X,\ucG_0,D)}=\Gr_{ \calG}.
\end{equation}
In the notation of Proposition \ref{Fiber_BDGrass}, the geometric generic fiber $M_{\{\mu\}}(\sF)$ contains the element 
\[
\mu\in \Gr_{ T}(\sF)=\Gr_{\calT}(\sF),
\]
for any representative $\mu\in X_*(\tilde T)$ of $\{\mu\}$. As $\Gr_{ \calT}$ is ind-finite (hence ind-proper) by \S\ref{subsect_ind_finite}, the element $\mu\in \Gr_{ \calT}(\sF)$ uniquely extends to a point $\tilde\mu\in\Gr_{\calT}(\calO_\sF)$ by the valuative criterion for properness. Composed with \eqref{Closed_Torus_Imm}, this defines a point (still denoted) $\tilde\mu\in \Gr_{\calG}(\calO_\sF)$. Since $M_{\{\mu\}}\subset \Gr_{ \calG,\calO_E}$ is a closed subscheme, we have
\begin{equation}\label{translation_ele}
\tilde\mu\in M_{\{\mu\}}(\sF)\cap\Gr_{ \calG}(\calO_\sF)\,=\,M_{\{\mu\}}(\calO_\sF),
\end{equation}
and its special fiber ${\bar{\mu}}:=\tilde\mu_{\bar{k}}\in M_{\{\mu\}}(\bar{k})$ is non-empty. This concludes \S\ref{sect_nonempty_special_fiber}.

\subsubsection{Conclusion of Proof of Theorem \ref{BD_rep_thm}}\label{sect_conclusion}
We need to show that $\Gr_{ \calG}\to \Spec(\calO_F)$ is ind-proper. It suffices to prove that the map $(\Gr_{ \calG}\otimes \calO_\sF)_\red\to \Spec(\calO_\sF)$ is ind-proper. In view of \S\ref{sect_nonempty_special_fiber}, we have to show that the closed immersion 
\begin{equation}\label{pres_subset}
\bigcup_{\{\mu\}}(M_{\{\mu\},\calO_\sF})_\red\,\subset\, (\Gr_{ \calG}\otimes \calO_\sF)_\red
\end{equation} 
is an equality. Here $\{\mu\}$ ranges over all ${G}(\sF)$-conjugacy classes of geometric cocharacters. As both ind-schemes in \eqref{pres_subset} are reduced, one can check the equality on the underlying topological spaces. As in \cite[\S 2.5]{Ri16b} (resp.~\cite[Thm.~4.2.11]{Lev16}), this follows from Lemma \ref{Spec_Tor_Lem} combined with \eqref{Schubert_union_flag} and \eqref{translation_ele}. This concludes \S\ref{sect_conclusion}, and hence the proof of Theorem \ref{BD_rep_thm}.

  \section{Actions of $\bbG_m$ on Weil-restricted affine Grassmannians}\label{Gm_Act_Sec} 
 %\section{$\bbG_m$-actions on affine Grassmannians for Weil-restricted groups}\label{Gm_Act_Sec} 
 
 %Geometry
 \subsection{Geometry of $\bbG_m$-actions on affine Grassmannians}\label{Geometry_Gm_Act}
Fix the data and notation as in \S \ref{Local_Models_Sec}. 
In particular, we denote the group schemes over $X=\bbA^1_{\calO_{K_0}}$ by $(\underline{\calG}_0,\underline{\calA}_0,\underline{\calS}_0,\underline{\calT}_0)$.

%Main geometric result 
\subsubsection{Main geometric result} 
Let $\chi\co \bbG_{m,K}\to A_0\subset G_0$ be a cocharacter which acts on $G_0$ by conjugation. 
As in \eqref{hyperlocgroup}, the centralizer is a Levi subgroup $M_0\subset G_0$, and the attractor (resp.\,repeller) subgroup $P^+_0$ (resp. $P^-_0$) is a parabolic subgroup with $P^+_0\cap P^-_0=M_0$. Further, we have semidirect product decompositions $P^\pm_0=M_0\rtimes N^\pm_0$ defined over $K$. 

Via the fixed isomorphism $\ucG_{0,K}\simeq G$ compatible with $\underline{\calA}_{0,K}\simeq A_0$, we may view $\chi$ as a cocharacter of $\underline{\calA}_{0,K}$. 
As $X$ is connected and $\underline{\calA}$ a split torus, $\chi$ extends uniquely to a cocharacter also denoted
 \begin{equation}\label{chi_dfn}
 \chi\co \bbG_{m,X}\,\longto\, \underline{\calA}_0\,\subset\, \underline{\calG}_0.
 \end{equation}
Hence, the cocharacter $\chi$ acts by conjugation on $\underline{\calG}_0$ via the rule $\bbG_{m,X}\times_X\underline{\calG}_0\to\underline{\calG}_0$, $(\la,g)\mapsto \chi(\la)\cdot g\cdot\chi(\la)^{-1}$. Using the dynamic method promulgated in \cite{CGP10}, the functors \eqref{flow} define $X$-subgroup schemes of $\underline{\calG}_0$ given by the fixed points $\underline{\calM}_0=\underline{\calG}_0^{0,\chi}$, and the attractor $\underline{\calP}^+_0=\underline{\calG}^{+,\chi}_0$ (resp. the repeller $\underline{\calP}^-_0=\underline{\calG}_0^{-,\chi}$). 
Note that $\underline{\calM}_0$ is by definition the schematic centralizer of $\chi$ in $\underline{\calG}_0$. 
 
\begin{lem}\label{dynamic_group_lem}
 i\textup{)} The $X$-group schemes $\underline{\calM}_0$ and $\underline{\calP}^\pm_0$ are smooth closed subgroup schemes of $\underline{\calG}_0$ with geometrically connected fibers.\smallskip\\
ii\textup{)} The centralizer $\underline{\calM}_0$ is a parahoric $X$-group scheme for $M_0$ in the sense of Theorem \ref{parahoric_group_thm}.\smallskip\\
iii\textup{)} There is a semidirect product decomposition as $X$-group schemes $\underline{\calP}^\pm_0=\underline{\calM}_0\ltimes \underline{\calN}^\pm_0$ where $\underline{\calN}^\pm_0$ is a smooth affine group scheme with geometrically connected fibers.\smallskip\\
iv\textup{)} The fixed isomorphism $\uG_{0,K}\simeq G_0$ induces isomorphisms of $K\pot{z}$-groups $\underline{\calM}_{0,K\pot{z}}\simeq M_0\otimes_KK\pot{z}$, and $\underline{\calP}^\pm_{0,K\pot{z}}\simeq P^\pm_0\otimes_KK\pot{z}$ compatible with the semidirect product decomposition in iii\textup{)}. 
\end{lem}
\begin{proof} The method of \cite[Lem.~5.15]{HaRi} extends to give i), ii) and iii) of the lemma. 
Part iv) is immediate from the construction of $\chi$, and the proof of Proposition \ref{Fiber_BDGrass} i). 
\end{proof}

By \eqref{hyper_loc_maps}, there are natural maps of $X$-group schemes
\begin{equation}\label{group_maps}
\underline{\calM}_0\;\leftarrow\;\underline{\calP}_0^\pm\;\to\; \underline{\calG}_0.
\end{equation}
The maps \eqref{group_maps} induce, by functoriality of BD-Grassmannians, maps of $\calO_F$-spaces
\begin{equation}\label{group_maps_Grass}
\Gr_{\calM} \;\leftarrow\; \Gr_{{\calP}^\pm} \;\to\; \Gr_{{\calG}},
\end{equation}
where $\Gr_{{\calG}}:=\Gr_{(X,\underline{\calG}_0,D)}$ (resp.\,$\Gr_{{\calM}}:=\Gr_{(X,\underline{\calM}_0,D)}$; resp.\,$\Gr_{{\calP}^\pm}:=\Gr_{(X,\underline{\calP}^\pm_0,D)}$) by notational convention. 
In light of \cite[Cor 11.7]{PZ13} and Corollary \ref{Aff_Grass_Rep_Cor} i), the functors in \eqref{group_maps_Grass} are representable by separated $\calO_F$-ind-schemes of ind-finite type. 
Note that by Theorem \ref{BD_rep_thm} i) and Lemma \ref{dynamic_group_lem} ii), the $\calO_F$-ind-schemes $\Gr_{{\calG}}$ and $\Gr_{{\calM}}$ are even ind-projective. 
The $\calO_F$-ind-scheme $\Gr_{{\calP}}$ is never ind-projective besides the trivial cases.
 
By functoriality of the loop group, we obtain via the composition
\begin{equation}\label{Gm_act_dfn}
\bbG_{m,\calO_F}\,\subset\, L^{+}_D\bbG_{m,X} \overset{L^{+}_D\chi\phantom{h}}{\longto}\, L^{+}_D\underline{\calA}_0\,\subset\, L^{+}_D\underline{\calG}_0
\end{equation}
a $\bbG_{m,\calO_F}$-action on $\Gr_{{\calG}}\to \Spec(\calO_F)$. 

\begin{lem}\label{Loc_Lin_BD_Lem}
The $\bbG_m$-action on $\Gr_{{\calG}}$ is Zariski locally linearizable.
\end{lem}
\begin{proof}
By \cite[Cor 11.7]{PZ13} there exists an monomorphism of $X$-groups $\underline{\calG}_0\hookto \Gl_{n,X}$ such that the fppf-quotient $\Gl_{n,X}/\underline{\calG}_0$ is quasi-affine. 
Hence, the induced monomorphism $\iota\co \Gr_{{\calG}}\hookto\Gr_{\Gl_{n,X}}$ is representable by a quasi-compact immersion (cf. Proposition \ref{Aff_Grass_Rep_Prop}) which is even a closed immersion because $\Gr_{{\calG}}$ is ind-proper, cf.\,Theorem \ref{BD_rep_thm}. 
The map $\iota$ is $\bbG_m$-equivariant for the cocharacter $\bbG_{m,X}\overset{\chi}{\to}\underline{\calG}_0\to\Gl_{n,X}$, and we reduce to the case $\underline{\calG}_0=\Gl_{n,X}$. 
By \cite[Prop.~6.2.11]{Co14} (use $\Pic(X)=0$), the cocharacter $\chi\co \bbG_{m,X}\to \Gl_{n,X}$ is conjugate to a cocharacter with values in the standard diagonal torus, and hence defined over $\calO_F$. The lemma follows from the proof of Lemma \ref{Loc_Lin_Lem}. 
\end{proof}

In light of Theorem \ref{BD_rep_thm} and Theorem \ref{Gm_thm}, we obtain maps of separated $\calO_F$-ind-schemes
\begin{equation}\label{Gm_maps_Grass}
(\Gr_{{\calG}})^0\;\leftarrow\; (\Gr_{{\calG}})^\pm\;\to\; \Gr_{{\calG}}.
\end{equation}
The following theorem compares \eqref{group_maps_Grass} with \eqref{Gm_maps_Grass}.

 %Main Geometric Input
 \begin{thm}\label{BD_Gm_decom}
The maps  induce a commutative diagram of $\calO_F$-ind-schemes
\begin{equation}\label{BD_Gm_Decom_Diag}
\begin{tikzpicture}[baseline=(current  bounding  box.center)]
\matrix(a)[matrix of math nodes, 
row sep=1.5em, column sep=2em, 
text height=1.5ex, text depth=0.45ex] 
{\Gr_{{\calM}} & \Gr_{{\calP}^\pm} & \Gr_{{\calG}} \\ 
(\Gr_{{\calG}})^0& (\Gr_{{\calG}})^\pm& \Gr_{{\calG}}, \\}; 
\path[->](a-1-2) edge node[above] {}  (a-1-1);
\path[->](a-1-2) edge node[above] {}  (a-1-3);
\path[->](a-2-2) edge node[below] {}  (a-2-1);
\path[->](a-2-2) edge node[below] {} (a-2-3);
\path[->](a-1-1) edge node[left] {$\iota^0$} (a-2-1);
\path[->](a-1-2) edge node[left] {$\iota^\pm$} (a-2-2);
\path[->](a-1-3) edge node[left] {$\id$} (a-2-3);
\end{tikzpicture}
\end{equation}
where the maps $\iota^0$ and $\iota^\pm$ satisfy the following properties:\smallskip\\
i\textup{)} In the generic fiber, the diagram is isomorphic to \eqref{BD_Gm_decom:diag1} below, and the maps $\iota^0_F$ and $\iota^\pm_F$ are isomorphisms.\smallskip\\
ii\textup{)} In the special fiber, the diagram is isomorphic to \eqref{BD_Gm_decom:diag2} below, and the maps $\iota^0_k$ and $\iota^\pm_k$ are closed immersions which are open immersions on the underlying reduced loci.\smallskip\\
iii\textup{)} The maps $\iota^0$ and $\iota^\pm$ are closed immersions which are open immersions on the underlying reduced loci.
\end{thm}

The diagram is constructed as follows. The fppf-quotient $\underline{\calG}_0/\underline{\calM}_0$ is quasi-affine by \cite[Thm.~2.4.1]{Co14}, which implies that the map $\Gr_{{\calM}}\to \Gr_{{\calG}}$ as in the proof of Lemma \ref{Loc_Lin_BD_Lem} is representable by a closed immersion. 
Since the $\bbG_m$-action on $\Gr_{{\calM}}$ is trivial, the map factors as $\Gr_{{\calM}}\to(\Gr_{{\calG}})^0\to\Gr_{{\calG}}$, and we obtain the closed immersion $\iota^0$. 

The map $\iota^\pm$ is given via a Rees construction in terms of the moduli description \eqref{BD_Grass_dfn}, cf. \S \ref{Construction_Sec}. Alternatively, if we choose a monomorphism of $X$-groups $\underline{\calG}_0\hookto \Gl_{n,X}$
such that $\Gl_{n,X}/\underline{\calG}_0$ is quasi-affine (cf.~ \cite[Cor 11.7]{PZ13}), then the same argument as in \eqref{commsquareBD} applies, and we conclude that $\iota^\pm$ is representable by a quasi-compact immersion. 
We do not repeat the argument here, but instead refer the reader to \S \ref{Construction_Sec} for details. This constructs the commutative diagram \eqref{BD_Gm_Decom_Diag}. 

%Proof main geometric input
\begin{proof}[Proof of Theorem \ref{BD_Gm_decom}]
{\it Part i\textup{)}.} In the generic fiber, \eqref{BD_Gm_Decom_Diag} is by \eqref{global_act_generic} and Lemma \ref{dynamic_group_lem} iv), the commutative diagram of $F$-ind-schemes
\begin{equation}\label{BD_Gm_decom:diag1}
\begin{tikzpicture}[baseline=(current  bounding  box.center)]
\matrix(a)[matrix of math nodes, 
row sep=1.5em, column sep=2em, 
text height=1.5ex, text depth=0.45ex] 
{\Gr_{{M}} & \Gr_{{P}^\pm} & \Gr_{{G}} \\ 
(\Gr_{{G}})^0& (\Gr_{{G}})^\pm& \Gr_{{G}}, \\}; 
\path[->](a-1-2) edge node[above] {}  (a-1-1);
\path[->](a-1-2) edge node[above] {}  (a-1-3);
\path[->](a-2-2) edge node[below] {}  (a-2-1);
\path[->](a-2-2) edge node[below] {} (a-2-3);
\path[->](a-1-1) edge node[left] {$\iota^0_F$} (a-2-1);
\path[->](a-1-2) edge node[left] {$\iota^\pm_F$} (a-2-2);
\path[->](a-1-3) edge node[left] {$\id$} (a-2-3);
\end{tikzpicture}
\end{equation}
where ${G}=\Res_{K/F}(G_0)$ (resp. ${M}=\Res_{K/F}(M_0)$; resp. ${P}^\pm=\Res_{K/F}(P^\pm_0)$). The $\bbG_m$-action on the diagram is induced by the $L^+_z$-construction applied to the cocharacter
\[
\tilde{\chi}\co \bbG_{m,F}\subset \Res_{K/F}(\bbG_{m,K}) \overset{\Res_{K/F}(\chi)\phantom{h}}{\longto}  \Res_{K/F}(A_0) \subset  \Res_{K/F}(G_0)={G},
\]
combined with the inclusion $\bbG_{m,F}\subset L^+_z\bbG_{m,F}$. We claim that the conjugation action of $\tilde{\chi}$ on ${G}$ gives the group of fixed points ${M}={G}^{0,\tilde{\chi}}$ and the attractor (resp. repeller) group ${{P}}^+={G}^{+,\tilde{\chi}}$ (resp. ${{P}}^-={G}^{-,\tilde{\chi}}$). Indeed, the canonical maps of $F$-subgroups of ${G}$,
 \begin{align*}
\Res_{K/F}(M_0)&\,\hookto\, {G}^{0, \tilde{\chi}}\\
\Res_{K/F}(P^\pm_0)&\,\hookto\, {G}^{\pm, \tilde{\chi}}
\end{align*}
are isomorphisms. By descent, it is enough to prove this after passing to $\sF$. 
But ${G}\otimes_F\sF\simeq \prod_{K\hookto \sF}G_0\otimes_{K,\psi}\sF$, where the $\bbG_m$-action induced by $\tilde{\chi}$ is the diagonal action on the product. 
Lemma \ref{Gm_product_lem} implies the claim. Part i) follows from \cite[Prop.~3.4]{HaRi} applied to the pair $({G},\tilde{\chi})$.\smallskip\\
{\it Part ii\textup{)}.} In the special fiber, \eqref{BD_Gm_Decom_Diag} is the commutative diagram of $k$-ind-schemes
\begin{equation}\label{BD_Gm_decom:diag2}
\begin{tikzpicture}[baseline=(current  bounding  box.center)]
\matrix(a)[matrix of math nodes, 
row sep=1.5em, column sep=2em, 
text height=1.5ex, text depth=0.45ex] 
{\Fl_{\calM^\flat} & \Fl_{{\calP^\flat}^{\pm}} & \Fl_{\calG^\flat} \\ 
( \Fl_{\calG^\flat})^0& ( \Fl_{\calG^\flat})^\pm&  \Fl_{\calG^\flat}. \\}; 
\path[->](a-1-2) edge node[above] {}  (a-1-1);
\path[->](a-1-2) edge node[above] {}  (a-1-3);
\path[->](a-2-2) edge node[below] {}  (a-2-1);
\path[->](a-2-2) edge node[below] {} (a-2-3);
\path[->](a-1-1) edge node[left] {$\iota^0_k$} (a-2-1);
\path[->](a-1-2) edge node[left] {$\iota^\pm_k$} (a-2-2);
\path[->](a-1-3) edge node[left] {$\id$} (a-2-3);
\end{tikzpicture}
\end{equation}
The $\bbG_m$-action on the diagram is given as follows. 
Base changing \eqref{chi_dfn} along $\calO_{K_0}[u]\to k_0\pot{u}$ and taking restriction of scalars along $k_0\pot{u}/k\pot{u}$, we obtain the cocharacter
\[
\chi^\flat\co \bbG_{m,k\pot{u}}\subset \Res_{k_0\pot{u}/k\pot{u}}(\bbG_{m,k_0\pot{u}}) \subset \calG^\flat,
\]
which factors through $\calA^\flat\subset\calG^\flat$, the natural $k\pot{u}$-extension of the maximal $k\rpot{u}$-split torus $A^\flat\subset G^\flat$ in \eqref{analogues}.
Then $\bbG_{m,k}$ acts on the diagram via $\chi^\flat$ after applying the $L^+$-construction combined with the inclusion $\bbG_{m,k}\subset L^+\bbG_{m,k\pot{u}}$. Since taking fixed points (resp.~attractors; resp.~repellers) commutes with base change \cite[(1.3)]{Ri19} and is compatible with restriction of scalars along \'etale extensions, we have $\calM^\flat=(\calG^\flat)^{0,\chi^\flat}$ and $\calP^{\flat, \pm}=(\calG^\flat)^{\pm,\chi^\flat}$. Part ii) follows from \cite[Prop.~4.7]{HaRi} applied to the pair $(\calG^\flat,\chi^\flat)$.\smallskip\\
{\it Part iii\textup{)}.} This follows as in \cite[Thm.~5.5, 5.17]{HaRi} using Proposition \ref{conn_comp_BD} below, and we sketch the argument for convenience. With the notation of Proposition \ref{conn_comp_BD}, the map $\iota^0$ (resp. $\iota^\pm$) factors as a set-theoretically bijective quasi-compact immersion
\[
\iota^{0,c}\co \Gr_{{\calM}}\to (\Gr_{{\calG}})^{0,c}\;\;\;\;\; \text{(resp. $\iota^{\pm,c}\co \Gr_{{\calP}^\pm}\to (\Gr_{{\calG}})^{\pm,c}$)},
\]
where $(\Gr_{{\calG}})^{0,c}$ (resp. $(\Gr_{{\calG}})^{\pm,c}$) is an open and closed $\calO_F$-sub-ind-scheme of $(\Gr_{{\calG}})^{0}$ (resp. $(\Gr_{{\calG}})^{\pm}$). But any such map $\iota^{0,c}$ (resp. $\iota^{\pm,c}$) is a closed immersion which is an isomorphism on the underlying reduced loci, cf. \cite[Lem.~5.7]{HaRi}. 
\end{proof}

We record the following properties.

\begin{lem}\label{basic_geo_BD}
i\textup{)} The map $ (\Gr_{{\calG}})^\pm\to \Gr_{{\calG}}$ is schematic.
\smallskip\\
ii\textup{)} The map $(\Gr_{{\calG}})^\pm\to (\Gr_{{\calG}})^0$ is ind-affine with geometrically connected fibers, and induces an isomorphism on the group of connected components $\pi_0((\Gr_{{\calG}})^\pm)\simeq\pi_0((\Gr_{{\calG}})^0)$.
\end{lem}
\begin{proof} 
These are general properties of attractors in ind-schemes endowed with \'{e}tale locally linearizable $\mathbb G_m$-actions, cf.\,Lemma  \ref{Loc_Lin_BD_Lem}, and Theorem \ref{Gm_thm} ii) or \cite[Thm.\,2.1 ii)]{HaRi}.
\end{proof}

The following proposition decomposes the image of the maps $\iota^0$ and $\iota^\pm$ into connected components, and will be important in what follows.

\begin{prop}\label{conn_comp_BD}
Let either $N=N^+$ or $N=N^-$. There exists an open and closed $\calO_F$-ind-subscheme $(\Gr_{{\calG}})^{0,c}$ \textup{(}resp. $(\Gr_{{\calG}})^{\pm,c}$\textup{)} of $(\Gr_{{\calG}})^0$ \textup{(}resp. $(\Gr_{{\calG}})^{\pm}$\textup{)} together with a disjoint decomposition, depending up to sign on the choice of $N$, as $\calO_F$-ind-schemes
\[
(\Gr_{{\calG}})^{0,c}=\coprod_{m\in \bbZ}(\Gr_{{\calG}})^{0}_m  \;\;\;\;\text{\textup{\big(}resp. $(\Gr_{{\calG}})^{\pm,c}=\coprod_{m\in \bbZ}(\Gr_{{\calG}})^{\pm}_m$\textup{\big)}},
\]
which has the following properties.\smallskip\\
i\textup{)} The map $\iota^0\co \Gr_{{\calM}}\to (\Gr_{{\calG}})^{0}$ \textup{(}resp. $\iota^\pm\co \Gr_{{\calP}^\pm}\to (\Gr_{{\calG}})^{\pm}$\textup{)} factors through $(\Gr_{{\calG}})^{0,c}$ \textup{(}resp. $(\Gr_{{\calG}})^{\pm,c}$\textup{)} inducing a closed immersion $\iota^{0,c}\co \Gr_{{\calM}}\to (\Gr_{{\calG}})^{0,c}$ \textup{(}resp. $\iota^{\pm,c}\co \Gr_{{\calP}^\pm}\to (\Gr_{{\calG}})^{\pm,c}$\textup{)} which is an isomorphism on reduced loci. \smallskip\\
ii\textup{)} The complement $(\Gr_{{\calG}})^{0}\bslash (\Gr_{{\calG}})^{0,c}$ \textup{(}resp. $(\Gr_{{\calG}})^{\pm}\bslash (\Gr_{{\calG}})^{\pm,c}$ \textup{)} has empty generic fiber, i.e., is concentrated in the special fiber.
\end{prop}
\begin{proof}
The proof follows closely \cite[Prop.~5.6, 5.19]{HaRi}. We recall some steps of the construction. Let us denote $\calO:=\calO_F$, and $\breve{\calO}:=\calO_{\bF}$. Let $\pi_1(M)=X_*(T)/X_*(T_{M_{\scon}})$ be the algebraic fundamental group of $M$ in the sense of \cite{Bo98}, and denote by $\pi_1(M)_{I_F}$ the coinvariants. By \cite[Thm.~5.1]{PR08}, the group of connected components is given by
\[
\pi_0(\Fl_{\calM^\flat,\bar{k}})\,=\,\pi_1(M^\flat)_{I_{k\rpot{u}}}\,=\,\pi_1(M)_{I_F},
\]
where the last equality follows from the proof of Lemma \ref{Iwahorilem}. 
Note that $\pi_1(M)_{I_F}= {\rm Ind}^{\Gamma_{k}}_{\Gamma_{k_0}} \pi_1(M_0)_{I_K}$, cf.\,(\ref{Spec_Tor_Lem:eq1}). Since $\Gr_{{\calM},\breve{\calO}}\to \Spec(\breve{\calO})$ is ind-proper and $\breve{\calO}$ is Henselian, the natural map
\[
\pi_0(\Gr_{{\calM},\breve{\calO}})\,\overset{\simeq}{\longto}\, \pi_0(\Fl_{\calM^\flat,\bar{k}})
\]
is an isomorphism by \cite[Arcata; IV-2; Prop.~2.1]{SGA4}.  This shows that there is a decomposition into connected components
\begin{equation}\label{decom_connected_eq}
\Gr_{{\calM},{\breve{\calO}}} = \coprod_{\bnu\in \pi_1(M)_{I_F}} (\Gr_{{\calM},{\breve{\calO}}})_\bnu 
\end{equation}
such that $(\Gr_{{\calM},{\breve{\calO}}})_\bnu\otimes \bar{k}\simeq (\Fl_{\calM^\flat,{\bar{k}}})_\bnu$. By Lemma \ref{Spec_Tor_Lem}, the generic fiber decomposes as $(\Gr_{{\calM},{\breve{\calO}}})_\bnu\otimes \sF\simeq\coprod_{\nu\mapsto\bnu}(\Gr_{{M},\sF})_\nu$ where $\nu \in \pi_1({M})$ runs over the elements which map to $\bar{\nu}$ under the reduction map $\pi_1({M})\to \pi_1({M})_{I_F}$.

By Theorem \ref{BD_Gm_decom} i) and ii), it is easy to see that the closed immersion $\iota^0\co\Gr_{{\calM},\breve{\calO}}\to (\Gr_{{\calG},\breve{\calO}})^0$ is open on the underlying topological spaces (e.g.,\,its image is closed under generization), i.e., the image identifies each connected component of $\Gr_{{\calM},\breve{\calO}}$ with a connected component of $(\Gr_{{\calG},\breve{\calO}})^0$. Using Lemma \ref{basic_geo_BD} ii), we get an inclusion
\[
\pi_1(M)_{I_F}\,=\,\pi_0(\Gr_{{\calM},\breve{\calO}})\,\subset\,\pi_0 \left((\Gr_{{\calG},\breve{\calO}})^0\right)\,=\, \pi_0 \left((\Gr_{{\calG},\breve{\calO}})^\pm\right).
\]
For $\bar{\nu}\in \pi_1(M)_{I_F}$, we denote the corresponding connected component of $(\Gr_{{\calG},\breve{\calO}})^0$ (resp. $(\Gr_{{\calG},\breve{\calO}})^\pm$) by $(\Gr_{{\calG}})_{\bar{\nu}}^0$ (resp. $(\Gr_{{\calG}})_{\bar{\nu}}^\pm$).

Let $\rho$ denote the half-sum of the roots in $\Res_{K/F}(N)_\sF$ with respect to $\Res_{K/F}(T)_\sF$. For $\pi_1({M})\ni \nu\mapsto \bar{\nu}\in \pi_1({M})_{I_F}$, and $\dot{\nu}\in X_*(\Res_{K/F}(T))$ a lift of $\nu$, we define the integer $n_\nu:=\lan2\rho,\dot{\nu}\ran$ (resp. $n_\bnu:=\lan2\rho,\dot{\nu}\ran$) which is well-defined independent of the choice of $\dot{\nu}$, cf. \cite[(3.19)]{HaRi}. Note that we have $n_\nu=n_{\bnu}$ for all $\nu\mapsto \bnu$ by definition. For fixed $m\in \bbZ$, we consider the disjoint union
\[
(\Gr_{{\calG}})^{0}_m\defined\coprod_{\bar{\nu}}(\Gr_{{\calG}})^{0}_{\bar{\nu}}  \;\;\;\;\text{(resp. $(\Gr_{{\calG}})^{\pm}_m\defined\coprod_{\bar{\nu}}(\Gr_{{\calG}})^{\pm}_{\bar{\nu}}$)},
\]
where the disjoint sum is indexed by all $\bar{\nu}\in \pi_1(M)_{I_F}$ such that $n_{\bar{\nu}}=m$. The Galois action preserves the integers $n_{\bar{\nu}}$, and hence the ind-scheme $(\Gr_{{\calG}})^{0}_m$ (resp. $(\Gr_{{\calG}})^{\pm}_m$) is defined over $\calO$. Note that $(\Gr_{{\calG}})^{\pm}_m$ is the preimage of $(\Gr_{{\calG}})^{0}_m$ along $(\Gr_{{\calG}})^{\pm}\to (\Gr_{{\calG}})^{0}$. We obtain a decomposition as $\calO$-ind-schemes
\[
(\Gr_{{\calG}})^{0,c}\defined\coprod_{m\in \bbZ}(\Gr_{{\calG}})^{0}_m  \;\;\;\;\text{(resp. $(\Gr_{{\calG}})^{\pm,c}\defined\coprod_{m\in \bbZ}(\Gr_{{\calG}})^{\pm}_m$)}.
\]
Properties i) and ii) are immediate from the construction.
\end{proof}

 %Cohomology
\subsection{Cohomology of $\bbG_m$-actions on affine Grassmannians}\label{Cohomology_Gm_Act}  
The conventions are the same as in \cite[\S 3.4]{HaRi}. We fix a prime $\ell\not= p$, and an algebraic closure $\algQl$ of $\bbQ_\ell$. We fix once and for all $q^{1/2}\in \algQl$, and the square root of the cyclotomic character $\on{cycl}\co \Ga_F\to \bbZ_\ell^{\times}$ which maps any lift of the geometric Frobenius $\Phi_F$ to $q^{-1/2}$. The Tate twists are normalized such that the geometric Frobenius $\Phi_F$ acts on $\algQl(-{1/2})$ by $q^{1/2}$.

For a separated ind-scheme $X=\on{colim}_iX_i$ of finite type over a field (e.g.~ $F$) or a discrete valuation ring (e.g.~ $\calO_F$), we denote the bounded derived category $D_c^b(X)=D_c^b(X,\algQl)$ of $\algQl$-complexes with constructible cohomologies by
\[
D_c^b(X)\defined \on{colim}_iD_c^b(X_i,\algQl).
\]
There is the full abelian subcategory $\on{Perv}(X)\subset D_c^b(X)$ of perverse sheaves, cf. e.g. \cite[A.1]{Zhu} in the setting of ind-schemes. For a complex $\calA\in D_c^b(X)$, we denote for any $n\in\bbZ$ the shifted and twisted complex by
\[
\calA\lan n\ran\defined \calA[n](\nicefrac{n}{2}).
\]

Let us briefly recall the nearby cycles functor. Let $S=\Spec(\calO_F)$ with open (resp. closed) point $\eta=\Spec(F)$ (resp. $s=\Spec(k)$). Let $\bar{\eta}:=\Spec(\sF)\to \eta$ (resp. $\bar{s}:=\Spec(\bar{k})\to s$) denote the geometric point with Galois group $\Ga=\Gal(\bar{\eta}/\eta)$. Let $\bar{S}$ denote the integral closure of $S$ in $\bar{\eta}$. This gives rise to the seven tuple $(S,\eta,s,\bar{S},\bar{\eta},\bar{s},\Ga)$. Now if $X$ is an $\calO_F$-ind-scheme of ind-finite type, there is by \cite[Exp. XIII]{SGA7} (cf. also \cite[App]{Il94}) the functor of nearby cycles 
\begin{equation}\label{Nearby_Dfn}
\Psi_X\co D_c^b(X_\eta)\,\longto\, D_c^b(X_s\times_S\eta),
\end{equation}
where $D_c^b(X_s\times_S\eta)$ denotes the bounded derived category of $\algQl$-sheaves on $X_{\bar{s}}$ with constructible cohomologies, and with a continuous action of $\Ga$ compatible with the action on $X_{\bar{s}}$. The nearby cycles preserve perversity and restrict to a functor $\Psi_X\co \on{Perv}(X_\eta)\to \on{Perv}(X_s\times_S\eta)$. We refer the reader to \cite[\S10]{PZ13} for the extension to ind-schemes.

For a map of $\calO_F$-ind-schemes $f\co X\to Y$, the nearby cycles are functorial in the obvious way, cf. \cite[Exp. XIII, 1.2.7-1.2.9]{SGA7}. Further if $f$ is a nilpotent thickening, i.e., a closed immersion defined by an nilpotent ideal sheaf, then $f$ induces $\Psi_X\simeq \Psi_Y$. 

\subsubsection{Geometric Satake for Weil-restrictions} \label{Geo_Sat_Weil_Sec}
Recall the geometric Satake equivalence from \cite{Gi, Lu81, BD, MV07, Ri14a, RZ15, Zhu}. We work under the same conventions as in \cite[\S 3.4]{HaRi}, and we refer the reader to this reference for more details.

Let $G_0$ be a reductive group over $K$. We are interested in the geometric Satake isomorphism for the group ${G}=\Res_{K/F}(G_0)$. For a conjugacy class $\{\mu\}$ of geometric cocharacters in ${G}$, denote the inclusion of the open $L^+_z{G}_\sF$-orbit by 
\[
j\co \Gr_{{G}}^{\{\mu\}}\,\hookto\,\Gr_{{G}}^{\leq \{\mu\}},
\]
cf. \eqref{Schubert_Dfn}. The map $j$ is defined over the reflex field $E=E(\{\mu\})$. We define the normalized intersection complex by
\begin{equation}\label{IC_dfn}
\IC_{\{\mu\}}\defined j_{!*}\algQl\lan d_{\mu}\ran\,\in\, P(\Gr_{{G},E}),
\end{equation}
where $d_\mu$ denotes the dimension of $\Gr_{{G}}^{\leq \{\mu\}}$. The category $P_{L^+_z{G}}(\Gr_{{G}})$ of $L^+_z{G}$-equivariant perverse sheaves (cf.~ e.g.~ \cite[A.1]{Zhu} for equivariant perverse sheaves on ind-schemes) is generated by the intersection complexes \eqref{IC_dfn} and local systems concentrated on the base point $e_0\in \Gr_{{G}}(F)$. More precisely, every indecomposable object in $P_{L^+_z{G}}(\Gr_{{G}})$ is of the form
\begin{equation}\label{Satake_Objects}
\left(\oplus_{\ga\in\Ga_F/\Ga_E}\IC_{\ga\cdot \{\mu\}}\right)\otimes \calL,
\end{equation}
where $\calL$ is a $\algQl$-local system on $e_0=\Spec(F)$. The {\em Satake category} $\Sat_{{G}}$ is the full subcategory of $P_{L^+_z{G}}(\Gr_{{G}})$ generated by objects \eqref{Satake_Objects} where the local system $\calL$ is trivial over a finite field extension $\tilde{F}/F$. 

We view $\Ga_F$ as a pro-algebraic group, and we let $\Rep_{\algQl}^{\on{alg}}(\Ga_F)$ be the category of algebraic $\algQl$-representations of $\Ga_F$, i.e., finite dimensional representations which factor through a finite quotient of $\Ga_F$. There is the Tate twisted global cohomology functor
\begin{equation}\label{galoistwist}
\begin{aligned}
\om\co \Sat_{{G}} &\longto \Rep^{\on{alg}}_{\algQl}(\Ga_F) \\
\calA & \longmapsto \bigoplus_{i\in \bbZ}\on{H}^i(\Gr_{{G},\sF}, \calA_\sF)(\nicefrac{i}{2}).
\end{aligned}
\end{equation}
By the geometric Satake equivalence, the functor $\om$ can be upgraded to an equivalence of abelian tensor categories 
\begin{equation} \label{Satake_Equivalence}
\om\co \Sat_{{G}}\,\overset{\simeq}{\longto}\, \Rep_{\algQl}(^L{G})
\end{equation}
where $^L{G}={G}^\vee\rtimes \Ga_F$ denotes the $L$-group viewed as a pro-algebraic group over $\algQl$. The tensor structure on $\Sat_{{G}}$ is given by the convolution of perverse sheaves, cf.~ \S\ref{Central_Sheaves_Sec} below. The normalized intersection complex $\IC_{\{\mu\}}$ is an object in the category $\Sat_{{G}_E}$, and its cohomology $\om(\IC_{\{\mu\}})$ is under the geometric Satake equivalence \eqref{Satake_Equivalence} the $^L{G}_E:={G}^\vee\rtimes \Ga_E$-representation $V_{\{\mu\}}$ of highest weight $\{\mu\}$ defined in \cite[\S 6.1]{Hai14}, cf.~ \cite[Cor 3.12]{HaRi}. 

Let us describe the dual group ${G}^\vee=\Res_{K/F}(G_0)^\vee$ and the representation $V_{\{\mu\}}$ explicitly in terms of $G^{\vee}_0$. Of course, ${G}^\vee$ is canonically isomorphic to the product $\prod_{K\hookto \sF}G^\vee_0$, but the Galois action does not respect the factors in general.

%\footnote{For example, let $L/\bbQ_p$ be a cyclic degree $4$ extension, let $K\subset L$ be a degree $2$ subextension, and consider $G=\Res_{L/K}(\bbG_m)$. Then $\Res_{K/\bbQ_p}(G)^\vee\simeq \bbG_m^4$ with the Galois group acting by cyclic permutations. This action does not preserve the product decomposition $\Res_{L/K}(\bbG_m)^\vee\times \Res_{L/K}(\bbG_m)^\vee$. In general, if the $\Ga_K$-action on $G^\vee$ is trivial, then $\Res_{K/F}(G)^\vee$ is the product with the Galois group acting via permutations of the factors.}

Let $\underline{\Hom}_{\algQl}(\Ga_{F},G^\vee_0)$ be the sheaf of $\algQl$-scheme morphisms where again $\Ga_F$ is viewed as a pro-algebraic group. Then $\underline{\Hom}_{\algQl}(\Ga_{F},G^\vee_0)$ is a group functor, and the pro-algebraic group $\Ga_{K}$ acts on $\underline{\Hom}_{\algQl}(\Ga_{F},G^\vee_0)$ via $\algQl$-group automorphisms by the rule $(\ga*f)(g)= \ga(f(\ga^{-1}g))$. Following \cite[I.5]{Bo79}, we define the induced group as the $\Ga_K$-fixed point sheaf
\begin{equation}\label{Induced_Grp_Dfn}
I_{\Ga_K}^{\Ga_F}(G^\vee_0)\defined \underline{\Hom}_{\algQl}^{\Ga_{K}}(\Ga_{F},G^\vee_0),
\end{equation}
which is a group functor. 
Note that choosing any finite extension $\tilde{K}/K$ which is Galois over $F$ and splits $G_0$, we get an isomorphism of $\algQl$-groups
\begin{equation}\label{Ind_Lim_Dual}
\underline{\Hom}_{\algQl}^{\Ga_{\tilde{K}/K}}(\Ga_{\tilde{K}/F},G^\vee_0)\,\overset{\simeq}{\longto}\, I_{\Ga_K}^{\Ga_F}(G^\vee_0),
\end{equation}
where $\Ga_{\tilde{K}/K}=\Gal(\tilde{K}/K)$ (resp. $\Ga_{\tilde{K}/F}=\Gal(\tilde{K}/F)$). In particular, $I_{\Ga_K}^{\Ga_F}(G^\vee_0)$ is an algebraic group, and is the colimit indexed by the filtered direct system  \eqref{Ind_Lim_Dual} indexed by the splitting fields $\tilde{K}$. 
In this way, we get as in \cite[I.5]{Bo79} an $\Ga_{F}$-equivariant isomorphism of algebraic $\algQl$-groups
\begin{equation}\label{Dual_Group}
{G}^\vee\,\simeq\, I_{\Ga_K}^{\Ga_F}(G^\vee_0).
\end{equation}

Let us turn to the representation $V_{\{\mu\}}$. We write the conjugacy class as $\{\mu\}=\left(\{\mu_\psi\}\right)_\psi$ according to ${G}_\sF\simeq \prod_{\psi\co K\hookto\sF}G_0\otimes_{\psi, K} \sF$. 
The reflex field $E$ of $\{\mu\}$ is the intersection (inside $\sF$) of the reflex fields $E_\psi$ of $\{\mu_\psi\}$. 
For each $\psi$, let $V_{\{\mu_\psi\}}$ the representation of $G^\vee_0$ of highest weight $\{\mu_\psi\}$ where we view $\{\mu_\psi\}$ as a Weyl orbit in the dual torus $X^*(T^\vee)$. The following lemma is immediate from the construction, and left to the reader.

\begin{lem}\label{Induced_Rep_Lem}
The $\prod_{\psi}G^\vee_0$-representation $\boxtimes_{\psi} V_{\{\mu_\psi\}}$ uniquely extends to the $^L{G}_E={G}^\vee\rtimes \Ga_E$ representation $V_{\{\mu\}}$ defined above.
\end{lem}
\hfill\ensuremath{\Box} 

%{\cb I think for the constant term morphism we can simply refer to the manuscript \cite{HaRi} when we use it in the analogue of \cite[Lem.~7.9]{HaRi}. Please remove if you agree.}

\subsubsection{Constant terms commute with nearby cycles} 
We proceed with the notation as in \S \ref{Geometry_Gm_Act}, and view the cocharacter $\chi$ as in \eqref{chi_dfn}. Combining Theorem \ref{BD_Gm_decom} and Proposition \ref{conn_comp_BD} from the previous section, we have constructed a commutative diagram of $\calO_F$-ind-schemes

\begin{equation}\label{Gm_decom_BD_restrict}
\begin{tikzpicture}[baseline=(current  bounding  box.center)]
\matrix(a)[matrix of math nodes, 
row sep=1.5em, column sep=2em, 
text height=1.5ex, text depth=0.45ex] 
{\Gr_{{\calM}} & \Gr_{{\calP}^\pm} & \Gr_{{\calG}} \\ 
(\Gr_{{\calG}})^{0,c}& (\Gr_{{\calG}})^{\pm,c}& \Gr_{{\calG}}, \\}; 
\path[->](a-1-2) edge node[above] {$q^\pm$}  (a-1-1);
\path[->](a-1-2) edge node[above] {$p^\pm$}  (a-1-3);
\path[->](a-2-2) edge node[below] {}  (a-2-1);
\path[->](a-2-2) edge node[below] {} (a-2-3);
\path[->](a-1-1) edge node[left] {$\iota^{0,c}$} (a-2-1);
\path[->](a-1-2) edge node[left] {$\iota^{\pm,c}$} (a-2-2);
\path[->](a-1-3) edge node[left] {$\id$} (a-2-3);
\end{tikzpicture}
\end{equation}
The generic fiber of \eqref{Gm_decom_BD_restrict} is \eqref{BD_Gm_decom:diag1}, and the special fiber of \eqref{Gm_decom_BD_restrict} is \eqref{BD_Gm_decom:diag2}. The maps $\iota^{0,c}\co\Gr_{{\calM}}\hookto (\Gr_{{\calG}})^{0,c}$ and $\iota^{\pm,c}\co\Gr_{{\calP}^\pm}\hookto (\Gr_{{\calG}})^{\pm,c}$ are nilpotent thickenings by Proposition \ref{conn_comp_BD}, and we may and do identify their derived categories of $\ell$-adic complexes. Then there is a natural isomorphism of functors $D_c^b(\Gr_{{M}})\to D_c^b(\Fl_{\calM^\flat}\times_S\eta)$,
\begin{equation}\label{nearbyiso}
\Psi_{\Gr_{{\calM}}}\simeq \Psi_{(\Gr_{{\calG}})^{0,c}}.
\end{equation}
We write $\Psi_{{\calG}}=\Psi_{\Gr_{{\calG}}}$ (resp. $\Psi_{{\calM}}=\Psi_{\Gr_{{\calM}}}$) in what follows. Since $\iota^{0,c}$ and $\iota^{\pm,c}$ are nilpotent thickenings, Proposition \ref{conn_comp_BD} gives us a decomposition 
\[
q^\pm=\coprod_{m\in\bbZ}q_m^\pm\co  \;\Gr_{{\calP}^\pm}\,=\,\coprod_{m\in\bbZ}\Gr_{{{\calP}^\pm},m}\,\longto\, \coprod_{m\in\bbZ}\Gr_{{\calM},m}\,=\,\Gr_{{\calM}},
\]
according to the choice of the parabolic $P^{\pm}$. We use the generic and the special fiber of diagram \eqref{Gm_decom_BD_restrict} to define normalized geometric constant term functors as follows.

\begin{dfn} \label{gctnormal}
We define the functor $\on{CT}_{{M}}\co D_c^b(\Gr_{{G}})\to D_c^b(\Gr_{{M}})$ (resp.\,$\on{CT}_{\calM^\flat}\co D_c^b(\Fl_{\calG^\flat}\times_S\eta)\to D_c^b(\Fl_{\calM^\flat}\times_S\eta)$) as the shifted pull-push functor
\[
\on{CT}_{{M}}\defined \bigoplus_{m\in\bbZ}(q^+_{m,\eta})_!(p^+_\eta)^*\langle m\rangle \;\;\;\text{(resp. $\on{CT}_{\calM^\flat}\defined \bigoplus_{m\in\bbZ}(q^+_{m,s})_!(p^+_s)^*\langle m\rangle$).}
\]
\end{dfn}

As in \cite[Thm.~6.1, (6.11)]{HaRi}, the functorialities of nearby cycles give a transformation of functors $D_c^b(\Gr_{{G}})\to D_c^b(\Fl_{\calM^\flat}\times_S\eta)$ as
\begin{equation}\label{commctnearby}
\on{CT}_{\calM^\flat}\circ \Psi_{{\calG}}\longto \Psi_{{\calM}}\circ \on{CT}_{{M}}.
\end{equation}

\begin{thm}\label{comm_ctnearby_thm}
The transformation \eqref{commctnearby} is an isomorphism of functors $\Sat_{{G}}\to D_c^b(\Fl_{\calM^\flat}\times_S\eta)$. In particular, for every $\calA\in \Sat_{{G}}$, the complex $\on{CT}_{\calM^\flat}\circ \Psi_{{\calG}}(\calA)$ is naturally an object in the category $\on{Perv}_{L^+{\calM^\flat}}(\Fl_{\calM^\flat}\times_S\eta)$.
\end{thm}
\begin{proof} Every object in $\Sat_{{G}}$ is $\bbG_m$-equivariant. In view of Theorem \ref{BD_Gm_decom} and \eqref{Gm_decom_BD_restrict}, the extension of the method used in \cite[Thm.~6.5]{HaRi} to this more general situation is obvious. We do not repeat the arguments. 
\end{proof}

%%%%%%Constant terms for tori%%%%%%%%%
\subsection{Constant terms for tori} 
We aim to make Theorem \ref{comm_ctnearby_thm} more explicit in the special case where $ M= T$ is a torus, cf.~Theorem \ref{ram_cons_term_thm} and Corollary \ref{AB_cor}. We keep the notation as in \S \ref{Cohomology_Gm_Act}. 

Let $\Sat_{\calT^\flat}\subset \on{Perv}_{L^+{\calT^\flat}}(\Fl_{\calT^\flat}\times_S\eta)$ denote the semi-simple full subcategory defined as in \cite[Def 5.10]{Ri16a}. By \cite[Thm.~5.11]{Ri16a}, the category $\Sat_{\calT^\flat}$ has a Tannakian structure with tensor structure given by the convolution product, and with fiber functor given by the global sections functor $\om_{\calT^\flat}\co \Sat_{\calT^\flat}\to \Rep_{\algQl}(\Ga_F)$, $\calA\mapsto \on{H}^0(\Fl_{\calT^\flat,\bar{k}}, \calA_{\bar{k}})$. Note that $\Fl_{\calT^\flat}$ is ind-finite, and hence there is no higher cohomology and the convolution product is given by the usual tensor product. Further, for every $\calA\in \Sat_{\calT^\flat}$ the $\Ga_F$-action on $\om_{\calT^\flat}(\calA)$ factors by definition through a finite quotient.

\begin{lem}\label{Tannaka_equiv} The functor $\om_{\calT^\flat}$ can be upgraded to an equivalence of Tannakian categories
\[
\Sat_{\calT^\flat}\,\overset{\simeq}{\longto}\, \Rep_{\algQl}(^L T_r),
\]
where $^L T_r=( T^\vee)^{I_F}\rtimes \Ga_F$ viewed as a pro-algebraic subgroup of $^L T$. 
Here the subscript $(\str)_r$ stands for `ramified'.
\end{lem}
\begin{proof} By Lemma \ref{Spec_Tor_Lem}, there are $\Gal(\bar{k}/k)$-isomorphisms of abelian groups
\[
\Fl_{\calT^\flat}(\bar{k})\,{\simeq}\, X_*(T^\flat)_{I_{k\rpot{u}}}\,{\simeq}\, X_*( T)_{I_F}\,{\simeq}\, X^*\left(( T^\vee)^{I_F}\right),
\]
where the equivariance of the last isomorphism holds by construction of the dual torus. This induces a $\Gal(\bar{k}/k)$-equivariant isomorphism of $\bar{k}$-schemes 
\begin{equation}\label{Tannaka_equiv_iso}
(\Fl_{\calT^\flat,\bar{k}})_\red\,\simeq\, \underline{X^*\left(( T^\vee)^{I_F}\right)}.
\end{equation}
By definition, the objects in $\Sat_{\calT^\flat}$ are finite dimensional $\algQl$-vector spaces on \eqref{Tannaka_equiv_iso} (viewed as complexes concentrated in cohomological degree $0$) together with an action of $\Ga_F$ which is equivariant over the base, and which factors through a finite quotient. The lemma follows from this description.  
\end{proof}

The following proposition is the analogue of \cite[Thm.~10.18, 10.23]{PZ13} in the special case of a torus.

\begin{prop}\label{Torus_Prop} 
There is a commutative diagram of Tannakian categories
\[
\begin{tikzpicture}[baseline=(current  bounding  box.center)]
\matrix(a)[matrix of math nodes, 
row sep=2em, column sep=2.5em, 
text height=1.5ex, text depth=0.45ex] 
{ \Sat_{ T}& \Sat_{\calT^\flat} \\ 
\Rep_{\algQl} ({^L{ T}})& \Rep_{\algQl}( {^L{ T}}_r ), \\}; 
\path[->](a-1-1) edge node[above] {$\Psi_{ \calT}$}  (a-1-2);
\path[->](a-2-1) edge node[above] {$\res$}  (a-2-2);
\path[->](a-1-1) edge node[left] {$\om_{{T}}$} node[right] {$\simeq$} (a-2-1);
\path[->](a-1-2) edge node[left] {$\om_{\calT^\flat}$} node[right] {$\simeq$} (a-2-2);
\end{tikzpicture}
\]
where $\res$ denotes the restriction of representations along the inclusion ${^L{ T}}_r\subset {^L{ T}}$.
\end{prop}
\begin{proof} This is a reformulation of Lemma \ref{Spec_Tor_Lem} as follows. In view of \eqref{Satake_Equivalence} and Lemma \ref{Tannaka_equiv} the diagram is well defined, and it suffices to prove the commutativity. Let $f\co \Gr_{ T}\to \Spec(\calO_F)$ denote the structure map. Since $f$ is ind-proper, there is a $\Ga_F$-equivariant isomorphism
\[
\Psi_{\calO_F}\circ f_{\bar{\eta},*}\,\overset{\simeq}{\longto}\,f_{\bar{s},*}\circ \Psi_{ \calT},
\]
and passing to the $0$-th cohomology defines a $\Ga_F$-equivariant isomorphism $\al\co \res \circ \om_{ T}\simeq \om_{\calT^\flat}\circ \Psi_{ \calT}$. We have to show that $\al$ is a map of ${^L{ T}}_r$-representations. As we already know the $\Ga_F$-equivariance, it is enough to check that $\al$ is a map of $({ T^\vee})^{I_F}$-representations, i.e., respects the grading by $X^*(({ T}^\vee)^{I_F})=X_*( T)_{I_F}$ on the underlying $\algQl$-vector spaces. By \eqref{decom_connected_eq}, we have a decomposition into connected components 
\[
\Gr_{ \calT}\otimes \calO_{\breve F} \,=\,\coprod_{\bar{\nu} \in X_*( T)_{I_F}}(\Gr_{ \calT,\calO_{\breve F}})_{\bar{\nu}},
\]
where $(\Gr_{ \calT,\calO_{\breve F}})_{\bar{\nu}}\otimes\bar{k}=\{\bar{\nu}\}$ and $(\Gr_{ \calT,\calO_{\breve F}})_{\bar{\nu}}\otimes\bar{F}=\coprod_{\nu\mapsto \bar{\nu}}\{\nu\}$ on the underlying reduced subschemes, cf.~also Lemma \ref{Spec_Tor_Lem}. The proposition follows from the fact that nearby cycles of a disjoint sum are computed as the sum of the single components. \iffalse {\cb and the fact that the flat closure of $\{\nu\}$ is smooth over the base, which enures that, ignoring Galois actions, the nearby cycles of $\bar{\mathbb Q}_{\ell}$ on $\{\nu\}$ is $\bar{\mathbb Q}_\ell$ over $\{\bar{\nu}\}$.}\fi
\end{proof}

\begin{rmk} It would be interesting to see whether the analogue of Proposition \ref{Torus_Prop} for more general {\it very} special parahoric group schemes as in \cite[Thm.~10.18, 10.23]{PZ13} holds true.
\end{rmk}

Combining Proposition \ref{Torus_Prop} with Theorem \ref{comm_ctnearby_thm}, we arrive as in \cite[\S6.2]{HaRi} at the following theorem which is the analogue of \cite[Thm.~4]{AB09} in our situation.

\begin{thm} \label{ram_cons_term_thm}
i\textup{)} For every $\calA\in \Sat_{ G}$, one has $\on{CT}_{\calT^\flat}\circ \Psi_{ \calG}(\calA)\in \Sat_{\calT^\flat}$.\smallskip\\
ii\textup{)} The functor $\on{CT}_{\calT^\flat}\circ \Psi_{ \calG}\co \Sat_{ G}\to\Sat_{\calT^\flat}$ admits a unique structure of a tensor functor together with an isomorphism $\om_{\calT^\flat}\circ \on{CT}_{\calT^\flat} \circ\Psi_{ \calG}\simeq \om_{ G}$. Under the geometric Satake equivalence, it corresponds to the restriction of representations $\res\co \Rep_{\algQl}({^L G})\to \Rep_{\algQl}({^L T}_r)$ along the inclusion ${^L T}_r\subset {^L G}$.  
\end{thm}
\hfill{\ensuremath{\Box}}
\medskip

We now apply Theorem \ref{ram_cons_term_thm} in a special case. For more details, we refer to \cite[\S6.2.1]{HaRi} which is analogous. Assume $F=\breve F$, and hence that $K/F$ is totally ramified. Let $\chi\co \bbG_{m,K}\to A_0\subset G_0$ be a regular cocharacter, i.e., its centralizer $M_0=T_0$ is a maximal torus, and let the parahoric $\calO_K$-group scheme $\calG_0$ be an Iwahori. Hence, ${\calG}=\Res_{\calO_K/\calO_F}(\calG_0)$ is an Iwahori $\calO_F$-group scheme as well, cf.~Proposition \ref{Parahoric_Weil_Prop}. There is a decomposition into connected components
\[
(\Fl_{\calG^\flat})^+ \,=\,\coprod_{w\in W}(\Fl_{\calG^\flat})^+_{w},
\]
where $W=W( G,  A, F)=W(G,A,K)$ is the Iwahori-Weyl group, cf.~Lemma \ref{IW_lem}. Let $\La_{ T}= T(F)/\calT(\calO_F)\subset W$ be the subset of ``translation'' elements. Let $X_*( T)_{I_F}\simeq \La_{ T}, \bar{\la}\mapsto t^{\bar{\la}}$ be the isomorphism given by the Kottwitz map. Let $2\rho\in X^*( T)$ be the sum of the positive roots contained in the positive Borel $ B^+$ of $ G_\sF$ determined by $\chi$. Then the integer $\lan2\rho,\bar{\la}\ran:=\lan2\rho,\la\ran$ is well defined independent of the choice of $\la\in X_*( T)$ with $\la\mapsto \bar{\la}$.

\begin{cor} \label{AB_cor}
Let $V\in \Rep_{\algQl}({{G}^\vee})$, and denote by $\calA_V\in\Sat_{ G,\sF}$ the object with $\om_{ G,\sF}(\calA_V)=V$. As $\algQl$-vector spaces, the compactly supported cohomology groups are given by the equality
\[
\bbH_c^i((\Fl_{\calG^\flat})_{w}^+,\Psi_{ \calG}(\calA_V))= \begin{cases} V({\bar{\la}}) &\text{if $w=t^{\bar{\la}}$ and $i=\langle2\rho,\bar{\la}\rangle$;}\\ 0 &\text{else,}\end{cases}
\]
where $V({\bar{\la}})$ is the $\bar{\la}$-weight space in $V|_{({ T^\vee})^{I_F}}$.
\end{cor}
\hfill{\ensuremath{\Box}}

%%%%%%%Special fibers of local models%%%%%%
\subsection{Special fibers of local models}
Levin proved in \cite[Thm.\,2.3.5]{Lev16} the analogue of the following theorem in the special case where $p \nmid |\pi_1(G_\der)|$.  As in \cite[\S6.2, 6.3]{HaRi}, Corollary \ref{AB_cor} can be used to obtain this result on the special fibers of local models, with no hypothesis on $p$. We do not need this result for the proof of our Main Theorem, but include it for completeness: together with the corresponding result in \cite{HaRi}, we can conclude that the admissible sets ${\rm Adm}^{\bf f}_{\{\mu\}}$ parametrize the strata in the special fiber of $M_{\{\mu\}}$ {\em for all known local models} $M_{\{\mu\}}$.  

The following is precisely the analogue of \cite[Thm.\,6.12]{HaRi} in the current Weil restriction setting. 
We will assume for simplicity here that $K_0 = F$; a similar result holds without this assumption. 
Since $\breve{K} = K\breve{F}$, we may work over $F = \breve{F}$, so that $K = \breve{K}$ and $k = \bar{k}$. The special fiber $M_{\{\mu\}, k}$ and the relevant Schubert varieties live in the affine flag variety attached to {\em equal characteristic} analogues $G^\flat = G^\flat_{k(\!(u)\!)}, A^\flat=S^\flat$ defined in \eqref{analogues}, and by Lemmas \ref{IW_lem} and \ref{Iwahorilem} there is an identification of Iwahori-Weyl groups
\[
W=W( G,  A, F)=W(G_0,A_0,K)=W(G^\flat,A^\flat,k\rpot{u}).
\]
For $w\in W$, we define the Schubert varietiy ${\mathcal Fl}_{\calG^\flat}^{\leq w}$ exactly as in \cite[$\S3.2$]{HaRi}.

\begin{thm}\label{Weil_specialfiber}
The smooth locus $(M_{\{\mu\}})^{\on{sm}}$ is fiberwise dense in $M_{\{\mu\}}$, and on reduced subschemes
\[
(M_{\{\mu\},k})_\red=\bigcup_{w\in\Adm_{\{\mu\}}^\bbf}{\mathcal Fl}_{\calG^\flat}^{\leq w}.
\]
In particular, the special fiber $M_{\{\mu\},k}$ is generically reduced.
\end{thm}

\begin{proof}
We may imitate the proof of \cite[Thm.\,6.12]{HaRi}. First we follow the method of \cite[Lem.\,3.12]{Ri16b} to prove ${\rm Adm}^{\bf f}_{\{\mu\}} \subseteq {\rm Supp}^{\bf f}_{\{\mu\}} := {\rm Supp}\, \Psi_{\calG_{\bf f}}({\rm IC}_{\{\mu\}})$, using our Lemma \ref{Spec_Tor_Lem} in place of \cite[Lem.\,2.21]{Ri16b}.  

Also as in \cite[Thm.\,6.12]{HaRi}, we reduce to the case where ${\bf f} = {\bf a}$. Then is it enough to show that if $w \in {\rm Supp}^{\bf a}_{\{\mu\}}$ is maximal, then $w \in {\rm Adm}^{\bf a}_{\{\mu\}}$. Now we choose a regular cocharacter $\chi\co \bbG_{m,K}\to A_0\subset G_0$, and use Corollary \ref{AB_cor} as follows. As $\algQl$-vector spaces, we have
\[
\bbH^*_c((\Fl_{\calG^\flat})_{w}^+, \Psi_{ \calG}(\IC_{\{\mu\}}))\not = 0,
\]
because $\Fl_{\calG^\flat}^{\leq w}\cap (\Fl_{\calG^\flat})_{w}^+\subset\Fl_{\calG^\flat}^w$ is non-empty by \cite[Lem.~6.10]{HaRi}, and because up to shift and twist $\Psi_{ \calG}(\IC_{\{\mu\}})|_{\Fl^w_{\calG^\flat}}=\algQl^d$ for some $d>0$ since $w \in {\rm Supp}^{\bf a}_{\{\mu\}}$ is maximal. Thus, Corollary \ref{AB_cor} applies to show $w=t^{\bar{\la}}$ for some $\bar{\la}\in X_*( T)_{I_F}$ which is a weight in $V_{\{\mu\}}|_{({G}^\vee)^{I_F}}$. As in \cite[Thm\,6.12]{HaRi}, we can conclude that $w= t^{\bar{\lambda}} \in {\rm Adm}^{\bf a}_{\{\mu\}}$ by citing \cite[Thm.\,4.2\,and\,(7.11-12)]{Hai18}.

\iffalse
By flatness of $M_{\{\mu\}}$ over $\mathcal O$, there exists $x \in M_{\{\mu\}}(\mathcal O)$ with image $\bar{x} = w \in M_{\{\mu\}}(k)$. Since $L^\mathcal G_{F}$ acts on $M_{\{\mu\}}(F) = M_{\{\mu\}}(\mathcal O)$, by the Cartan decomposition we may assume $x = z^{\lambda'}$ for some dominant $\lambda' \in X_*(T)$. Let $\mu$ be a dominant representative of $\{\mu\}$ (recall $G_{\breve F}$ is quasi-split). We have $\lambda' \leq \mu$ is the usual dominance order, and it follows from ????? that 
{\cm under construction -- the above doesn't work and we have to use more than Theorem \ref{comm_ctnearby_thm}}
\fi

\end{proof}

\subsection{Central sheaves} \label{Central_Sheaves_Sec}
We recall some facts on central sheaves which will be used in what follows. We proceed with the notation as in \S \ref{Local_Models_Sec}. Let $\on{Perv}_{L^+\calG^\flat}(\Fl_{\calG^\flat}\times_S\eta)$ be the category of $L^+\calG^\flat$-equivariant perverse sheaves compatible with a continuous Galois action, cf.~ \cite[Def 10.3]{PZ13}. 

Recall that for objects in $\on{Perv}_{L^+\calG^\flat}(\Fl_{\calG^\flat}\times_S\eta)$ there is the convolution product defined by Lusztig \cite{Lu81}. Consider the convolution diagram
\[
\Fl_{\calG^\flat}\times\Fl_{\calG^\flat}\overset{\;q}{\leftarrow} L{\calG^\flat}\times \Fl_{\calG^\flat} \overset{p}{\to} L{\calG^\flat}\times^{L^+{\calG^\flat}}\Fl_{\calG^\flat}=:\Fl_{\calG^\flat}\tilde{\times}\Fl_{\calG^\flat}\overset{m\;}{\to} \Fl_{\calG^\flat}. 
\]
For $\calA,\calB\in \on{Perv}_{L^+{\calG^\flat}}(\Fl_{\calG^\flat}\times_S\eta)$, let $\calA\tilde{\times}\calB$ be the (unique up to canonical isomorphism) complex on $\Fl_{\calG^\flat}\tilde{\times}\Fl_{\calG^\flat}$ such that $q^*(\calA\boxtimes \calB)\simeq p^*(\calA\tilde{\times} \calB)$. By definition
\begin{equation}\label{convol_sheaf}
\calA\star \calB\defined m_*(\calA\tilde{\times} \calB)\in D_{c}^b(\Fl_{\calG^\flat}\times_S\eta, \algQl).
\end{equation}
In the following, we consider $P_{L^+{\calG^\flat}}(\Fl_{\calG^\flat})$ as a full subcategory of $P_{L^+{\calG^\flat}}(\Fl_{\calG^\flat}\times_S\eta)$. 

Let $W=W(G,A,K)=W(G^\flat,A^\flat,F^\flat)$ be the associated Iwahori-Weyl group, cf.\,Lemma \ref{Iwahorilem}. For each $w\in W$, the associated Schubert variety $\Fl_{\calG^\flat}^{\leq w}\subset \Fl_{\calG^\flat}$ is defined over $k_F$. Let $j\co \Fl_{\calG^\flat}^{w}\hookto \Fl_{\calG^\flat}^{\leq w}$, and denote by $\IC_w=j_{!*}(\algQl[\dim(\Fl_{\calG^\flat}^{w})])$ the intersection complex. We have the functor of nearby cycles
\[
\Psi_{{\calG}}\co \on{Perv}_{L^+_z{G}}(\Gr_{{G}})\longto \on{Perv}_{L^+{\calG^\flat}}(\Fl_{\calG^\flat}\times_{S}\eta).
\]
The next theorem follows from \cite[Thm.~10.5]{PZ13} if $K/F$ is tamely ramified, and from \cite[Thm.~5.2.10]{Lev16} in general:

\begin{thm}[Gaitsgory, Zhu, Pappas-Zhu, Levin] \label{CentralNB_Thm}
For each $\calA\in \on{Perv}_{L^+_z{G}}(\Gr_{{G}})$, and $w\in W$, both convolutions $\Psi_{{\calG}}(\calA)\star\IC_w$, $\IC_w\star\Psi_{{\calG}}(\calA)$ are objects in $P_{L^+{\calG^\flat}}(\Fl_{\calG^\flat}\times_S\eta)$, and as such there is a canonical isomorphism
\[
\Psi_{{\calG}}(\calA)\star\IC_w\simeq \IC_w\star\Psi_{{\calG}}(\calA).
\]
\end{thm}
\begin{proof} If $\calA= \IC_{\{\mu\}}$ where $\{\mu\}$ is a class which is defined over $F$, then the theorem is a special case of \cite[Thm.~5.2.10]{Lev16} which follows the method of \cite[Thm.~10.5]{PZ13}. However, the proof given there works for general objects $\calA\in P_{L^+_z{G}}(\Gr_{{G}})$, and only uses that the support $\Supp(\calA)$ is finite dimensional and defined over $F$. We do not repeat the arguments here.    
\end{proof}

%%%%%%%%%%%%%%%%Test functions for Local models of Weil-restricted groups%%%%%%%%%%%%%%%%%%%%%%%%%
\section{Test functions for Weil restricted local models}\label{Test_Functions_Sec} 

%Formulation
%\subsection{The conjecture}\label{Conjecture_Sec}

\subsection{Preliminaries}\label{prelim_sec}
Recall we let $G = {\Res}_{K/F}(G_0)$ and $^LG = G^\vee \rtimes \Gamma_F$. 
Recall that $\{\mu\}$ is defined over a field $E$, a separable field extension of $F$ which is a possibly nontrivial extension of the reflex field, and that $E_0/F$ is the maximal unramified subextension of $E/F$. 
We have $V_{\{\mu\}} \in {\rm Rep}(\,^L{G}_E)$ and $I(V_{\{\mu\}}) \in {\rm Rep}(\,^L{G}_{E_0})$, where $I(V) := {\rm Ind}_{\,^L{G}_{E}}^{\,^L{G}_{E_0}}(V)$ for $V \in {\rm Rep}(\,^L{G}_E)$. 
Writing $\calG := \calG_{\bf f}$ and $\calG_0 := \calG_{\bf f_0}$, the parahoric group scheme of $G = {\Res}_{K/F}(G_0)$ is given by ${\mathcal G} = {\Res}_{\mathcal O_K/\mathcal O_F}(\mathcal G_0)$ by Corollary \ref{Parahoric_Cor}.

Because the representation $I(V_{\{\mu\}})$ is ``defined over $E_0$'' (not $F$), it is convenient to reformulate the test function conjecture after base-changing all geometric objects from $\calO_F$ to $\calO_{E_0}$. This ultimately allows us to reduce to the case where $E_0 = F$ (see end of $\S\ref{prelim_sec}$, and $\S\ref{reduction_to_MT_subsec}$ below).  The next few lemmas are ingredients toward this reduction.

\begin{lem} \label{tensor_fields} The following statements hold.
\begin{enumerate}
\item[i)] We may write $K_0 \otimes_F E = \prod_j E_j$ and $K_0 \otimes_F E_0 = \prod_j E_{j,0}$, where $E_j/K_0$ is a finite extension of fields with maximal unramified subextension $E_{j,0}/K_0$, and where $j$ ranges over the finite index set of $\Gamma_{K_0}$-orbits of $F$-embeddings $E_0 \rightarrow \bar{K}_0$, i.e., over the set $\Gamma_{E_0} \backslash \Gamma_F /\Gamma_{K_0}$.  Similarly for rings of integers we have $\mathcal O_{K_0} \otimes_{\mathcal O_F} \mathcal O_{E_0} = \prod_{j} \mathcal O_{E_{j,0}}$.  Furthermore, the inertia groups satisfy $I_E = I_{E_j} \subset I_{E_0} = I_{E_{j,0}}$. 
\item[ii)] $K \otimes_F E_0 = \prod_{j} K E_{j,0}$.
\item[iii)] The canonical map $\Gamma_{E_0} \backslash \Gamma_F /\Gamma_{K} \rightarrow \Gamma_{E_0} \backslash \Gamma_F /\Gamma_{K_0}$ is a bijection.
\end{enumerate}
\end{lem}

\begin{proof}
Write $K_0 \otimes_F E_0 = \prod_{j} E_{j,0}$ and $E = E_0[X]/(Q)$ where $Q$ is an Eisenstein polynomial over $\mathcal O_{E_0}$. Each extension $E_{j,0}/F$ is unramified, and so $Q$ remains an Eisenstein polynomial in the over-field $E_{j,0}$ of $E_0$. As $K_0 \otimes_F E = \prod_j E_{j,0}[X]/(Q) =: \prod_j E_j$, it follows that $E_j/E_{j,0}$ is totally ramified and that $E_j = E E_{j,0}$, from which it follows that $E_j/E$ is unramified and hence $I_E = I_{E_j}$. 

Since $K/K_0$ is totally ramified, $K \otimes_{K_0} E_{j,0} = K E_{j,0}$, which implies ii).

Abstractly $K_0 \otimes_F E$ is a product of fields indexed by the set $\Gamma_{E} \backslash \Gamma_F/ \Gamma_{K_0}$, and this set coincides with $\Gamma_{E_0} \backslash \Gamma_F /\Gamma_{K_0}$ by the above argument.  Interchanging the roles of $E$ and $K$, we also get the bijection in iii).
\end{proof}

\begin{lem} \label{Res_G_E_0} We have $G_{E_0} = \prod_j \Res_{KE_{j,0}/E_0} G_{0,KE_{j,0}}$ and $\calG_{\calO_{E_0}} = \prod_j \Res_{\calO_{KE_{j,0}}/\calO_{E_0}} \calG_{0,\calO_{KE_{j,0}}}$.
\end{lem}

\begin{proof}
This is a consequence of the compatibility of Weil restriction of scalars with base change along the ring extension $F \rightarrow E_0$ (resp.\,$\mathcal O_{F} \rightarrow \mathcal O_{E_0}$) and Lemma \ref{tensor_fields} i) and ii). 
\end{proof}
%see e.g.\,\cite[A.2.7]{Oe84}).}

%{\cm Write $I^{F}_K (\,^LG_{0}^\circ) = \{ \psi: \Gamma_{F} \to \, ^LG_{0}^\circ \,  | \, \psi(\gamma' \gamma) = \gamma' \cdot \psi(\gamma), \forall \gamma' \in \Gamma_F,\, \gamma' \in \Gamma_K\}$.  }

By \cite[I.5]{Bo79} (cf.~\eqref{Dual_Group}) there are natural identifications
\begin{align*}
G^\vee &= I^{F}_K (G_0^\vee) \\
^LG &= I^{F}_K (G_0^\vee) \, \rtimes \, \Gamma_F,
\end{align*}
where we abbreviate $I_K^F:=I_{\Ga_K}^{\Ga_F}$ for the induction functor.
Using Lemma \ref{Res_G_E_0} we obtain the following lemma.

\begin{lem} \label{dual_res}  We have an identification
$$
^LG_{E_0} = \Big(\prod_j I^{E_0}_{E_{j,0}} I^{E_{j,0}}_{KE_{j,0}} (G^\vee_{0,KE_{j,0}})\Big) \, \rtimes \, \Gamma_{E_0}.
$$
\end{lem}

\iffalse
\begin{lem} \label{Res_vs_bc} We have ${G}_{E_0} = {\Res}_{K'/E_0}(G_{0,K'})$ and ${\mathcal G}_{\mathcal O_{E_0}} = {\Res}_{\mathcal O_{K'}/\mathcal O_{E_0}}(\mathcal G_{0,\mathcal O_{K'}})$. 
\end{lem}

\begin{proof}
This is a consequence of the compatibility of Weil restriction of scalars with base change along the finite free ring extension $F \rightarrow E_0$ (resp.\,$\mathcal O_{F} \rightarrow \mathcal O_{E_0}$). The compatibility is immediate from the definition of Weil restriction (see e.g.\,\cite[A.2.7]{Oe84}).
\end{proof}
\fi

Let $X = \mathbb A^1_{\calO_{K_0}} = {\rm Spec}(\calO_{K_0}[u])$ and $D = \{Q = 0\}$, viewed as a relative effective Cartier divisor on $X$ which is finite and flat over ${\rm Spec}(\calO_F)$. The following lemma helps us to determine ${\rm Gr}_{(X/\calO_F, \underline{\calG}_0, D)} \otimes_{\calO_F} \calO_{E_0}$; it handles the special case where $K/F$ is totally ramified.
\begin{lem}\label{Res_vs_bc_2} Assume $K_0 = F$, and let $K' = E_0K$, which is the maximal unramified subextension of $KE/K$, and let $\mathcal O_{K'} = \mathcal O_{K} \otimes_{\mathcal O_F} \mathcal O_{E_0}$ be its ring of integers. Since $\calO_{K_0} = \calO_F$, note $\underline{\calG}_0$ is defined over $\calO_F[u]$. Then we have identifications
\begin{enumerate}
\item[i)] $\underline{\calG}_0 \otimes_{\calO_F[u]} \calO_{E_0}[u] = \underline{\calG_{\calO_{K'}}}_0 =: \underline{\calG}_{0,{\calO_{E_0}}}$;
\item[ii)] $(L_D\ucG_0) \otimes_{\calO_F} \calO_{E_0} = L_{D_{\calO_{E_0}}} \underline{\calG}_{0,{\calO_{E_0}}}$ \textup{(}and similarly for $L^+_{D}$\textup{)};
\item[iii)] ${\rm Gr}_{(X, \ucG_0,D)} \otimes_{\calO_F} \calO_{E_0} = {\rm Gr}_{(X_{\calO_{E_0}},  \, \underline{\calG}_{0,{\calO_{E_0}}},\, D_{\calO_{E_0}})}$.
\end{enumerate}
\end{lem}
\begin{proof}
Part i) follows because the formation of $\uG_0$ and $\ucG_0$ as in \cite[Prop.\,3.1.2;\,Thm.\,3.3.3]{Lev} is compatible with change of base $\calO_F[u^\pm] \rightarrow \calO_{E_0}[u^\pm]$ (resp.,\,$\calO_F[u] \rightarrow \calO_{E_0}[u]$); see also Example \ref{torus_eg}. Part ii) follows formally from part i) and the identities $R\pot{D_{\calO_{E_0}}} = \varprojlim_{n} R[u]/Q^n = R\pot{D}$ (resp.,\,$R\rpot{D_{\calO_{E_0}}} = (\varprojlim_{n} R[u]/Q^n) [1/Q]= R\rpot{D}$) for $\calO_{E_0}$-algebras $R$. Part iii) follows from part ii) and Lemma \ref{BL_Lem} ii).
\end{proof}

\begin{prop} \label{bc_E0_prop} In the notation above, there are canonical isomorphisms
\begin{align*}
{\rm Gr}_{(X/\calO_F, \ucG_0, D)} \otimes_{\calO_F} \calO_{E_0} &= \prod_j \Res_{\calO_{E_{j,0}}/\calO_{E_0}} \big( {\rm Gr}_{(X/\calO_{K_0}, \ucG_0, D)} \otimes_{\calO_{K_0}} \calO_{E_{j,0}} \big) \\
&= \prod_j \Res_{\calO_{E_{j,0}}/\calO_{E_0}} \big( {\rm Gr}_{(X_{\calO_{E_{j,0}}}/\calO_{E_{j,0}},\, \, \underline{\calG}_{0,{\calO_{E_{j,0}}}}, \, D \otimes_{\calO_{K_0}} \calO_{E_{j,0}})} \big).
\end{align*}
\end{prop}

\begin{proof}
The first equality is proved using Lemma \ref{tomato_lem}. The second equality follows by applying Lemma \ref{Res_vs_bc_2} iii) to each factor indexed by $j$, replacing the data $(F, K, E_0)$ with $(K_0, K, E_{j,0})$.
\end{proof}
Recall that $\ucG_0$ is defined using the following data: the totally ramified extension $K/K_0$, the $K_0$-group $G_0$, the facet ${\bf f}_0$, and the choice of spreading $\uG_0/\calO_{K_0}[u^{\pm}]$; and the generic fiber of ${\rm Gr}_{(X/\calO_F, \ucG, D)}$ is the affine Grassmannian for $G = \Res_{K/F}G_0$, by Proposition \ref{Fiber_BDGrass}. By contrast $\underline{\calG}_{0, \calO_{E_{j,0}}}$ is defined from the data: the totally ramified extension $KE_{j,0}/E_{j,0}$, the $E_{j,0}$-group $G_{0, E_{j,0}}$, the facet ${\bf f}_0$, and the spreading $\uG_{0,\calO_{E_{j,0}}[u^{\pm}]}$; and the generic fiber of ${\rm Gr}_{(X_{\calO_{E_{j,0}}}/\calO_{E_{j,0}},\, \, \underline{\calG}_{0,{\calO_{E_{j,0}}}}, \, D \otimes_{\calO_{K_0}} \calO_{E_{j,0}})}$ is the affine Grassmannian for $\Res_{K E_{j,0}/E_{j,0}} G_{0,KE_{j,0}}$.  So when restricting attention to the part inside the Weil restriction in the $j$-th factor, we are in the situation ``$F = E_0$'' and ``$K_0 = F$''.  Our next goal is to show how we may effectively reduce our problem to the study of each ${\rm Gr}_{(X_{\calO_{E_{j,0}}}/\calO_{E_{j,0}},\, \, \underline{\calG}_{0,{\calO_{E_{j,0}}}}, \, D \otimes_{\calO_{K_0}} \calO_{E_{j,0}})}$ separately.

\subsection{Statement of theorem} \label{Stmt_Thm_subsec}

Given an irreducible algebraic representation $V$ of $^LG$, we define the parity $d_V \in \mathbb Z/2\mathbb Z$ as in \cite[(7.11)]{HaRi}. Then we define the function $\tau^{\rm ss}_{\mathcal G, V}$ on ${\rm Gr}_{\mathcal G}(k_F)$ by the identity
\begin{equation} \label{tau_V_def}
\tau^{\rm ss}_{\mathcal G, V} = (-1)^{d_V} \, {\rm tr}^{\rm ss}(\Phi \,  | \,  \Psi_{{\rm Gr}_{\mathcal G}}\big({\rm Sat}(V))\big).
\end{equation}
We extend this definition to general representations $V$ of $^LG$ (not necessarily irreducible) by linearity. By Theorem \ref{CentralNB_Thm}, Lemma \ref{hecke_identification}, and Corollary \ref{Center_Identify}, we may view $\tau^{\rm ss}_{\mathcal G, V}$ as an element in  the Hecke algebra $\mathcal Z(G(F), \mathcal G(\mathcal O_F))$. Given any algebraic representation $V$, we also define $z^{\rm ss}_{\mathcal G, V} \in \mathcal Z(G(F), \mathcal G(\mathcal O_F))$ to be the unique function such that, if $\pi$ is an irreducible smooth representation of $G(F)$ on a  $\bar{\mathbb Q}_\ell$-vector space such that $\pi^{\mathcal G(\mathcal O_F)} \neq 0$, then $z^{\rm ss}_{\mathcal G, V}$ acts on $\pi^{\mathcal G(\mathcal O_F)}$ by the scalar ${\rm tr}(s(\pi) \,  | \, V^{1 \rtimes I_F})$, where $s(\pi)$ is the Satake parameter of $\pi$ as defined in \cite{Hai15}.

\begin{thm} \label{Main_Thm_v2}
For $(G, \mathcal G, V)$ as above, we have the equality $\tau^{\rm ss}_{\mathcal G, V} = z^{\rm ss}_{\mathcal G, V}$.
\end{thm}

Recall that taking inertia invariants does not commute in general with forming tensor products of representations. Because of the products and unramified Weil restrictions appearing in Lemma \ref{dual_res} and Proposition \ref{bc_E0_prop}, it is problematic to reduce our Main Theorem to Theorem \ref{Main_Thm_v2}.  Instead we need a variant of Theorem \ref{Main_Thm_v2} {\em without semisimplifying the trace, for a fixed lift $\Phi$ of geometric Frobenius}.  This is formulated as follows. For each fixed lift $\Phi$, we define a function $z^\Phi_{\calG, V} \in \mathcal Z(G(F), \calG(\calO_F))$; similarly we define a function $\tau^\Phi_{\calG, V}$ on ${\rm Gr}_{\calG}(k_F)$; see the Appendix $\S\ref{appendix_non_ss}$. By the same arguments due to Gaitsgory, Pappas-Zhu, and Levin cited above, this function can be viewed as an element of $\mathcal Z(G(F), \calG(\calO_F))$.

\begin{thm} \label{Main_Thm_Phi}
For $(G, \mathcal G, V)$ as above and for every fixed choice of lift $\Phi$ of geometric Frobenius, we have the equality $\tau^{\Phi}_{\mathcal G, V} = z^{\Phi}_{\mathcal G, V}$.
\end{thm}

In fact we will prove Theorem \ref{Main_Thm_Phi}, and we deduce Theorem \ref{Main_Thm_v2} immediately by Lemma \ref{Phi_to_ss}.  However we will not require Theorem \ref{Main_Thm_v2}, but only Theorem \ref{Main_Thm_Phi}, to prove our Main Theorem.

\subsection{Reducing the Main Theorem to Theorem \ref{Main_Thm_Phi}} \label{reduction_to_MT_subsec}

Following the method of \cite[7.3]{HaRi}, we show that the main theorem is a consequence of Theorem \ref{Main_Thm_Phi} as follows. Recall that $V_{\{\mu\}}$ is a representation of $^LG_E = G^\vee \rtimes \Gamma_E$, the $L$-group of $\Res_{K/F}(G_0) \,\otimes_F E$, and that $I(V_{\{\mu\}})$ is a representation of $^LG_{E_0} = G^\vee \rtimes \Gamma_{E_0}$, the $L$-group of 
$\Res_{K/F}(G_0) \,\otimes_F E_0$. Arguing as in \cite[$\S7.3$]{HaRi}, up to the sign $(-1)^{d_\mu}$ the function $\tau^{\rm ss}_{\{\mu\}}$ is identified with the function in $\mathcal Z(G(E_0), \calG(\calO_{E_0}))$
\begin{equation} \label{tau_ss_side}
{\rm tr}^{\rm ss}(\Phi_E \, | \, \Psi_{{\rm Gr}_{\calG, \mathcal O_{E}}}({\rm IC}_{\{\mu\}})) = {\rm tr}^{\rm ss}(\Phi_{E_0} \, | \, \Psi_{{\rm Gr}_{\calG, \mathcal O_{E_0}}}({\rm Sat}(I(V_{\{\mu\}}))).
\end{equation}
Also, $z^{\rm ss}_{\{\mu\}}$ acts on $\pi^{\calG(\calO_{E_0})} \neq 0$ by
\begin{equation} \label{z_ss_side}
{\rm tr}\big(s(\pi) \,  | \, V_{\{\mu\}}^{I_E}\big) = {\rm tr}\big(s(\pi) \, | \, I(V_{\{\mu\}})^{I_{E_0}}\big).
\end{equation}
Therefore, to prove the Main Theorem, it suffices to prove $\tau^{\rm ss}_{\calG_{\calO_{E_0}}, I(V_{\{\mu\}})} = z^{\rm ss}_{\calG_{\calO_{E_0}}, I(V_{\{\mu\}})}$.  All irreducible constituents of $I(V_{\{\mu\}})$ have the same parity, so we may replace $I(V_{\{\mu\}})$ with an arbitrary irreducible representation $V$ of $^LG_{E_0}$.  By Lemma \ref{Phi_to_ss}, it suffices to prove 
\begin{equation} \label{tau=z_E_0}
\tau^{\Phi_{E_0}}_{\calG_{\calO_{E_0}}, V} = z^{\Phi_{E_0}}_{\calG_{\calO_{E_0}}, V}
\end{equation}
for every fixed lift $\Phi_{E_0}$ of geometric Frobenius.

Now $^LG_{E_0} = \big(\prod_j I^{E_0}_{E_{j,0}} \, (\Res_{KE_{j,0}/E_{j,0}} G_{0, KE_{j,0}})^\vee\big) \rtimes \Gamma_{E_0} $ by Lemma \ref{dual_res}. Because $E_{j,0}/E_0$ is unramified, cf.\,Lemma \ref{tensor_fields} i), any irreducible representation $V$ is built up, as explained in the Appendix $\S\ref{appendix_non_ss}$, from irreducible representations of $^L(\Res_{KE_{j,0}/E_{j,0}} G_{0, KE_{j,0}})$. There is a parallel description of the corresponding perverse sheaves on the generic fiber of ${\rm Gr}_{\calG, \calO_{E_0}}$, thanks to Proposition \ref{bc_E0_prop}. Using Lemmas \ref{product_lem} and \ref{unram_Weil_lem}, we easily see that (\ref{tau=z_E_0}) will be proved, if we can prove Theorem \ref{Main_Thm_Phi} for any irreducible representation of $^L(\Res_{KE_{j,0}/E_{j,0}} G_{0,KE_{j,0}})$ and corresponding nearby cycles along ${\rm Gr}_{(X_{\calO_{E_{j,0}}}/\calO_{E_{j,0}},\, \, \underline{\calG}_{0,{\calO_{E_{j,0}}}}, \, D \otimes_{\calO_{K_0}} \calO_{E_{j,0}})}.$

Therefore, we may assume $F = E_0$ henceforth, and we have seen that in order to prove the Main Theorem it is enough to prove Theorem \ref{Main_Thm_Phi}, and in fact we may even assume $K_0 = F$, ie., $K/F$ is totally ramified, and that $V$ is an irreducible representation of $^LG$. By the argument of \cite[Lem.\,7.7]{HaRi}, we may also assume that $V|_{G^\vee \rtimes I_F}$ is irreducible, whenever convenient.

\iffalse
Furthermore, by Lemma \ref{Res_vs_bc_2} (iii), the Main Theorem holds for $V_{\{\mu\}}$ provided Theorem \ref{Main_Thm_v2} holds for the representation $I(V_{\{\mu\}})$ of $^L\Res_{KE_0/E_0} G_{KE_0}$ and nearby cycles along ${\rm Gr}_{(X_{\calO_{E_0}}, \underline{\calG_{\calO_{K'}}}, D_{\calO_{E_0}})}$.  Therefore it suffices to assume $F = E_0$ henceforth. Since all the irreducible factors of the representation $I(V_{\{\mu\}})$ of $\,^L\tilde{G}$ have the same parity, we are reduced to proving Theorem \ref{Main_Thm_v2} for irreducible representations $V$ of $^L\tilde{G}$. By \cite[Lem.\,7.7]{HaRi}, we may assume that $V|_{\tilde{G}^\vee \rtimes I_F}$ is irreducible, whenever convenient.
\fi

\subsection{Proof of Theorem \ref{Main_Thm_Phi}}\label{proof_main_thm_sec}

As above, we will assume $V|_{G^\vee \rtimes I_F}$ is irreducible. Following the proof of \cite[Thm.\,7.9]{HaRi}, there are three main steps:

\begin{enumerate}
\item Step 1: Reduction to minimal $F$-Levi subgroups of $G$.
\item Step 2: Reduction from anisotropic mod center groups to quasi-split groups.
\item Step 3: Proof for quasi-split groups.
\end{enumerate}

The proofs work exactly the same way as in \cite{HaRi}, with only a few additional remarks.  For Step 1, we use Lemma \ref{Torus_Weil_Lem} to ensure that a minimal $F$-Levi subgroup of $G$ is of the form $M = {\Res}_{K/F}(M_0)$, for $M_0$ a minimal $K$-Levi subgroup of $G$; in light of Theorem \ref{CentralNB_Thm} and Theorem \ref{comm_ctnearby_thm} the argument of \cite{HaRi} goes through to reduce us to proving the Theorem \ref{Main_Thm_Phi} for $M$, i.e., for ${\rm Gr}_{(X, \ucM_0, D)}$.  For Step 2, we assume $G$ is $F$-anisotropic mod center and we observe that if $G^*_0$ is a $K$-quasi-split inner form of $G_0$, then $G^* = {\Res}_{K/F}(G^*_0)$ is an $F$-quasi-split inner form of $G$. More to the point, ${\rm Gr}_{(X, \ucG_0, D)}$ and ${\rm Gr}_{(X, \ucG^*_0, D)}$ become isomorphic over $\breve{\calO}_F$ and hence we may think of them as the same geometrically, with differing Galois actions $\Phi$ and $\Phi^*$ of the geometric Frobenius element; applying the argument of \cite{HaRi}, we reduce to proving Theorem \ref{Main_Thm_Phi} for $G^*$, i.e.,\,for ${\rm Gr}_{(X, \ucG^*_0, D)}$.  

For Step 3, we apply Step 1 to $G^*$, and we are reduced to proving the theorem for a torus of the form $T^* = {\Res}_{K/F}(T^*_0)$, i.e.,\,for ${\rm Gr}_{(X, \ucT^*_0, D)}$. The theorem for any torus $T = \Res_{K/F}(T_0)$ is easy. Let us explain following the method of \cite[$\S7.6$]{HaRi}).  
Let $V$ be a representation of $T^\vee \rtimes \Gamma_F$ such that $V|_{T^\vee \rtimes I_F}$ is irreducible.  As in \cite[Def.\,7.11]{HaRi}, let $\omega_V \in \pi_1(T)^{\Phi}_{I_F}$ be the common image of the $T^\vee$-weights in $V|_{T^\vee}$. Then $\omega_V$ can be viewed as the unique $k$-rational point in the support of $\Psi({\rm Sat}(V|_{T^\vee \rtimes I_F}))$, and also as the element indexing the unique coset in the support of $z^{\Phi}_{\calT, V}$. Further, by Proposition \ref{Torus_Prop}, we have an identification of $^LT_r = (T^\vee)^{I_F} \rtimes \Gamma_F$-modules
$$
{\mathbb H}^*(\Psi_{{\rm Gr}_{\calT}}({\rm Sat}(V))) = {\mathbb H}^*({\rm Gr}_{T,\bar{F}}, {\rm Sat}(V))|_{\,^LT_r}.
$$
 By the Grothendieck-Lefschetz fixed point theorem, it suffices to prove
$$
z^{\Phi}_{\calT, V}(\omega_V) = {\rm tr}(\Phi \, | \, V) = {\rm tr}(\Phi \,|\, {\mathbb H}^*({\rm Gr}_{T, \bar{F}}, {\rm Sat}(V))).
$$
The second equality comes from the identifications ${\mathbb H}^*({\rm Gr}_{T, \bar{F}}, {\rm Sat}(V)) =  {\mathbb H}^0({\rm Gr}_{T, \bar{F}}, {\rm Sat}(V)) = V$ as $^LT$-representations under the Satake correspondence.  Therefore we need to prove the first equality. Note that all the weights in $V$ are $I_F$-conjugate, and $z^{\Phi}_{\calT, V}$ acts on a weakly unramified character $\chi: T(F)/T(F)_1 \rightarrow \bar{\mathbb Q}_\ell^\times$ by the scalar
$$
{\rm tr}(s^\Phi(\chi) \, | \, V) = {\rm tr}(\chi \rtimes \Phi \, | \, V) = \chi(\omega_V) \, {\rm tr}(\Phi \, | \, V),
$$
the second equality holding since $s^{\Phi}(\chi) \in (T^\vee)^{I_F} \rtimes \Phi$. Thus $z^{\Phi}_{\calT, V} = {\rm tr}(\Phi \, | \, V) \,\mathds{1}_{\omega_V}$, as desired. This completes the proof of Step 3 and therefore of Theorem \ref{Main_Thm_Phi}. \qed

\subsection{Values of Test Functions}\label{Values_Sec}

As in the Main Theorem of \cite{HaRi}, the function $q^{d_\mu/2}_{E_0} z^{\rm ss}_{\mathcal G, \{\mu\}}$ takes values in $\mathbb Z$ and is independent of $\ell \neq p$.  The proof given in \cite[$\S7.7$]{HaRi} uses only general facts about the Bernstein functions and related combinatorics, and applies equally well to all groups, including those which are Weil-restricted groups such as $G$.

\section{Test functions for modified local models}\label{modified_local_mod_sec}
%If $p\geq 5$, then every reductive group over a $p$-adic field $F$ is isogenous to a reductive group of the form $\Res_{K/F}(G)$ where $K/F$ is totally ramified and $G$ is a tamely ramified $K$-group.
The aim of this final section is to formulate and prove the test function conjecture for all reductive groups and all primes $p\geq 5$ using the modified local models introduced in \cite[\S 2.6]{HPR}.
This is a consequence of our main theorem and some geometric results in \cite{HaRi2}, cf.~Corollary \ref{topological_invariance_cor} below. 

\subsection{Modified local models}
We denote by $G$ a reductive group over a non-archimedean local field $F$ of mixed characteristic $(0,p)$.
We fix an isomorphism
\begin{equation}\label{weak_assumption}
 G_\ad\;\simeq\;\prod_{j\in J} \Res_{K_j/F}(G_j), 
\end{equation}
where each $K_j/F$ is a finite field extension, and each $G_j$ is an absolutely simple, reductive $K_j$-group. 
We assume that each $G_j$ is tamely ramified.
This is only a restriction for $p=2,3$: whenever $p\geq 5$ this assumption is automatically satisfied by the classification, cf.~\cite[\S1.12; \S 4]{Ti77} (see also \cite[\S7.a]{PR08}). 

We further fix a facet $\bbf\subset \scrB(G,F)$ which corresponds to facets $\bbf_j\subset \scrB(G_j,K_j)$, $j\in J$ under the identifications
\[
\scrB(G,F)=\scrB(G_\ad,F)=\prod_{j\in J}\scrB(G_j,K_j)
\]
deduced from Proposition \ref{building_prop} applied to each pair $(G_j,K_j/F)$.
We denote by $\calG=\calG_\bbf$ over $\calO_F$, and by $\calG_j=\calG_{\bbf_j}$ over $\calO_{K_j}$ the associated parahoric group schemes.

We fix a uniformizer $\varpi_j\in K_j$ and an $\calO_{K_{j,0}}[u^\pm]$-extension $\uG_{j,0}$ of $G_j$ where $K_{j,0}/F$ denotes the maximal unramified subextension in $K_j/F$.
Each geometric conjugacy class of cocharacters $\{\mu\}$ in $G$ induces a geometric conjugacy class $\{\mu_\ad\}$ in $G_\ad$ and hence for each $j\in J$ a geometric conjugacy class $\{\mu_j\}$ in $\Res_{K_j/F}(G_j)$.
We note that the reflex field $E$ of $\{\mu\}$ is naturally an overfield of the reflex field $E_j$ of $\{\mu_j\}$.

\begin{dfn}\label{defi_mod_loc_mod}
The modified local model $\bbM_\calG(G,{\{\mu\}})=\bbM\big(K_j/F, \uG_{j,0}, \bbf_j, \{\mu_j\}, \varpi_j\,;\, j\in J\big)$ is the $\calO_E$-product of the $\calO_E$-schemes 
\[
\prod_{j\in J}M_{\{\mu_j\}}^{\on{norm}}\otimes_{\calO_{E_j}}\calO_E, 
\]
where $M_{\{\mu_j\}}=M\big(K_j/F, \uG_{j,0}, \bbf_j, \{\mu_j\}, \varpi_j\big)$ is the local model over $\calO_{E_j}$ as in \S\ref{Dfn_Loc_Mod_Sec} and $M_{\{\mu_j\}}^{\on{norm}}\to M_{\{\mu_j\}}$ denotes its normalization.
\end{dfn}

For convenience we summarize the results on the modified local models obtained in \cite{PZ13, Lev16, HPR, HaRi2}.

\begin{thm} \label{modified_local_models}
i\textup{)} If $G$ splits over a tamely ramified extension of $F$, then $\bbM_\calG(G,{\{\mu\}})$ is isomorphic to the modified local model defined in \cite[\S 2.6]{HPR}.\smallskip\\
ii\textup{)} A morphism of local model triples $(G',\{\mu'\},\calG')\to (G,\{\mu\},\calG)$ with $G'_\ad\simeq G_\ad$ satisfying the tameness assumption in \eqref{weak_assumption} induces an isomorphism of $\calO_{E'}$-schemes
\[
\bbM_{\calG'}(G',{\{\mu'\}})\overset{\simeq}{\longto} \bbM_\calG(G,{\{\mu\}})\otimes_{\calO_E}\calO_{E'}.
\]
iii\textup{)} The modified local model $\bbM_\calG(G,{\{\mu\}})$ is normal with geometrically reduced special fiber, and if $p>2$ it is Cohen-Macaulay as well.
\end{thm}
\begin{proof}
As in \cite[\S2.6]{HPR} we choose for each $j\in J$ a suitable $z$-extension $\tilde G_j\to G_j$ whose derived group $\tilde G_{j,\der}$ is simply connected.
Then the geometric conjugacy class $\{\mu_j\}$ in $\Res_{K_j/F}(G_j)$ can be lifted to $\{\tilde \mu_j\}$ in $\Res_{K_j/F}(\tilde G_j)$ with the same reflex field $\tilde E_j=E_j$, cf.~{\it loc.~cit.}.
Denote by $\tilde\bbf_j$ the facet of $\tilde G_j$ corresponding to $\bbf_j$.
This induces a morphism of $\calO_{E_j}$-schemes on Weil restricted local models
\begin{equation}\label{z_extension_normalisation}
M_{\{\tilde \mu_j\}}:=M\big(K_j/F, \tilde\uG_{j,0}, \tilde\bbf_j, \{\tilde \mu_j\}, \varpi_j\big)\to M\big(K_j/F, \uG_{j,0}, \bbf_j, \{\mu_j\}, \varpi_j\big)=:M_{\{\mu_j\}},
\end{equation}
where $\tilde\uG_{j,0}\to \uG_{j,0}$ is an $\calO_{K_{j,0}}[u^\pm]$-extension of $\tilde G_j\to G_j$.
Now for tamely ramified extensions $K_j/F$ the Weil restricted local models agree by \cite[Prop.~4.2.4]{Lev16} with the local models of \cite{PZ13}.
Further, the morphism \eqref{z_extension_normalisation} is a finite, birational, universal homeomorphism by \cite[Cor.~2.3]{HaRi2}. 
Since $M_{\{\tilde \mu_j\}}$ is normal by \cite[Thm.~9.1]{PZ13} (see also \cite[Thm.~4.2.7]{Lev16}), the map induces an isomorphism $M_{\{\tilde \mu_j\}}\simeq M_{\{\tilde \mu_j\}}^{\on{norm}}$ on normalizations, and the former are the modified local models of \cite[\S2.6]{HPR}. 
Extending scalars to $\calO_E$ and taking the product over $j\in J$ implies part i). 
Part iii) follows from \cite[Cor.~2.5]{HaRi2}, see also the references cited there.
For part ii) we remark that the morphism is a finite, birational, universal homeomorphism by the same reasoning as in i), and that the target is normal: its generic fiber is normal by definition and its special fiber is reduced by iii). 
As the local model is flat and of finite type, the target is normal by Serre's criterion, cf.~\cite[Prop.~9.2]{PZ13}. 
\end{proof}

\begin{rmk} 
As in Remark \ref{Unique_Model_dfn} one can show that $\bbM_\calG(G,{\{\mu\}})$ depends up to equivariant isomorphism only on the data $(G,\{\mu\},\calG)$ which justisfies the notation.
\end{rmk}

We also record the following property which is important for the proof of the test function conjecture. 

\begin{cor}\label{topological_invariance_cor}
The product of the normalization morphisms
\[
\bbM_\calG(G,{\{\mu\}})=\prod_{j\in J}M_{\{\mu_j\}}^{\on{norm}}\otimes_{\calO_{E_j}}\calO_E\to \prod_{j\in J}M_{\{\mu_j\}}\otimes_{\calO_{E_j}}\calO_E
\]
is finite, birational and a universal homeomorphism. 
In particular, this morphism induces an equivalence on the associated \'etale topoi of source and target \cite[03SI]{StaPro}.
\end{cor}
\begin{proof}
This is immediate from the isomorphism $M_{\{\tilde \mu_j\}}\simeq M_{\{\tilde \mu_j\}}^{\on{norm}}$, $j\in J$ in the proof of Theorem \ref{modified_local_models} i), together with \cite[Cor.~2.3]{HaRi2}. 
\end{proof}

\subsection{Test functions} 
Let $(G,\{\mu\},\calG)$ be a triple as above where $G/F$ satisfies the tameness assumption in \eqref{weak_assumption}. 
Denote by $\bbM_{\{\mu\}}=\bbM_\calG(G,{\{\mu\}})$ the modified local model as Definition \ref{defi_mod_loc_mod}.
For a finite extension $E/F$ over which $\{\mu\}$ is defined, we consider the associated semi-simple trace of Frobenius function on the sheaf of nearby cycles
\[
\tau^{\on{ss}}_{\{\mu\}}\co \bbM_{\{\mu\}}(k_{E})\to \algQl, \;\;\; x\mapsto (-1)^{d_\mu}\on{tr}^{\on{ss}}\big(\Phi_{E}\,|\, \Psi_{\bbM_{\{\mu\}}}(\IC_{\{\mu\}})_{\bar{x}}\big).
\]
where $\IC_{\{\mu\}}$ denotes the normalized intersection complex of the generic fiber of $\bbM_{\{\mu\}}$.
As an application of our main theorem we deduce the test function conjecture for modified local models:

\begin{thm}\label{test_modified_local_mod}
Let $(G,\{\mu\},\calG)$ be as above, and denote by $E/F$ an extension which contains the reflex field of $\{\mu\}$. 
Let $E_0/F$ be the maximal unramified subextension. 
Then $\tau^{\on{ss}}_{\{\mu\}}$ naturally defines an element in $\calZ(G(E_0),\calG(\calO_{E_0}))$, and one has
$$
\tau^{\on{ss}}_{\{\mu\}} =  z^{\rm ss}_{\{\mu\}}
$$
where $z^{\rm ss}_{\{\mu\}}\in \calZ(G(E_0),\calG(\calO_{E_0}))$ is the unique function which acts on any $\calG(\calO_{E_0})$-spherical smooth irreducible $\bar{\mathbb Q}_\ell$-representation $\pi$ by the scalar
\[
 \on{tr}\Bigl(s(\pi)\;\big|\; \on{Ind}_{^L{G}_E}^{^L{G}_{E_0}}(V_{\{\mu\}})^{1\rtimes I_{E_0}}\Bigr),
\]
where $s(\pi)\in [(G^\vee)^{I_{E_0}}\rtimes \Phi_{E_0}]_{\on{ss}}/(G^\vee)^{I_{E_0}}$ is the Satake parameter for $\pi$ \cite{Hai15}. 
The function $q^{d_\mu/2}_{E_0}\tau^{\on{ss}}_{\{\mu\}}$ takes values in $\mathbb Z$ and is independent of $\ell \neq p$ and $q^{1/2} \in \bar{\mathbb Q}_\ell$.
\end{thm}

\begin{proof}
First, we will show that $\tau^{\rm ss}_{\{\mu\}}$ naturally defines an element of $\mathcal Z(G(E_0), \calG(\calO_{E_0}))$ for which we need to prove the analogue of \eqref{tau_ss_side}. 
As in \eqref{weak_assumption} we denote by $(H_j, \{\mu_j\}, \calH_j)$, $j\in J$ the local model triple where $H_j:=\Res_{K_j/F}(G_j)$ and $\calH_j:=\Res_{\calO_{K_j}/\calO_F}(\calG_j)$.
Note that $G_\ad=\prod_jH_j$ and $\calG_\ad=\prod_j\calH_j$. 
For the next statement, observe that ${\rm IC}_{\{\mu\}} = {\rm Sat}(V_{\{\mu_{\rm ad}\}})$ under the equivalence of \'{e}tale topoi of Corollary \ref{topological_invariance_cor} for the generic fibers $\bbM_{\{\mu\}, E}$ and $\prod_j M_{\{\mu_j\}} \otimes_{\calO_{E_j}}  E \subset \Gr_{\calG_\ad,\calO_{E_0}} \otimes_{\calO_{E_0}} E = {\rm Gr}_{G_{\rm ad}, E}$.

\begin{lem} \label{6.2_lem}
We have
$$
{\rm tr}^{\rm ss}\big(\Phi_E \, | \, \Psi_{\bbM_{\{\mu\}}}({\rm IC}_{\{\mu\}})\big) = {\rm tr}^{\rm ss}\big(\Phi_{E_0} \, | \, 
\Psi_{\ad} ({\rm Sat}(I(V_{\{\mu_{{\cb \rm ad}}\}})))\big) 
$$
where $\Psi_{\ad}$ denotes the nearby cycles functor for $\Gr_{\calG_\ad,\calO_{E_0}}=\prod_j {\rm Gr}_{{\calH}_j, \calO_{E_0}}$.
\end{lem}

\begin{proof} The argument follows the proof of (\ref{tau_ss_side}) as in \cite[$\S7.3$]{HaRi}. Consider the finite morphism 
$$
f\co \Gr_{\calG_\ad,\calO_{E_0}} \otimes_{\calO_{E_0}} \calO_E ~ \longto ~  \Gr_{\calG_\ad,\calO_{E_0}}.
$$
Abbreviate the nearby cycles for $\Gr_{\calG_\ad,\calO_{E_0}} \otimes_{\calO_{E_0}} \calO_E$ by $\Psi_{\ad, E}$.  We have
\begin{align*}
{\rm tr}^{\rm ss}\big(\Phi_E \, | \, \Psi_{\ad, E}({\rm Sat}(V_{\{\mu_{\rm ad}\}}))\big) &= {\rm tr}^{\rm ss}\big(\Phi_{E_0} \, | \, f_{\bar{s}, *}\Psi_{\ad,E}({\rm Sat}(V_{\{\mu_{\rm ad}\}}))\big) \\
&= {\rm tr}^{\rm ss}\big(\Phi_{E_0} \, | \, \Psi_{\ad}(f_{\eta,*} {\rm Sat}(V_{\{\mu_{\rm ad}\}}))\big) \\
&= {\rm tr}^{\rm ss}\big(\Phi_{E_0} \, | \, \Psi_{\ad}({\rm Sat}(I(V_{\{\mu_{\rm ad}\}})))\big).
\end{align*}
We used the analogue of \cite[Lem.~8.1]{Hai18} for the first equality, proper base change for the second equality and \cite[Prop.\,3.14]{HaRi} for the final equality.
By the topological invariance of the \'etale site in Corollary \ref{topological_invariance_cor} we have 
\[
{\rm tr}^{\rm ss}\big(\Phi_E \, | \, \Psi_{\bbM_{\{\mu\}}}({\rm IC}_{\{\mu\}})\big) = {\rm tr}^{\rm ss}\big(\Phi_E \, | \, \Psi_{\rm ad, E}({\rm Sat}(V_{\{\mu_{\rm ad}\}}))\big),
\]
which proves the lemma.
\end{proof}

By Lemma \ref{6.2_lem} (combined with Theorem \ref{CentralNB_Thm}), the test function
\begin{equation}\label{test_function_last_eq}
\tau^{\rm ss}_{\{\mu\}}=(-1)^{d_\mu}{\rm tr}^{\rm ss}\big(\Phi_{E_0} \, | \, 
\Psi_{\ad} ({\rm Sat}(I(V_{\{\mu_{\rm ad}\}})))\big) 
\end{equation}
naturally defines an element of $\mathcal Z(G_\ad(E_0),\calG_\ad(\calO_{E_0}))$.
By \S\ref{adjoint_sec} below, the natural projection $p\co G\to G_\ad$ induces a canonical morphism of algebras
\[
p_*\co \calZ(G(E_0),\calG(\calO_{E_0}))\to \calZ(G_\ad(E_0),\calG_\ad(\calO_{E_0})).
\]
There is a disjoint union 
\begin{equation}\label{disjoint_union_support_eq}
G(E_0)\;=\; \coprod_{\om \in \pi_1(G)_{I}} G(E_0)_{\om}
\end{equation}
where $G(E_0)_{\om}=\kappa_{G,E_0}^{-1}(\om)$ is the fiber of the Kottwitz morphism $\kappa_{G,E_0}\co G(E_0)\to \pi_1(G)_{I}$ with $I=I_{E_0}=I_F$ (use $E_0/F$ unramified), and likewise for $G$ replaced by $G_\ad$.
The key observation is proved in Lemma \ref{z_ad_red} below: If $\om \mapsto \om_\ad$ under the map $\pi_1(G)_{I}\to \pi_1(G_\ad)_{I}$, then $p_*$ restricts to an isomorphism
\begin{equation}\label{support_iso_center_eq}
\calZ(G(E_0),\calG(\calO_{E_0}))_\om\overset{\simeq}{\longto}  \calZ(G_\ad(E_0),\calG_\ad(\calO_{E_0}))_{\om_\ad}
\end{equation}
on the functions whose support is contained in $G(E_0)_\om$, respectively in $G_\ad(E_0)_{\om_\ad}$.
We apply this to $\om=\om_{\{\mu\}}$, the image of $\{\mu\}$ inside $\pi_1(G)_{I}$. To explain, note that the representation $I(V_{\{\mu\}})$ of $^LG_{E_0}$ need not be irreducible, but all its $B^\vee$-highest $T^\vee$ are $I_{E_0}$-conjugate to $\mu$.  Similar remarks apply to the restriction $I(V_{\{\mu_{\rm ad}\}})$ of $I(V_{\{\mu\}})$ to $^L(G_{\rm ad})_{E_0}$; note as well $\om_{\{\mu\}} \mapsto \om_{\{\mu_{\rm ad}\}}$. Geometrically, this means that the support of (\ref{test_function_last_eq}) is contained in the connected component of ${\rm Gr}_{\calG_{\rm ad}} \otimes_{\calO_F} k_{E_0}$ indexed by $\om_{\{\mu_{\rm ad}\}}$, and hence it belongs to $\mathcal Z(G_{\rm ad}(E_0), \calG_{\rm ad}(\calO_{E_0}))_{\om_{\{\mu_{\rm ad}\}}}$. Finally via (\ref{support_iso_center_eq}) we see that $\tau^{\on{ss}}_{\{\mu\}}$ identifies with an element of $\calZ(G(E_0),\calG(\calO_{E_0}))_{\om_{\{\mu\}}}\subset \calZ(G(E_0),\calG(\calO_{E_0}))$.

Next we prove the equality of $\tau^{\on{ss}}_{\{\mu\}}=z^{\on{ss}}_{\{\mu\}}$ as elements of $\calZ(G(E_0),\calG(\calO_{E_0}))$.
Their values then satisfy the required integrality properties independently of the choice of $\ell \neq p$ and $q^{1/2} \in \bar{\mathbb Q}_\ell$ by \S\ref{Values_Sec}.
It is enough to show the equality $\tau^{\on{ss}}_{\{\mu_{\rm ad}\}}=z^{\on{ss}}_{\{\mu_{\rm ad}\}}$, that is, the equality of the two functions when they are viewed as elements of $\calZ(G_\ad(E_0),\calG_\ad(\calO_{E_0}))$ via \eqref{support_iso_center_eq}.
Here $\tau^{\rm ss}_{\{\mu_{\rm ad}\}}$ is just a relabeling of (\ref{test_function_last_eq}) and we are using Lemma \ref{z_ad_red} below with $V = I(V_{\{\mu\}})$ and $V_{\rm ad} = I(V_{\{\mu_{\rm ad}\}})$ to justify that $z^{\rm ss}_{\{\mu\}} \mapsto z^{\rm ss}_{\{\mu_{\rm ad}\}}$ under \eqref{support_iso_center_eq}.

Fix a lift of geometric Frobenius $\Phi = \Phi_{E_0}$. 
We will show that $\tau^{\Phi}_{\{\mu_{\rm ad}\}} =  z^{\Phi}_{\{\mu_{\rm ad}\}}$,  cf.~\S\ref{appendix_non_ss}. The result will then follow from Lemma \ref{Phi_to_ss} by averaging over the different lifts $\Phi$.

For each $j\in J$ as in \eqref{weak_assumption}, we denote by $\tau^{\Phi}_{\{\mu_j\}}$ (resp.~$z^{\Phi}_{\{\mu_j\}}$) the central functions associated with the local model triple $(H_j, \{\mu_j\}, \calH_j)$ over the extension $E/F$.  Parallel to (\ref{test_function_last_eq}) we have
$$
\tau^{\Phi}_{\{\mu_j\}} = (-1)^{d_{\mu_j}} {\rm tr}\big(\Phi | \, \Psi_{j}({\rm Sat}(I(V_{\{\mu_j\}})))\big),
$$
where $\Psi_{j}$ stands for the nearby cycles functor for $\Gr_{\calH_j} \otimes_{\calO_F} \calO_E$.
The tameness assumption in \eqref{weak_assumption} guarantees that Theorem \ref{Main_Thm_Phi}, and in particular \eqref{tau=z_E_0}, is applicable, and we deduce the equality
\[
\tau^{\Phi}_{\{\mu_j\}}\;=\; z^{\Phi}_{\{\mu_j\}}
\]
as functions in $\calZ(H_j(E_0),\calH_j(\calO_{E_0}))$ for all $j\in J$. Recall that ${\rm IC}_{\{\mu\}} = {\rm Sat}(V_{\{\mu_{\rm ad}\}})$ can be expressed as the external tensor product $\boxtimes_j {\rm Sat}(V_{\{\mu_j\}}) = \boxtimes_j {\rm IC}_{\{\mu_j\}}$ on the generic fiber $\Gr_{\calG_\ad,\calO_{E_0}} \otimes_{\calO_{E_0}} E = \prod_j {\rm Gr}_{{H}_j, E}$. Since the formation of the non-semisimplified functions commutes with direct products of groups by Lemma \ref{product_lem}, we get the equality
\[
\tau^{\Phi}_{\{\mu_{\rm ad}\}}\;=\;\prod_j\tau^{\Phi}_{\{\mu_j\}}\;=\;\prod_{j\in J}z^{\Phi}_{\{\mu_{j}\}} \;=\; z^{\Phi}_{\{\mu_{\rm ad}\}}
\]
inside $\calZ(G_\ad(E_0),\calG_\ad(\calO_{E_0}))=\prod_j\calZ(H_j(E_0),\calH_j(\calO_{E_0}))$. We have also used the equality $\sum_j d_{\mu_j} \equiv d_\mu \,\,{\rm mod}\, 2$.
This completes the proof of the theorem.
\end{proof}

\subsection{Passing to adjoint groups}\label{adjoint_sec}
In this section, we do not need any tameness assumptions -- the arguments hold for general groups.
We work over the unramified extension $E_0/F$. 
Let $Z$ be the center of $G$. 
Let $A, S, T, M$ be as usual for the $E_0$-group $G$, and denote by $A_{\rm ad}, S_{\rm ad}, T_{\rm ad}, M_{\rm ad}$ their images under the canonical map $p: G \rightarrow G_{\rm ad} = G/Z$. 
Recall that $M_\ad$ is not the adjoint group of $M$, but is a minimal $E_0$-Levi subgroup of $G_\ad$. Let $\mathcal G$ (resp.\,$\mathcal G_{\bf a}$) be the parahoric group $\mathcal O_{E_0}$-scheme with generic fiber $G_{E_0}$ attached to a facet ${\bf f}$ (resp.\,an alcove ${\bf a}$ with ${\bf f} \subset \bar{\bf a}$). 
Let $\mathcal G_{\rm ad}$ (resp.\,$\mathcal G_{\rm ad, \bf a}$) be their analogues for $G_{\rm ad}$. 
The morphism $p\co G(E_0) \rightarrow G_{\rm ad}(E_0)$ sends $\mathcal G(\mathcal O_{E_0})$ into $\mathcal G_{\rm ad}(\mathcal O_{E_0})$, and so by the \'{e}toff\'{e} property of $\mathcal G$, $p$ extends to an $\mathcal O_{E_0}$-morphism $p: \mathcal G \rightarrow \mathcal G_{\rm ad}$ (cf.\,\cite[1.7]{BT84}). 
Similarly we have $p: \mathcal G_{\bf a} \rightarrow \mathcal G_{\rm ad, \bf a}$.

Let $\mathcal O$ stand for either $\mathcal O_{E_0}$ or $\mathcal O_{\breve{E}_0}$ with $K$ its fraction field. 
In general, the natural maps $p\co \mathcal G(\mathcal O) \rightarrow \mathcal G_{\rm ad}(\mathcal O)$  and $G(K)/\mathcal G(\mathcal O) \rightarrow G_{\rm ad}(K)/\mathcal G_{\rm ad}(\mathcal O))$ are not surjective. 
Nevertheless, using \cite[5.2.4]{BT84}, if $\mathcal M_{\rm ad}$ (resp.\,$\mathcal T_{\rm ad}$) denotes the unique parahoric group scheme for $M_{\rm ad}$ (resp.\,$T_{\rm ad}$), one can check that $p(\mathcal G(\mathcal O_{E_0})) \cdot \mathcal M_{\rm ad}(\mathcal O_{E_0}) = \mathcal G_{\rm ad}(\mathcal O_{E_0})$ (resp.~$p(\mathcal G(\mathcal O_{\breve{E}_0})) \cdot \mathcal T_{\rm ad}(\mathcal O_{\breve{E}_0}) = \mathcal G_{\rm ad}(\mathcal O_{\breve{E}_0})$). 
This implies that $p$ takes any $\mathcal G(\mathcal O)$-orbit in $G(K)/\mathcal G(\mathcal O)$ onto a $\mathcal G_{\rm ad}(\mathcal O)$-orbit in $G_{\rm ad}(K)/\mathcal G_{\rm ad}(\mathcal O)$, since $\mathcal M_{\rm ad}(\mathcal O_{E_0})$ (resp.\,$\mathcal T_{\rm ad}(\mathcal O_{\breve{E}_0})$) is normalized by the group $N_{G_{\rm ad}}(A_{\rm ad})(E_0)$ (resp.\,$N_{G_{\rm ad}}(S_{\rm ad})(\breve{E}_0)$) giving rise to representatives of those $\mathcal G_{\rm ad}(\mathcal O)$-orbits.

Write $\mathcal O = \mathcal O_{E_0}$ and $\breve{\mathcal O} = \mathcal O_{\breve{E}_0}$. 
We make some abbreviations: Iwahori subgroups $\breve{I} = \mathcal G_{\bf a}(\breve{\mathcal O})$ and $I = \mathcal G_{\bf a}(\mathcal O)$; parahoric subgroups $\breve{J} = \mathcal G(\breve{\mathcal O})$ and $J = \mathcal G(\mathcal O)$. 
Similarly define their analogues for $G_{\rm ad}$: $\breve{I}_{\rm ad}$, $I_{\rm ad}$, $\breve{J}_{\rm ad}$, and $J_{\rm ad}$. 
Let $W_0$ denote the finite relative Weyl group for $G/E_0$.
There are decompositions of the Iwahori-Weyl groups over $E_0$
\begin{align} \label{W_ad_decomp}
W(G) &= W \cong X^*(Z(M^\vee)^{I_{E_0}})^\Phi \rtimes W_0 \cong W_{\rm sc} \rtimes \Omega_{\bf a}\\
W(G_{\rm ad}) &= W_{\rm ad} \cong X^*(Z((M_{\rm ad})^\vee)^{I_{E_0}})^\Phi \rtimes W_0 \cong W_{\rm sc} \rtimes \Omega_{\bf a_{\rm ad}}. \notag
\end{align}
and $p\co W \rightarrow W_{\rm ad}$ is compatible with these decompositions. 
In general $p$ maps $\Omega_{\bf a}$ to $\Omega_{\bf a_{\rm ad}}$, but neither surjectively nor injectively.  
However, for each fixed $\omega \in \Omega_{\bf a}$ with image $\omega_{\rm ad} \in \Omega_{\bf a_{\rm ad}}$,  we obtain an isomorphism
\begin{equation} \label{p_slice_W}
p: W_{\rm sc} \rtimes \omega \, \overset{\sim}{\rightarrow} \, W_{\rm sc} \rtimes \omega_{\rm ad}.
\end{equation}

Recall that $\mathcal H(G(E_0), I)$ is generated by $T^G_w = 1_{I \dot{w} I}$, where $\dot{w} \in N_G(T)(E_0)$ is a lift of $w \in W(G)$ along the Kottwitz homomorphism $\kappa_{G,E_0}$.  
The algebra $\mathcal H(G(E_0), I)$ has an Iwahori-Matsumoto presentation, i.e., it is isomorphic to an affine Hecke algebra over $\mathbb Z[v, v^{-1}]$ with possibly unequal parameters, after a specialization $v \mapsto \sqrt{q_{E_0}} \in \bar{\mathbb Q}_\ell$ (cf.\,\cite[11.3.2]{Hai14} or \cite{Ro15} for a proof in this generality). 
The map $T^G_w \mapsto T^{G_{\rm ad}}_{p(w)}$ respects the braid and quadratic relations, hence gives an algebra homomorphism 
$$\mathcal H(G(E_0), I) \rightarrow \mathcal H(G_{\rm ad}(E_0), I_{\rm ad}).$$  
Let $e_J$ and $e_{J_{\rm ad}}$ be the idempotents corresponding to $J$ and $J_{\rm ad}$; these both correspond to the same set of reflections in $W_{\rm sc}$ (those which fix ${\bf f}$). 
Therefore using the usual relations in the Iwahori-Hecke to understand the products of such idempotents by the standard generators $T^G_w$ , we see that the map
$$
e_J T^G_w e_J \mapsto e_{J_{\rm ad}} T^{G_{\rm ad}}_{w_{\rm ad}} e_{J_{\rm ad}}, 
$$
where $w_{\rm ad} = p(w)$, determines a homomorphism of algebras
$$
p_*\co \mathcal H(G(E_0), J) \rightarrow \mathcal H(G_{\rm ad}(E_0), J_{\rm ad}).
$$
The homomorphism above preserves centers; to see this one uses the Bernstein presentation of $\mathcal H(G(E_0), I)$ and $\mathcal H(G_{\rm ad}(E_0), I_{\rm ad})$ (we refer to \cite{Ro15} for a proof of the Bernstein presentation in this most general setting). 
Put $\Lambda_M = X^*(Z(M^\vee)^{I_{E_0}})^{\Phi}$. 
The Bernstein presentations reflect the decomposition $W = \Lambda_M \rtimes W_0$ of (\ref{W_ad_decomp}) and for each $W_0$-orbit $\bar{\lambda} \subset \Lambda_M$, there is a basis element $z_{\bar{\lambda}} \in \mathcal Z(G(E_0), I)$ (cf.\,\cite[11.10.2]{Hai14}). 
The map sends $z_{\bar{\lambda}}$ to $z_{\bar{\lambda}_\ad}$. 
To see this, we write out each of $z_{\bar{\lambda}}$ and $z_{\bar{\lambda}_\ad}$ in terms of the standard basis elements $T^G_w$ and $T^{G_{\rm ad}}_{p(w)}$. 
(Use that the alcove walk description of Bernstein functions depends only on the combinatorics of $W_{\rm sc}$, cf.\,\cite[$\S14.2$]{HR12}).
Hence the following diagram commutes
\begin{equation} \label{p_vs_Bern}
\xymatrix{
\mathcal Z(G(E_0), J) \ar[r] \ar[d]_{\wr} & \mathcal Z(G_{\rm ad}(E_0), J_{\rm ad}) \ar[d]_{\wr} \\
\bar{\mathbb Q}_\ell[X^*(Z(M^\vee)^{I_{E_0}})^\Phi]^{W_0} \ar[r] & 
\bar{\mathbb Q}_\ell[X^*(Z((M_\ad)^\vee))^{I_{E_0}})^\Phi]^{W_0},
}
\end{equation}
where the vertical arrows are the Bernstein isomorphisms of \cite[11.10.1]{Hai14}.

Now, the top arrow is neither injective nor surjective in general.  
To remedy this, we establish the analogue of (\ref{p_slice_W}) on the level of centers. 
Assume $V \in {\rm Rep}(\,^LG_{E_0})$ has the following property:
$$
\noindent (*) \,\,\,\, \mbox{all the $B^\vee$-highest $T^\vee$-weights in $V|_{G^\vee}$ are $I_{E_0}$-conjugate}
$$
e.g., $V|_{G^\vee \rtimes I_{E_0}}$ is irreducible. Then the restriction $V_{\rm ad} := V|_{\,^L(G_{\rm ad})_{E_0}}$ has the same property.
Let $\omega_V \in \pi_1(G)_{I_{E_0}}^{\Phi_E}$ be the common image of the $\lambda_i \in X_*(T)$ appearing as $B^\vee$-highest $T^\vee$-weights in $V|_{G^\vee}$, and define $\omega_{V_{\rm ad}}$ similarly. Note that $\om_V \mapsto \om_{V_{\rm ad}}$ under $\pi_1(G)_{I_{E_0}} \to \pi_1(G_{\rm ad})_{I_{E_0}}$, i.e.\,$(\om_V)_{\rm ad} = \om_{V_{\rm ad}}$. 

Let $\mathcal H(G(E_0), J)_{\omega_V}$ be the subspace generated by the elements $e_{J} T^G_w e_J$, $w \in W_{\rm sc} \rtimes \omega_V$. 
Define $\mathcal Z(G(E_0), J)_{\omega_V} = \mathcal H(G(E_0), J)_{\omega_V} \cap \mathcal Z(G(E_0), J)$. 
As in \eqref{disjoint_union_support_eq} these are the functions supported on $G(E_0)_{\om_V}\subset G(E_0)$.

\begin{lem}  \label{z_ad_red} 
Assume $V$ satisfies property $(*)$, e.g.~$V=I(V_{\{\mu\}})$. 
Then the map $p_*\co\mathcal Z(G(E_0), J) \rightarrow  \mathcal Z(G_{\rm ad}(E_0), J_\ad)$ induces a vector space isomorphism
\begin{equation} \label{slice_homo}
\mathcal Z(G(E_0), J)_{\omega_V} \, \overset{\sim}{\longrightarrow} \, \mathcal Z(G_{\rm ad}(E_0), J
_\ad)_{\omega_{V_\ad}}
\end{equation}
taking $z^{\Phi}_{\mathcal G, V}$ to $z^{\Phi}_{\mathcal G_{\rm ad}, V_\ad}$.
\end{lem}

\begin{proof}
Since $p(z_{\bar{\lambda}}) = z_{\bar{\lambda}_\ad}$ for $\bar{\lambda}$ (resp.~$\bar{\lambda}_\ad$) ranging over $W_0$-orbits in $\Lambda_M \cap (W_{\rm sc} \rtimes \omega_V)$ (resp.~in $\Lambda_{M_\ad} \cap (W_{\rm sc} \rtimes \omega_{V_\ad})$), the first statement is clear.  
It remains to show that $p_*(z^{\Phi}_{\mathcal G, V}) = z^{\Phi}_{\mathcal G_\ad, V_\ad}$. 

Using the construction of Satake parameters (\cite[$\S9$]{Hai15}, \cite{Hai17}), the map $G \rightarrow G_\ad$ induces a commutative diagram (here for notational convenience we write $\widehat{G}$ in place of $G^\vee$):
$$
\xymatrix{
(Z(\widehat{M})^{I_{E_0}})_{\Phi}/W_0 \ar[r] &  (\widehat{T^*}^{I_{E_0}})_{\Phi^*}/W^*_0 \ar[r] &  (\widehat{T^*}^{I_E})_{\Phi^*}/W^*_{0,E} \ar@{=}[r] & [\widehat{G^*}^{I_{E}} \rtimes \Phi^* ]_{\rm ss}/\widehat{G^*}^{I_{E}} \\
(Z(\widehat{M_\ad})^{I_{E_0}})_{\Phi}/W_0 \ar[u] \ar[r] &  (\widehat{T_\ad^*}^{I_{E_0}})_{\Phi^*}/W^*_0 \ar[u] \ar[r] &  (\widehat{T_\ad^*}^{I_E})_{\Phi^*}/W^*_{0,E} \ar[u] \ar@{=}[r] & [\widehat{G_\ad^*}^{I_{E}} \rtimes \Phi^* ]_{\rm ss}/\widehat{G_\ad^*}^{I_{E}}. \ar[u]
}
$$
Fix $V \in {\rm Rep}(\,^LG_E)$. Starting with the regular function $g^* \rtimes \Phi^* \mapsto {\rm tr}(g^* \rtimes \Phi^* \, | \, V)$ on the variety in the upper right corner, we pull-back along the diagram to get a regular function on the lower left hand corner. Pulling-back in one way yields $z^{\Phi}_{\mathcal G_\ad, V_\ad}$, and pulling-back the other way, by (\ref{p_vs_Bern}), yields $p_*(z^{\Phi}_{\mathcal G, V})$.
\end{proof}

\section{Appendix: non-semisimplified trace}\label{appendix_non_ss}

\subsection{Basic definitions}

Let $V \in {\rm Rep}(\,^LG)$, and let $\Phi \in {\rm Gal}(\bar{F}/F)$ denote a fixed lift of geometric Frobenius. Let $\calG/\mathcal O_F$ denote a parahoric group scheme. Define $z^\Phi_{\calG, V} \in \mathcal Z(G(F), \calG(\mathcal O_F))$ to be the unique function such that, if $\pi$ is an irreducible smooth representation of $G(F)$ on a  $\bar{\mathbb Q}_\ell$-vector space such that $\pi^{\calG(\mathcal O_F)} \neq 0$, then $z^\Phi_{\calG, V}$ acts on $\pi^{\calG(\mathcal O_F)}$ by the scalar $${\rm tr}(s^\Phi(\pi) \,  | \, V),$$ where $s^\Phi(\pi) \in [(G^\vee)^{I_F} \rtimes \Phi]_{\rm ss}/(G^\vee)^{I_F}$ is the Satake parameter of $\pi$ as defined in \cite{Hai15}, relative to the fixed choice of $\Phi$.

Similarly, if $V$ is irreducible with parity $d_V \in \mathbb Z/2\mathbb Z$ as in \cite[(7.11)]{HaRi}, we define the function $\tau^\Phi_{\calG, V}$ on ${\rm Gr}_{\calG}(k_F)$ by the identity
$$
\tau^\Phi_{\calG, V} = (-1)^{d_V} \, {\rm tr}\big(\Phi \, | \, \Psi_{{\rm Gr}_{\calG}}({\rm Sat}(V))\big).
$$
We extend by linearity to define $\tau^\Phi_{\calG, V}$ for all $V$. By the same arguments which proved Theorem \ref{CentralNB_Thm}, Lemma \ref{hecke_identification}, and Corollary \ref{Center_Identify}, we may view $\tau^{\Phi}_{\mathcal G, V}$ as an element in  the Hecke algebra $\mathcal Z(G(F), \mathcal G(\mathcal O_F))$.

\subsection{Averaging over inertia}

Choose any normal finite-index subgoup $I_1 \subseteq I_F$ having the property that $1 \rtimes I_1$ acts trivially on $V$ and $I_1$ acts purely unipotently on all cohomology stalks of $\Psi_{{\rm Gr}_{\calG}}({\rm Sat}(V))$.  

\begin{lem} \label{Phi_to_ss}
Let $\dot{\gamma} \in I_F$ range over a set of lifts of the elements $\gamma \in I_F/I_1$. Then
\begin{align}
\tau^{\rm ss}_{\calG, V} &= \frac{1}{|I_F/I_1|} \, \sum_{\dot{\gamma}} \tau^{\Phi \dot{\gamma}}_{\calG, V} \label{tau} \\
z^{\rm ss}_{\calG, V} &= \frac{1}{|I_F/I_1|} \, \sum_{\dot{\gamma}} z^{\Phi \dot{\gamma}}_{\calG, V}. \label{z}
\end{align}
Consequently, $\tau^{\rm ss}_{\calG, V} = z^{\rm ss}_{\calG, V}$, if $\tau^{\Phi}_{\calG, V} = z^\Phi_{\calG, V}$ for all lifts $\Phi$ of Frobenius.
\end{lem}

\begin{proof}
Let $H$ be a finite group acting on a finite dimensional $\bar{\mathbb Q}_\ell$-vector space $\mathcal V$. Let $\Phi$ denote an arbitrary linear operator on $\mathcal V$.  Then we have the identity
\begin{equation} \label{average}
{\rm tr}(\Phi \, | \, \mathcal V^H) = \frac{1}{|H|} \sum_{h \in H} {\rm tr}(\Phi \circ h \, | \, \mathcal V).
\end{equation}
Now in a cohomology stalk of $\Psi_{{\rm Gr}_{\calG}}({\rm Sat}(V))$, choose a ${\rm Gal}(\bar{F}/F)$-stable filtration on which $I_F$ acts through a finite quotient on the associated graded, denoted $\mathcal V$. Then (\ref{average}) yields (\ref{tau}).  Similarly, using that $1 \rtimes I_F$ already acts through a finite quotient on $V$, we obtain (\ref{z}).
\end{proof}

\subsection{Products and unramified Weil restrictions}

\subsubsection{Products}
Let $G_j$, $j = 1, \dots, n$, connected reductive groups with corresponding parahoric groups $\calG_j$. Write $G = \prod_j G_j$ and $\calG = \prod_j \calG_j$.

Suppose $V_j$ are representations of $^LG_j$. We form the dual group $^L(\prod_j G_j) = \big(\prod_j G^\vee \big) \rtimes \Gamma$, with $\Gamma = {\rm Gal}(\bar{F}/F)$ acting diagonally on the factors. 

\begin{lem} \label{product_lem}
Let $V = \boxtimes_j V_j$, the representation of $^L(\prod_j G_j)$ with $\Gamma$ acting diagonally in the obvious manner.  Then we have equalities of functions in $\mathcal Z(G(F), \calG(\mathcal O_F)) = \prod_j \mathcal Z(G_j(F), \calG(\mathcal O_F))$
\begin{align*}
\tau^\Phi_{\calG, V} &= \prod_j \tau^\Phi_{\calG_j, V_j}  \notag \\
z^\Phi_{\calG, V} &= \prod_j z^\Phi_{\calG_j, V_j}. \notag 
\end{align*}
\end{lem}

\subsubsection{Unramified Weil restrictions}

Let $F_n/F$ be a finite unramified extension of degree $n$, and let $G_0$ be a connected reductive $F_n$-group; let $G = \Res_{F_n/F}G_0$.  Then $^LG$ identifies with the induced group $^LG = (I^{F}_{F_n}G^\vee_0) \rtimes \Gamma_F$, where $\Gamma_F$ acts on the $\bar{\mathbb Q}_\ell$-group $I^F_{F_n}G^\vee_0 \cong \prod_{j=0}^{n-1} G^\vee_0$ in the obvious way. 
Explicitly, as a $\Gamma_n$-group $I^F_{F_n}G^\vee_0 = \prod_{j=0}^{n-1} \, ^{\Phi^{-j}}G^\vee_0$, where $^{\Phi^0}G^\vee_0$ is $G^\vee_0$ endowed with the given action of $\Gamma_n$, and $^{\Phi^{-j}}G^\vee_0$ is the same group but with $\Gamma_n$ acting through the given action precomposed with the automorphism $\gamma_n \mapsto \Phi^{j} \gamma_n \Phi^{-j}$. 
The action of $\Phi$ on $I^F_{F_n}G^\vee_0$ is given by $(g_0, g_1, \dots, g_{n-1}) \mapsto (g_1, g_2, \dots, g_{n-1}, \Phi^n(g_0))$.

An irreducible algebraic representation of $^LG$ is of the form $V= \boxtimes_{j=0}^{n-1} \, ^{\Phi^{-j}}V_0$, where $V_0$ is an irreducible representation of $^LG_0 = G^\vee_0 \rtimes \Gamma_n$, and $\Gamma_n$ acts on $^{\Phi^{-j}}V_0$ by precomposing the given action on $V _0$ with the automorphism $\gamma_n \mapsto \Phi^{j}\gamma_n \Phi^{-j}$.
Then as before $\Gamma_F = \langle \Phi \rangle \Gamma_{F_n}$ operates as follows: $\Gamma_{F_n}$ acts ``diagonally'' on vectors of the form $v_0 \boxtimes v_1 \boxtimes \cdots \boxtimes v_{n-1}$, and $\Phi$ sends such a vector to the vector $v_1 \boxtimes v_2 \boxtimes \cdots \boxtimes v_{n-1} \boxtimes \Phi^n(v_0)$.

\begin{lem} \label{unram_Weil_lem}
We have the identity
$$
{\rm tr}(\Phi \, | \, V) = {\rm tr}(\Phi^n \, | \, V_0),
$$
and this implies the identities in $\mathcal Z(G(F), \mathcal G(\mathcal O_F)) = \mathcal Z(G_0(F_n), \calG_0(\mathcal O_{F_n}))$
\begin{align*} 
\tau^{\Phi}_{\calG, V} &= \tau^{\Phi^n}_{\calG_0, V_0} \\
z^{\Phi}_{\calG, V} &= z^{\Phi^n}_{\calG_0, V_0}.
\end{align*}
\end{lem}

\begin{proof}
The first identity is a special case of a result of Saito-Shintani, cf.\,\cite[Lem.\,6.12]{Feng}. The other assertions follow from this one.
\end{proof}

\end{document}